\documentclass[12pt]{amsart}
\usepackage{amssymb,latexsym,mathtools}
\usepackage[margin=1.25in]{geometry}
\usepackage{mathrsfs}
\usepackage{mathtools}
\usepackage{graphicx}
\usepackage{esint}
\usepackage{braket}
\usepackage{url, hyperref}
\usepackage{tikz}
\usetikzlibrary{matrix,arrows.meta}

\newtheorem{theorem}{Theorem}
\newtheorem{lemma}{Lemma}
\newtheorem{definition}{Definition}

\newtheorem{prop}{Proposition}

\newtheorem{remark}{Remark}
\newtheorem{corollary}{Corollary}

\newcommand{\R}{\mathbb{R}}

\newcommand{\N}{\mathbb{N}}
\newcommand{\F}{\mathscr{F}}
\newcommand{\E}{\mathbb{E}}

\DeclareMathOperator{\supp}{\text{supp}}

\makeatletter
\pgfmathdeclarefunction{erf}{1}{%
  \begingroup
    \pgfmathparse{#1 > 0 ? 1 : -1}%
    \edef\sign{\pgfmathresult}%
    \pgfmathparse{abs(#1)}%
    \edef\x{\pgfmathresult}%
    \pgfmathparse{1 / (1 + 0.3275911 * \x)}%
    \edef\t{\pgfmathresult}%
    \pgfmathparse{%
      1 - ( ( ( ( 1.061405429 * \t - 1.453152027 ) * \t + 1.421413741 ) * \t - 0.284496736 ) * \t + 0.254829592 ) * \t * exp( - \x * \x)}%
     \edef\y{\pgfmathresult}%
    \pgfmathparse{(\sign)*\y}%
    \pgfmath@smuggleone\pgfmathresult%
  \endgroup
}
\makeatother

\begin{document}

\title[Phase Transitions with Microscopic Heterogeneities I]{Diffuse Interface Energies with Microscopic Heterogeneities I: Homogenization}
\author[P.S.\ Morfe and C.\ Wagner]{Peter S.\ Morfe and Christian Wagner}
\address{%
  Max Planck Institute for Mathematics in the Sciences,
  Inselstrasse 22-26,
  04103 Leipzig, Germany; now affiliated at: Pennsylvania State University, 219B McAllister Building, University Park, State College, PA 16802
}
\email{pmorfe@psu.edu}
\address{%
  Max Planck Institute for Mathematics in the Sciences,
  Inselstrasse 22-26,
  04103 Leipzig, Germany; now affiliated at: Institute for Science and Technology Austria (ISTA), Am Campus 1, 3400 Klosterneuburg, Austria
}
\email{christian.wagner@ist.ac.at}

\begin{abstract} We analyze Allen-Cahn functionals with stationary ergodic coefficients in the regime where the length scale $\delta$ of the heterogeneities is much smaller (microscopic) than the interface width $\epsilon$ (mesoscopic). In the main result of this paper, we prove that if the ratio $\delta \epsilon^{-1}$ decays fast enough compared to $\epsilon$, then homogenization effects dominate, and the $\Gamma$-limit of the energy is the same as if the coefficients had been replaced by their homogenized values. As a byproduct of the proof, this implies that homogenization holds in the periodic setting whenever $\delta \epsilon^{-1}$ vanishes with $\epsilon$, no matter how slowly. In a companion paper, we prove this is sharp: if $\delta \epsilon^{-1}$ decays too slowly, then improbable or atypical local configurations of the medium begin to play a role, and the $\Gamma$-limit may be smaller than the one predicted by homogenization theory. We refer to this as the rare events regime, and we prove that it can occur in both random and almost periodic media. \end{abstract}

\date{\today}

\maketitle

\section{Introduction}

In this paper, we analyze the effect of microscopic heterogeneities on the macroscopic behavior of (mesoscopic) diffuse interfaces. Specifically, we consider the $\Gamma$-limit of the following Allen-Cahn-type energy functional with diffuse interface width $\epsilon$ and heterogeneity length scale $\delta$ in the regime $\delta \ll \epsilon$:
	\begin{equation} \label{E: our energy}
		\mathscr{F}_{\epsilon,\delta}(u; U) = \int_{U} \left( \frac{\epsilon}{2} a(\delta^{-1} x) \nabla u \cdot \nabla u + \frac{1}{\epsilon} \theta(\delta^{-1}x) W(u) \right) \, dx.
	\end{equation}
Here $a$ is a uniformly elliptic matrix field and $\theta$, a positive function, both of which are taken to be samples of stationary, ergodic random fields. The nonlinearity $W$ is a double-well potential with minima at $1$ and $-1$, the prototypical choice being $W(u) = (1 - u^{2})^{2}$. 

Energies of the form \eqref{E: our energy} appear in materials science, where they provide a phenomenological, mesoscopic-scale description of phase transitions.  In our context, since $W$ has two minima, the material is composed of two phases.  The so-called phase field $u$ is a function taking values in some compact interval, say, $[-1,1]$, defined in such a way that the set $\{u \approx 1\}$ corresponds to the bulk of the first phase, $\{u \approx -1\}$ corresponds to the bulk of the second phase, and the remainder of space is understood to be the transition region.  For more information on diffuse interface modeling, we refer the interested reader to the books \cite{emmerich2003diffuse,glicksman2010principles,PresuttiBook} and the article \cite{langer1992introduction}.

When the underlying material is spatially homogeneous, the coefficients $a$ and $\theta$ are constant, and then classical results of Modica and Mortola \cite{modica_mortola,modica} imply that the energy $\Gamma$-converges in the limit $\epsilon \to 0$, in such a way that the limiting phase field $u$ takes values in $\{-1,1\}$ and the limiting energy is proportional to the surface area of the interface $\partial \{u = 1\}$.  If $a$ and $\theta$ are stationary ergodic fields and the length scales $\delta$ and $\epsilon$ are commensurate in the sense that $\delta \epsilon^{-1} \sim c$ for some positive $c > 0$, it is known that $\Gamma$-convergence still holds, the only difference being that the limiting surface energy is anisotropic; this was proved by Ansini, Braides, and Chiad\`{o}-Piat \cite{ansini_braides_chiado-piat} and Cristoferi, Fonseca, Hagerty, and Popovici \cite{cristoferi_fonseca_hagerty_popovici} in periodic media and the first author \cite{morfe} and Marziani \cite{marziani} in the stationary ergodic setting.  In this work, we tackle the case when $\epsilon^{-1} \delta \to 0$ as $\epsilon \to 0$.

Since $\delta \ll \epsilon$, it is natural to expect that the homogenization ($\delta \to 0$) limit takes precedence over the sharp interface ($\epsilon \to 0$) limit. This leads to the guess that the limiting sharp interface energy functional ought to have the form
	\begin{equation}\label{eqn:gamma-limit}
		\bar{ \mathscr{E} } (u; U) =  \int_{\partial \{u = 1\} \cap U} \bar{\sigma}( \nu_{ \{ u = 1 \} } )  \, d \mathcal{H}^{d-1},
		\quad { \rm where } \quad 
		\bar{\sigma}( \nu )^2 = \sigma_{W}^2 \bar{\theta} \nu \cdot \bar{a} \nu.
	\end{equation}
Above $\bar{a}$ is the homogenized matrix associated with the gradient term in \eqref{E: our energy}, $\bar{\theta}$ is the expected value of $\theta$, and $\sigma_{W}$ is a constant determined by $W$ (see formula \eqref{E: homogeneous surface tension} below).

Indeed, we prove in the appendix that, in general, the homogenized energy $ \bar{ \mathscr{E} } $ always serves as an upper bound:
	\begin{align} \label{E: gamma upper bound}
		\Gamma\text{-}\limsup_{\epsilon \to 0} \F_{\epsilon,\delta(\epsilon)}(u;U) \leq \bar{ \mathscr{E} }(u;U);
	\end{align}
see Appendix \ref{A: upper bound} for the proof.  In the main result of the paper, we prove that, as long as the ratio $\epsilon^{-1} \delta$ vanishes fast enough, the energy $\F_{\epsilon,\delta}$ $\Gamma$-converges to $ \bar{ \mathscr{E} } $ as expected. This generalizes previous results of Ansini, Braides, and Chiad\`{o}-Piat \cite{ansini_braides_chiado-piat}, Hagerty \cite{hagerty}, and Cristoferi, Fonseca, and Ganedi \cite{cristoferi_fonseca_ganedi_supercritical} in the periodic setting. We refer to this as the \emph{homogenization regime}.

At the same time, if the ratio $\epsilon^{-1} \delta$ vanishes more slowly, it is possible that the $\Gamma$-limit of $\F_{\epsilon,\delta}$ is strictly smaller than $ \bar{ \mathscr{E} }$. We refer to this as the \emph{rare events regime}, and it is the subject of the companion paper \cite{part2}, which treats both random and almost periodic counterexamples.  We use the term ``rare events" because, in this regime, improbable or atypical configurations of the medium become relevant.  This reflects the competition between averaging and energy minimization inherent in the problem: the former drives the energy towards the typical behavior of the coefficients $a$ and $\theta$, while the latter seeks to exploit favorable deviations from the mean.

\subsection{Assumptions} \label{S: assumptions}We assume that there is a complete probability space $(\Omega,\mathcal{F},\mathbb{P})$ supporting a measurable action $\{\tau_{x}\}_{x \in \mathbb{R}^{d}}$ of $\mathbb{R}^{d}$ such that, for any $(x,\omega) \in \mathbb{R}^{d} \times \Omega$,
	\begin{equation*}
		a(x,\omega) = A(\tau_{x}\omega), \quad \theta(x,\omega) =\Theta(\tau_{x}\omega) .
	\end{equation*}
Here $A$ is a symmetric matrix-valued random variable and $\Theta$, a positive random variable.

Concerning $A$, we assume uniform ellipticity, that is, there are constants $\lambda, \Lambda > 0$ such that 
	\begin{equation*}
		\mathbb{P}\{\lambda \text{Id} \leq A \leq \Lambda \text{Id}\} = 1.
	\end{equation*}
Similarly, we assume that $\Theta$ is almost surely bounded above and below by positive constants $\theta_{*},\theta^{*} > 0$:
	\begin{equation*}
		\mathbb{P}\{\theta_{*} \leq \Theta \leq \theta^{*}\} = 1 .
	\end{equation*}

The action $\{\tau_{x}\}_{x \in \mathbb{R}^{d}}$ is assumed to be a bonafide group action, that is,
	\begin{equation*}
		\tau_{0} = \text{Id}, \quad \tau_{x + y} = \tau_{x} \circ \tau_{y} \quad \text{for each} \quad x,y \in \mathbb{R}^{d}.
	\end{equation*} 
We also assume that the map $(\omega,x) \mapsto \tau_{x} \omega$ is measurable with respect to the product $\sigma$-algebra $\mathcal{F} \otimes \mathscr{B}(\mathbb{R}^{d})$, where $\mathscr{B}(\mathbb{R}^{d})$ is the Borel $\sigma$-algebra of $\mathbb{R}^{d}$. Finally, the action is assumed to be both stationary and ergodic. Thus, for any event $E \in \mathcal{F}$ and any $x \in \mathbb{R}^{d}$,
	\begin{equation*}
		\mathbb{P}(\tau_{x}^{-1}(E)) = \mathbb{P}(E),
	\end{equation*}
and if $E \in \mathcal{F}$ satisfies the following invariance assumption
	\begin{equation*}
		\tau_{x}^{-1} (E) = E \quad \text{for each} \quad x \in \mathbb{R}^{d},
	\end{equation*}
then $\mathbb{P}(E) \in \{0,1\}$.

\begin{remark} As is well-known, the above assumptions are general enough to include the case of periodic and almost periodic media, see, for instance, the discussion in \cite{papanicolaou_varadhan}.  (However, care needs to be taken when interpreting probabilistic statements.) \end{remark}

Note that by Fubini's Theorem and stationarity, the assumptions on $A$ and $\Theta$ above imply that, on an event of probability one,
    \begin{align*}
        \lambda \text{Id} \leq a(y) \leq \Lambda \text{Id} \quad \text{and} \quad \theta_{*} \leq \theta(y) \leq \theta^{*} \quad \text{for a.e.} \, \, y \in \mathbb{R}^{d}.
    \end{align*}
We will restrict to this event in what follows.

Finally, here are the assumptions on the potential $W : \mathbb{R} \to [0,\infty)$. First, we assume that $ W $ is continuous and additionally
	\begin{gather*}
		W \, \, \text{is} \, \, C^{2} \, \, \text{in a neighborhood of} \, \, [-1,1], \\
		\{ u \in \mathbb{R} \, \mid \, W(u) = 0\} = \{-1,1\}.
	\end{gather*}
Following \cite{ansini_braides_chiado-piat}, we assume the following growth condition on W at infinity
\begin{align}\label{eqn:w-growth-at-inf}
\liminf _{|u| \to \infty} |u|^{-p} W(u) > 0
\quad { \rm for ~ some ~ } p \geq 2,
\end{align}
This technical assumption ensures that only functions taking values in the interval $ [-1,1] $ contribute to the $ \Gamma $-limit.
	
In addition, we make a weak nondegeneracy assumption at the minima $-1$ and $1$. Namely, we assume there is a $ \kappa \in \mathbb{N}$ such that $W$ is differentiable to order $2\kappa$ at $-1$ and $1$ and, furthermore,
	\begin{equation}\label{eqn:w-non-degeneracy}
		\left\{ \begin{array}{c}
			W^{(2k)}(-1) = W^{(2k)}(1) = 0 \quad \text{for each} \quad k \in \{0,1,\dots,\kappa - 1\}, \\
			\min\{W^{(2\kappa)}(-1),W^{(2\kappa)}(1)\} > 0.
		\end{array} \right.
	\end{equation}
This assumption is typically made in the PDE literature with $\kappa = 1$. Of course, all these assumptions are satisfied by the standard choice $W(u) = (1 - u^{2})^{2}$.

\begin{remark} In several places in this paper, we will significantly weaken the assumptions on $W$ above.  In Theorem \ref{T: planar homogenization theorem upper bound}, proved in Appendix \ref{A: upper bound} below, we only need to assume that $W$ is a nonnegative continuous function such that $\{u \in \mathbb{R} \, \mid \, W(u) = 0\} = \{-1,1\}$. Similarly, in Theorem \ref{T: planar homogenization theorem reduction}, which is proved in Appendix \ref{sec:gamma-convergence-in-prob}, we operate under this same assumption and, in addition, require that the (super)quadratic growth assumption \eqref{eqn:w-growth-at-inf} holds.  \end{remark}

Throughout the paper, $ \delta $ is microscopic, while $ \epsilon $ is mesoscopic, meaning that
\begin{align*}
\frac{\delta( \epsilon ) }{\epsilon} \rightarrow 0 \quad { \rm as } \quad \epsilon \rightarrow 0.
\end{align*}
This is a standing assumption; we will usually not restate it later.

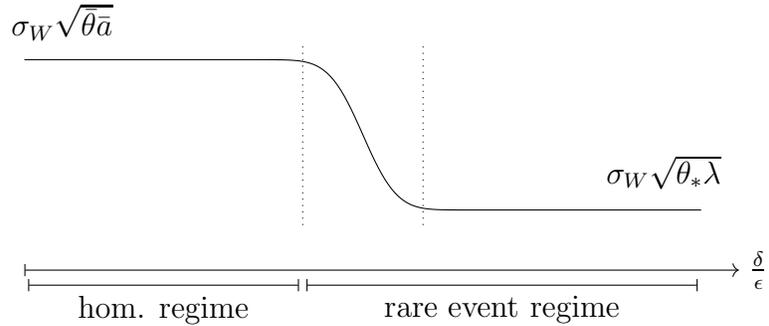
\begin{figure}

\centering

\begin{tikzpicture}

	\draw[|->] (-4.5, -0.8) -- node[at end, right] {$ \frac{\delta}{\epsilon} $} (5, -0.8);

	\draw plot[domain = -4.5 : 4.5, samples=200] (\x, { 1 - erf( 2 * \x) });

	\draw [dotted] (-0.8, -0.2) -- (-0.8, 2.2);
	\draw [dotted] (0.8, -0.2) -- (0.8, 2.2);

	\draw [|-|] ( -4.45, -1 ) -- node[below] {hom. regime} (-0.85, -1);
	\draw [|-|] ( -0.75, -1 ) -- node[below] {rare events regime} ( 4.45, -1);

	\node at (-4, 2.5) { $ \sigma_{W} \sqrt{\bar{\theta} \bar{a}} $ };
	\node at ( 4, 0.5) { $ \sigma_{W} \sqrt{\theta_{*} \lambda}$ };

\end{tikzpicture}
	
\caption{A depiction of the transition between the homogenization and rare events regimes:  Our main theorem shows that if $\delta \epsilon^{-1}$ vanishes rapidly enough, then the relevant effective interfacial energy density is $\sigma_{W} \sqrt{\bar{\theta} \bar{a}}$.  The examples in \cite{part2} show that if $\delta \epsilon^{-1}$ vanishes slowly enough, the minimum values $\lambda$ and $\theta_{*}$ may instead determine the macroscopic energy.  As suggested in the figure, we expect that intermediate values between $\sqrt{\bar{\theta} \bar{a}}$ and $\sqrt{\theta_{*} \lambda}$ are generically also relevant.  In the case of the 1D random checkerboard (with $a \equiv 1$), it is possible to make this diagram precise, see \cite[Theorem 1 \& Figure 1]{part2}. }

\end{figure}

\subsection{Homogenization Regime} Before stating the main result, recall the notion of gradient correctors $\phi$ and flux correctors $\sigma$ from the homogenization theory of elliptic operators in divergence form. For any $\xi \in \mathbb{R}^{d}$, $\phi_{\xi}$ and $\sigma_{\xi}$ are random fields with stationary, mean-zero gradients such that	
	\begin{align} \label{E: PDE for the correctors}
		- \nabla \cdot (a(x) (\xi + \nabla \phi_{\xi})) = 0, \quad - \Delta \sigma_{\xi} = \nabla \times q_{\xi}(x) \quad \text{in} \, \, \mathbb{R}^{d},
	\end{align}
where $q_{\xi} \coloneqq a(x) (\xi + \nabla \phi_{\xi})$ and $ ( \nabla \times q )_{ j k } \coloneqq \partial_j q_k - \partial_k q_j $.  Here $\phi_{\xi}$ is a scalar field and $\sigma_{\xi}$, a matrix field, both of which are unique up to additive constants.  Following the convention of Gloria, Neukamm, and Otto \cite{gloria_neukamm_otto}, we write $\phi$ and $\sigma$ for the vector field and $3$-tensor given by $\phi = (\phi_{e_{1}},\dots,\phi_{e_{d}})$ and $\sigma = (\sigma_{e_{1}},\dots,\sigma_{e_{d}})$, where $\{e_{1},\dots,e_{d}\}$ is the standard orthonormal basis.

In recent work in the field, the quantification of the sublinear growth of $\phi$ and $\sigma$ plays a recurring role, particularly as it pertains to estimating the rate of homogenization.  Toward that end, we define the stationary field ${\rm Sub}(\cdot)$ by
\begin{equation}\label{eqn:defn-sub}
{ \rm Sub }_{x} ( r ) \coloneqq \sup_{ R \ge r } \frac{1}{R} \left(  \fint_{ Q_{R}(x) } | ( \phi, \sigma ) - \fint_{ Q_{R}(x)} ( \phi, \sigma ) |^2 \right)^{ \frac{1}{2} }.
\end{equation}
This field is closely related to the so-called minimal radius of \cite{gloria_neukamm_otto}, which is a key ingredient in the large-scale regularity theory developed therein.  Specifically, if we define the minimal radius $r_{*}(\nu)$ for a given $\nu > 0$ by 
\begin{align} \label{E: minimal radius}
		r_{*}(\nu) = \inf \Big\{ ~ r > 0  ~\Big|~ \frac{ 1 }{ R } \Big( \fint_{ Q_{R} } \Big| (\phi,\sigma) - \fint_{ Q_{R} } (\phi,\sigma) \Big|^{2} \Big)^{\frac{1}{2}} \leq \nu \quad \text{for each} \, \, R \geq r ~ \Big\} ,
\end{align}
then
\begin{align*}
    r_{*}(\nu) > r \quad \text{if and only if} \, \, {\rm Sub}_{0}(r) > \nu .
\end{align*}
We will show below that the rate of sublinear growth of the correctors also plays an important role in the present work.

Next, we track the oscillations of $ \theta $ measured in the $ H^{ -1 } $ norm, more precisely,
\begin{equation}\label{eqn:defn-osc}
{ \rm Osc }_{x} ( r ) \coloneqq \sup_{ R \geq r } \frac{1}{R} \Vert \theta - \bar \theta \Vert_{ H^{-1}( Q_{R}(x) ) },
\end{equation}
where we recall that $ \bar \theta = \E[ \theta(0) ]$. Here we use a suitably scaled $H^{-1}$ norm; see \eqref{E: H minus one norm} below for the definition.

In the statement of our results, it will be convenient to note that there is a metric $d_{\Gamma}$ on a suitable space of lower semicontinuous functionals in $L^{1}$ such that, for any bounded Lipschitz open set $U \subseteq \mathbb{R}^{d}$ and any sequences $(\epsilon_{j})_{j \in \mathbb{N}}$ and $(\delta_{j})_{j \in \mathbb{N}}$ such that $\epsilon_{j} \to 0$, we have that 
	\begin{equation*}
		\mathscr{F}_{\epsilon_{j},\delta_{j}}(\cdot \, ; U) \overset{\Gamma}{\to} \bar{ \mathscr{E} }(\cdot \, ; U) \quad \text{if and only if} \quad d_{\Gamma}(\mathscr{F}_{\epsilon_{j},\delta_{j}}(\cdot \, ; U), \bar{ \mathscr{E} }(\cdot \, ; U)) \to 0,
	\end{equation*}
see Appendix \ref{sec:gamma-convergence-in-prob} for a review of the relevant details.

	\begin{theorem} \label{T: homogenization theorem} Assume that $a$, $\theta$, and $W$ satisfy the assumptions of Section \ref{S: assumptions}. If the microscopic length scale $\epsilon \mapsto \delta(\epsilon)$ is chosen in such a way that, for any $\nu > 0$
		\begin{gather}
			\lim_{\epsilon \to 0} \frac{1}{ \epsilon^{d} } \mathbb{P} \left\{  { \rm Sub }_{0} \left(  \frac{\epsilon}{\delta(\epsilon)} \right) > \nu \right\} = 0, \label{E: sublinear part} \\
			\lim_{\epsilon \to 0}  \frac{1}{\epsilon^{d} } \mathbb{P} \left\{ { \rm Osc }_{0}  \left(  \frac{\epsilon}{\delta(\epsilon)} \right) > \nu \right\} = 0, \label{E: averaging part}
		\end{gather}
	then $\mathscr{F}_{\epsilon,\delta} \overset{\Gamma}{\to} \bar{\mathscr{E}}$ in probability. More precisely, for any $\nu > 0$ and any bounded Lipschitz open set $U \subseteq \mathbb{R}^{d}$, 
		\begin{align*}
			\lim_{\epsilon \to 0} \mathbb{P}\{ d_{\Gamma}(\mathscr{F}_{\epsilon,\delta(\epsilon)}(\cdot \, ; U), \bar{ \mathscr{E} }(\cdot \, ; U)) > \nu\} = 0.
		\end{align*}
	\end{theorem}

In Appendix \ref{sec:gamma-convergence-in-prob}, we reduce the proof of the above theorem to proving convergence of the energy in an auxiliary family of cell problems, c.f.~Theorem \ref{T: planar homogenization theorem reduction}. This is the same reduction previously employed in \cite{ansini_braides_chiado-piat, cristoferi_fonseca_hagerty_popovici, morfe, marziani}, which will be explained in greater detail in Section \ref{section:  homogenization introduction}. After some preparation, most notably Section \ref{S: homogenization 1}, we will prove the convergence of the energy in these cell problems in Section \ref{S: homogenization proof},~c.f.~Theorem \ref{T: planar homogenization theorem lower bound}.

Estimates on the probability $\mathbb{P} \{ {\rm Sub}_{0}(R) > \nu \}$ can be found, for instance, in the previously mentioned work \cite{gloria_neukamm_otto} and the recent book \cite{armstrong_kuusi_book} of Armstrong and Kuusi under the assumption that the law of $a$ satisfies certain mixing conditions.  These can be used to derive concrete sufficient conditions ensuring that assumption \eqref{E: sublinear part} holds, see Section \ref{S: quantitative checkerboard -1} for the case of the random checkerboard.  We expect that similar arguments can be used to quantify $\mathbb{P} \{ { \rm Osc}_{0}(R) > \nu\}$ (see also Section \ref{S: quantitative checkerboard}).

In general, even in the setting of qualitative homogenization (i.e., without imposing mixing conditions), the next result shows that the previous theorem is never vacuous: there is always a choice of scale $ \epsilon \mapsto \delta(\epsilon)$ satisfying \eqref{E: sublinear part} and \eqref{E: averaging part}.

	\begin{prop}\label{prop:existence-of-scales} Given any stationary ergodic medium $(a,\theta)$ as in Section \ref{S: assumptions}, there is a choice of scale $\epsilon \mapsto \delta_{*}(\epsilon)$ such that if $\delta(\epsilon) \leq \delta_{*}(\epsilon)$ for all $\epsilon$ close to zero, then the scale $\epsilon \mapsto \delta(\epsilon)$ satisfies both \eqref{E: sublinear part} and \eqref{E: averaging part}.\end{prop}
	
The proof of Proposition \ref{prop:existence-of-scales} is located in Section \ref{S: existence of scales}.
	
It is possible to upgrade from convergence in probability to almost-sure convergence provided one works with sequences and demands slightly faster decay. This is the subject of the next corollary.

	\begin{corollary}\label{cor:upgrade-to-as-convergence} Let $a$ and $\theta$ be as in Theorem \ref{T: homogenization theorem}. If $(\epsilon_{j})_{j \in \mathbb{N}}$ is any sequence converging to zero and if $(\delta_{j})_{j \in \mathbb{N}}$ is a sequence such that $\frac{\delta_{j}}{\epsilon_{j}} \to 0$ as $j \to \infty$ and, for any $\nu > 0$
		\begin{gather}
			\sum_{j = 1}^{\infty} \frac{1}{ \epsilon_j^{d} } \mathbb{P} \left\{ { \rm Sub }_{0}  \left(  \frac{\epsilon_j}{\delta_j} \right) > \nu \right\} < \infty, \label{E: sublinear summable} \\
			\sum_{j = 1}^{\infty} \frac{1}{\epsilon_j^{d} } \mathbb{P} \left\{ { \rm Osc }_{0} \left(  \frac{\epsilon_j}{\delta_j} \right) > \nu \right\} < \infty, \label{E: averaging summable}
		\end{gather}
	then $\mathscr{F}_{\epsilon_{j},\delta_{j}} \overset{\Gamma}{\to} \bar{ \mathscr{E} }$ almost surely. More precisely, on an event of probability one, we have that, for any bounded Lipschitz open set $U \subseteq \mathbb{R}^{d}$, 
		\begin{equation*}
			\lim_{j \to \infty} d_{\Gamma}(\mathscr{F}_{\epsilon_{j},\delta_{j}}(\cdot \, ; U), \bar{ \mathscr{E} }(\cdot \, ; U)) = 0.
		\end{equation*}
	\end{corollary} 

The corollary will essentially follow from the proof of Theorem \ref{T: homogenization theorem}. We comment on the details in Section \ref{S: almost sure version}.
	
Our results apply in the full generality of stationary, ergodic media without the need for quantitative assumptions. In particular, they apply to periodic and almost periodic media. Of course, in the periodic case, the probabilities above are more-or-less trivial, hence we obtain an unconditional homogenization result for periodic media. We make this precise next in Corollary \ref{cor:periodic-homogenization}.

Actually, we can go slightly beyond periodicity. Recall that, for $p \geq 1$, the Besicovitch space $B^{p}(\mathbb{R}^{d})$ of almost periodic functions is the closed span of the (generalized) trigonometric polynomials with respect to the norm
	\begin{align*}
		\|f\|_{B^{p}(\mathbb{R}^{d})} = \left\{ \begin{array}{r l}
							\left( \lim_{R \to \infty} \fint_{Q_{R}} |f(x)|^{p} \, dx \right)^{\frac{1}{p}}, & \text{if} \, \, p < \infty, \\
							\|f\|_{L^{\infty}(\mathbb{R}^{d})}, & \text{if} \, \, p = \infty.
						\end{array} \right.
	\end{align*}
Of particular interest for us will be the case when $p = \infty$. It is well-known that if $f \in B^{\infty}(\mathbb{R}^{d})$, then cubical averages of $f$ converge to the mean uniformly with respect to the center point, that is, there is a real number $\bar{f}$ such that
	\begin{align} \label{E: uniform averaging}
		\lim_{R \to \infty} \sup \left\{ \left| \fint_{Q_{R}(x)} f(y) \, dy - \bar{f} \right| \, \, \middle| \, \, x \in \mathbb{R}^{d} \right\} = 0.
	\end{align}
This is not necessarily true for functions in $B^{p}(\mathbb{R}^{d})$ with $p < \infty$ (c.f.\ \cite[Remark 5]{part2}), hence this property distinguishes $B^{\infty}(\mathbb{R}^{d})$ as a space of functions that are, in this sense at least, more like periodic functions.  Note that such functions are also sometimes simply called \emph{uniformly almost periodic}.

For $B^{\infty}$ coefficients, it is not hard to see that the uniformity property \eqref{E: uniform averaging} extends to ${ \rm Osc}(\cdot)$, and it is also possible to prove it for a suitable localized version of ${ \rm Sub }(\cdot)$; see Section \ref{S: periodic almost periodic section} for more details and Remark \ref{R: sub remark} for a discussion of ${ \rm Sub }(\cdot)$ specifically.  We thus obtain unconditional homogenization results in both the periodic and $B^{\infty}$ settings.

	\begin{corollary}\label{cor:periodic-homogenization} Let $a$ and $\theta$ be as in Theorem \ref{T: homogenization theorem} and assume that $a$ and $\theta$ are each either $\mathbb{Z}^{d}$-periodic or in the Besicovitch space $B^{\infty}(\mathbb{R}^{d})$. If $\epsilon \mapsto \delta(\epsilon)$ is any function such that 
		\begin{equation*}
			\lim_{\epsilon \to 0} \frac{ \delta(\epsilon) }{ \epsilon } = 0,
		\end{equation*}
	then $\mathscr{F}_{\epsilon,\delta} \overset{\Gamma}{\to} \bar{ \mathscr{E} }$ as $\epsilon \to 0$. \end{corollary}

The proof of this corollary is given in Section \ref{S: proof of periodic corollary}.
	
\subsection{Outlook: Rare Events Regime} \label{S: intro rare events} To reiterate, some restriction on the microscale $\delta$ is necessary in Theorem \ref{T: homogenization theorem}. As the ratio $ \epsilon^{-1}\delta(\epsilon)$ increases, improbable or atypical local configurations become more and more relevant. This is the main content of the companion paper \cite{part2}, which in this section we put into the context of the above results.

In fact, in the case of the random checkerboard in dimension $d = 1$, the assumptions of Theorem \ref{T: homogenization theorem} are sharp. By the random checkerboard, we mean the medium $(a,\theta)$ defined by
	\begin{equation} \label{E: one d random checkerboard}
		(a(x),\theta(x)) = \sum_{z \in \mathbb{Z}^{d}} (A_{z},\Theta_{z}) \boldsymbol{1}_{Q_{1}(z)}(x),
	\end{equation}
where $Q_{1}(z)$ denotes the cube with unit sidelength centered at $z$ and $\{(A_{z},\Theta_{z})\}_{z \in \mathbb{Z}}$ are i.i.d.\ random vectors such that
	\begin{equation} \label{E: bernoulli}
		\mathbb{P} \{ \lambda \text{Id} \leq A_{z} \leq \Lambda \text{Id} \quad \text{and} \quad \theta_{*} \leq \Theta_{z} \leq \theta^{*} \quad \text{for each} \, \, z \in \mathbb{Z}^{d} \} = 1 .
	\end{equation}
Recall that, for such i.i.d.\ fields, large deviations theory says that atypical events at scale $R \gg 1$ occur with probabilities that scale like $\exp(-CR^{d})$. 

Indeed, building on earlier work of Armstrong and Smart \cite{armstrong_smart}, the work of Gloria, Neukamm, and Otto \cite{gloria_neukamm_otto} implies just such an estimate for the sublinear growth of the correctors $(\phi,\sigma)$ of the random checkerboard: namely, for any $\nu > 0$ small enough,
	\begin{align*}
		- \log \mathbb{P} \left\{ { \rm Sub }_{0} \left( R \right) > \nu \right\} \sim_{\nu} R^{d}.
	\end{align*}
(Above $\sim_{\nu}$ means the ratio of the left- and right-hand sides is bounded above and below by positive constants depending only on $\nu$.)  We refer the reader also to \cite[Chapter 6]{armstrong_kuusi_book} for another proof of the lower bound.  Using standard large deviations estimates, one can similarly quantify the oscillations of $\theta$ in this setting: for any $\nu > 0$ small enough,
	\begin{align*}
		- \log \mathbb{P} \left\{ { \rm Osc }_{0} \left( R \right) > \nu \right\} \sim_{\nu} R^{d}.
	\end{align*}
For more details, the reader is referred to Sections \ref{S: quantitative checkerboard -1} and \ref{S: quantitative checkerboard}.

In view of the previous two estimates, Theorem \ref{T: homogenization theorem} implies the following sufficient condition for convergence of $\mathscr{F}_{\epsilon,\delta(\epsilon)}$ to $\overline{\mathscr{E}}$:
	\begin{equation*}
		\lim_{ \epsilon \rightarrow 0 } \frac{ \delta( \epsilon ) | \log \epsilon |^{1/d} }{ \epsilon } = 0.
	\end{equation*}
In the companion paper \cite{part2}, this condition is proved to also be necessary when $d = 1$ and the matrix $a$ is constant: if the limit above equals some nonzero constant $\kappa$, then $\mathscr{F}_{\epsilon,\delta(\epsilon)}$ will converge to a limiting energy $\check{\mathscr{E}}_{\kappa}$ with surface tension $\check{\sigma}_{\kappa}$ depending nontrivially on $\kappa$. This shows that, in this very particular case, the additional assumptions \eqref{E: sublinear part} and \eqref{E: averaging part} in Theorem \ref{T: homogenization theorem} are both necessary and sufficient. This is in sharp contrast to the periodic case, see Corollary \ref{cor:periodic-homogenization}. 

More generally, in \cite{part2}, both random and almost periodic counterexamples are described in arbitrary dimensions $d \geq 1$ in which $\mathscr{F}_{\epsilon,\delta}$ does not converge to $\bar{\mathscr{E}}$.

\subsection{Outline of the Paper} The outline of the proof of Theorem \ref{T: homogenization theorem} is presented in Section \ref{section: homogenization introduction}.  As explained there, the bulk of the work involves proving a $\Gamma\text{-}\liminf$-style inequality, which is our Theorem \ref{T: planar homogenization theorem lower bound}.  The core deterministic arguments involved in Theorem \ref{T: planar homogenization theorem lower bound} are treated in Section \ref{S: homogenization 1}.  These are combined with averaging arguments in Section \ref{S: homogenization 2} to complete the proof of the theorem, the main step of the proof of which appears in Section \ref{S: homogenization proof}.

Section \ref{S: periodic almost periodic section} discusses the periodic and (uniformly) almost periodic settings.  The proof of unconditional homogenization for such media (Corollary \ref{cor:periodic-homogenization}) appears at the end of this section.

There are three appendices.  Appendix \ref{A: upper bound} establishes what we call the homogenization upper bound, which asserts that $\Gamma$-limit of $\mathscr{F}_{\epsilon,\delta(\epsilon)}$ is always no larger than the homogenized surface energy $\bar{\mathscr{E}}$ (see \eqref{E: gamma upper bound}).  Appendix \ref{sec:gamma-convergence-in-prob} develops a notion we call $\Gamma$-convergence in probability.  Using this notion, we prove, following \cite{ansini_braides_chiado-piat} and \cite{morfe}, that if the energies of the planar cell problems converge in probability to the homogenized surface tension $\bar{\sigma}$, then $\mathscr{F}_{\epsilon,\delta(\epsilon)}$ $\Gamma$-converges in probability to the homogenized surface energy $\bar{\mathscr{E}}$.
		
\subsection{Related Literature} Here we cite the most relevant works related to homogenization of Allen-Cahn-type energies with heterogeneous coefficients.  At the end, we also very briefly discuss some related models as well as the relevant background from the homogenization theory for divergence-form elliptic equations.

Some of the works that we cite consider vector-valued phase fields $u$, unlike the present work, which only considers the scalar case.  For brevity, we will not emphasize this point here.  Suffice it to say, though, that we do use techniques (particularly the De Giorgi-Nash-Moser estimate) that are not available in the vectorial setting.  

\subsubsection{$\Gamma$-Convergence of Allen-Cahn-type Energies: Deterministic Heterogeneities} The $\Gamma$-limit of the (constant-coefficient) Allen-Cahn functional was first proved by Modica and Mortola \cite{modica_mortola,modica}.  Expository accounts of this and related results can be found in the lecture notes of Alberti \cite{AlbertiLectureNotes} and the book of Braides \cite{braides_free-discontinuity_book}.

Extensions to periodic media were first treated by Ansini, Braides, and Chiad\`{o}-Piat \cite{ansini_braides_chiado-piat}, who considered energies with gradient terms more general than ours, but without an oscillatory contribution from the potential $W$.  They were able to prove $\Gamma$-convergence and characterize the $\Gamma$-limit in the regimes $\epsilon \sim \delta$, $\epsilon \ll \delta$, and $\delta \ll \epsilon^{\frac{3}{2}}$.  In the latter case, they also showed that the surface tension $\bar{\sigma}$ coincides with what one obtains by first sending $\delta \to 0$ and then $\epsilon \to 0$.\footnote{Notice that in the companion paper \cite{part2} we show that, in the random setting, the condition $\delta \ll \epsilon^{\frac{3}{2}}$ given in \cite{ansini_braides_chiado-piat} is no longer a sufficient condition for the homogenized surface tension $\bar{\sigma}$ to appear in the limit.}  In all three cases, they first studied the convergence of the normalized energy in a a family of planar cell problems (see Section \ref{section: homogenization introduction} below) and then combined more abstract compactness arguments with integral representation results to argue that this characterizes the $\Gamma$-limit. 

Shortly thereafter, Dirr, Lucia, and Novaga \cite{dirr_lucia_novaga_1,dirr_lucia_novaga_2} considered the problem when $a \equiv \text{Id}$ and the standard potential $\epsilon^{-1} W(u)$ is replaced by a singular perturbation of the form $\epsilon^{-1} W(u) + \epsilon^{-\alpha} g(\epsilon^{-\alpha} x) u$ for some $\alpha \in (0,1]$ and a suitable periodic function $g$. In case $\alpha \in (0,1)$, this is a two-scale problem with, in our notation, $\delta \gg \epsilon$, and, indeed, the results in \cite{dirr_lucia_novaga_1} establish that the $\Gamma$-limit coincides with what one might expect by first sending $\epsilon \to 0$ and then $\delta \to 0$.  When $\alpha = 0$, the scale of the oscillations of $g$ coincides with the diffuse interface width.  In \cite{dirr_lucia_novaga_2}, the authors proved $\Gamma$-convergence in the case $\alpha = 0$ assuming that $g = \Delta v$ for a suitably small $v$ (in the $W^{1,\infty}$ norm).  In contrast with the present contribution, in both these works, the constant functions $1$ and $-1$ are no longer global minima of the functional.  Instead, the assumptions of \cite{dirr_lucia_novaga_1,dirr_lucia_novaga_2} imply that there are two ``pure phases," that is, nonconstant, periodic global minimizers close to $1$ and $-1$, the oscillations of which need to be considered in the proofs.  This is a complication that is not encountered in the present work or \cite{ansini_braides_chiado-piat}.

In a similar vein, Fonseca and coauthors recently revisited the $\Gamma$-convergence problem in the setting in which only the potential term oscillates.  In contrast to \cite{dirr_lucia_novaga_1,dirr_lucia_novaga_2}, they considered general potentials of the form $\epsilon^{-1} W(\delta^{-1} x,u)$.  The work of Cristoferi, Fonseca, Hagerty, and Popovici \cite{cristoferi_fonseca_hagerty_popovici, cristoferi_fonseca_hagerty_popovici_erratum} treated the case $\epsilon \sim \delta$ in the setting where $\{W(y,\cdot) = 0\} = \{-1,1\}$ so $1$ and $-1$ are once again global minimizers.  They followed a strategy of proof that is similar to the one in \cite{ansini_braides_chiado-piat}.

Since then, Cristoferi, Fonseca, and Ganedi \cite{cristoferi_fonseca_ganedi_supercritical} considered the case when $\delta \ll \epsilon$ in periodic media, again under the assumption $\{W(y,\cdot) = 0\} = \{-1,1\}$.  They proved, as in our Corollary \ref{cor:periodic-homogenization}, that the $\Gamma$-limit is determined by the mean of the potential $W(y,u)$ in the $y$ variable.  (This improved earlier work of Hagerty \cite{hagerty}, who had imposed the additional restriction $\delta \ll \epsilon^{\frac{3}{2}}$ as in \cite{ansini_braides_chiado-piat}.)  The strategy in \cite{cristoferi_fonseca_ganedi_supercritical} is philosophically similar to ours, but quite different mathematically.  At a high level, both our work and theirs is based on the idea that one should locally replace a phase field $u$ by a another function $\tilde{u}$ that incorporates the behavior of the medium at or above scale $\delta$.  In \cite{cristoferi_fonseca_ganedi_supercritical}, the argument proceeds by dividing up the domain into cubes of size $\delta$ and then using a slicing argument to compare $u$ to a candidate for a certain cell problem that approximates $\int_{\mathbb{T}^{d}} W(y,u) \, dy$.  By contrast, we decompose the domain into cubes of size $\epsilon$.  In each cube, we compare a minimizer $u$ of $\mathscr{F}_{\epsilon,\delta}$ to a minimizer $\bar{u}$ of the homogenized energy with the same boundary values in that cube, obtaining a bound on the difference of the energies using homogenization error estimates and elliptic regularity results.

We also mention another work of Cristoferi, Fonseca, and Ganedi \cite{cristoferi_fonseca_ganedi_subcritical}, wherein they considered the case when $\delta \gg \epsilon$ and the potential $W(y,u)$ has space-dependent wells (i.e., the zero set $\{W(y,\cdot) = 0\}$ varies with $y$).  They study the volume scaling of the energy, hence a functional of the form
	\begin{align*}
		\mathscr{G}_{\epsilon,\delta}(u; U) = \int_{U} \left( \frac{\epsilon^{2}}{2} |\nabla u|^{2} + W(\delta^{-1} x, u) \right) \, dx.
	\end{align*}
(We refer to this as the volume scaling since the macroscopic energy in the box $Q_{r}$ scales like $r^{d}$ rather than $r^{d-1}$.)
They showed that, under certain assumptions, if $\{W(y,\cdot) = 0\} = \{a(y), b(y)\}$ for two periodic functions $a$ and $b$, then $\mathscr{G}_{\epsilon,\delta(\epsilon)}$ $\Gamma$-converges to a functional $G^{0}$ of the form $G^{0}(u; U) = \int_{U} W^{\text{hom}}(u) \, dx$ (Theorem 4.3), the minimizers of which are, in a sense that can be made precise, mixtures of $a$ and $b$ (Corollary 4.4).  Taking this a step further, they also fully characterized the first-order (in $\delta^{-1} \epsilon$) correction to $\mathscr{G}_{\epsilon,\delta}$ for a natural class of minimizers (Theorem 5.8).

Finally, since this work first appeared as a preprint, Cristoferi and Pignatelli \cite{cristoferi_pignatelli} extended the result of \cite{cristoferi_fonseca_ganedi_supercritical} to situations involving multiple (microscopic) scales.

\subsubsection{$\Gamma$-Convergence of Allen-Cahn-type Energies: Random Heterogeneities} \hfill The \\ first random results of which we are aware were contributed by Dirr and Orlandi \cite{dirr_orlandi}.  They considered a perturbation of the standard, constant-coefficient Allen-Cahn functional, in which the potential $\epsilon^{-1} W(u)$ is replaced by $\epsilon^{-1} W(u) + \epsilon^{-1} \delta(\epsilon) g(\epsilon^{-1}x) u$, where $g$ is a random checkerboard and $\delta(\epsilon)^{-1} = \zeta |\log(\epsilon)|$ for some $\zeta > 0$.  To simplify the problem, they studied the functional restricted to the torus $\mathbb{T}^{d}$ in dimensions $d \geq 3$.  They proved the existence of two random functions $u^{+}_{\epsilon}$ and $u^{-}_{\epsilon}$, which minimize the energy in $H^{1}(\mathbb{T}^{d})$ and converge to the unperturbed minima $1$ and $-1$, respectively, as $\epsilon \to 0$ (Theorem 2.1).  They also showed that, after shifting by the energy of $u^{+}_{\epsilon}$, the energy $\Gamma$-converges to the perimeter, as if $g$ were identically zero (Theorem 2.3).

More recently, the first author \cite{morfe} considered functionals similar to the ones studied here, albeit in the regime $\epsilon \sim \delta$ and with stationary ergodic oscillations only in the gradient term.  (In fact, only $\delta = \epsilon$ is considered there, but the arguments readily generalize to the case when $\epsilon^{-1} \delta(\epsilon) \to c$ for some $c > 0$.)  That work, which followed the strategy of \cite{ansini_braides_chiado-piat}, highlighted the fact that it is useful to consider the energy of the planar cell problem as an almost-monotone function in the normal variable and a subadditive process in the transversal variables.  That point-of-view is also useful in the present context, as it forms the basis for our proof of the (unconditional) homogenization upper bound (see Appendix \ref{A: upper bound} below).

The results of \cite{ansini_braides_chiado-piat, cristoferi_fonseca_hagerty_popovici, morfe}  were generalized by Marziani \cite{marziani}, who considered the general stationary ergodic setting with oscillations both in the gradient and well terms, still only in the regime $\epsilon \sim \delta$ and with the assumption that $1$ and $-1$ are global minima of the functional.  The strategy of proof is the same as in \cite{ansini_braides_chiado-piat}.  In contrast with \cite{morfe}, in \cite{marziani}, the convergence of the planar cell problem is proved by adapting the arguments of Cagnetti, Dal Maso, Scardia, and Zeppieri \cite{cagnetti_dal-maso_scardia_zeppieri}, who had earlier considered surface energy functionals with stationary ergodic coefficients.

We also mention the recent work of Dos Santos, Rodiac, and Sandier \cite{dos-santos_rodiac_sandier}, who considered specific classes of Ginzburg-Landau- and Allen-Cahn-type functionals with stationary ergodic (or periodic) coefficients.  In the Allen-Cahn case, they restricted to $a \equiv \text{Id}$ and considered potentials of the form $W(\delta^{-1}x,u) = (u^{2} - a(\delta^{-1} x))^{2}$, where $a$ is stationary ergodic (or periodic).  Using the special form of $W$, they were able to characterize the limiting behavior of volume-constrained minimizers in a bounded domain while assuming only that $\epsilon^{-1} \delta(\epsilon) \to 0$ as $\epsilon \to 0$ (Theorem 7.1).

\subsubsection{Related Models} There is a growing literature on $\Gamma$-convergence results for heterogeneous energy functionals incorporating surface effects.  Without delving deeply into this body of work, we mention the recent contributions of Bach, Esposito, Marziani, and Zeppieri \cite{bach_esposito_marziani_zeppieri_one_d, bach_esposito_marziani_zeppieri_higher_d}, who studied variants of the Ambrosio-Tortorelli functional with heterogeneous coefficients.  As in the literature on Allen-Cahn-type functionals reviewed above, in their work, the heterogeneity length scale $\delta(\epsilon)$ is coupled to the singular perturbation scale $\epsilon$, and they were able to characterize the $\Gamma$-limit in the periodic setting in the three regimes $\delta \gg \epsilon$, $\delta \sim \epsilon$, and $\delta \ll \epsilon$, the last case being the most difficult (and involving some extra assumptions).  When $\delta \sim \epsilon$, stochastic homogenization was proved by Bach, Marziani, and Zeppieri in \cite{bach_marziani_zeppieri_singular}, but, as far as we know, for this class of functionals, the $\delta \ll \epsilon$ regime in the random case remains a challenging open problem.


\subsubsection{Stochastic Homogenization of Divergence-Form Elliptic Equations} This paper takes advantage of recent developments in the theory of stochastic homogenization of divergence-form elliptic operators.  We refer the reader to the article of Josien and Otto \cite{JosienOtto} and the books of Armstrong, Kuusi, and Mourrat \cite{armstrong_kuusi_mourrat_book} and Armstrong and Kuusi \cite{armstrong_kuusi_book} both for historical background and the state-of-the-art.

Perhaps the most notable place where we benefit from recent insights is in the utilization of the quantities ${\rm Sub}(\cdot)$ and ${\rm Osc}(\cdot)$ defined in \eqref{eqn:defn-sub} and \eqref{eqn:defn-osc} above.  The quantity ${\rm Sub}(\cdot)$ is closely related to what has been dubbed the minimal radius by Gloria, Neukamm, and Otto \cite{gloria_neukamm_otto}, and, as is apparent in the proofs below, it turns out to be a very convenient tool for measuring how close the gradient energy term is to its homogenized form.  Our use of the augmented corrector $(\phi,\sigma)$ and ${\rm Sub}(\cdot)$ is inspired by the work of Otto and collaborators, particularly \cite{gloria_neukamm_otto} and \cite{JosienOtto}.

\subsection{Notation and Terminology} Throughout the paper, we denote by $\{e_{1},\dots,e_{d}\}$ the standard orthonormal basis of $\mathbb{R}^{d}$.  We write $Q_{r}$ for the cube
	\begin{equation*}
		Q_{r} = \left( - \frac{r}{2}, \frac{r}{2} \right)^{d} = \left\{y \in \mathbb{R}^{d} \, \mid \, |y \cdot e_{i}| < \frac{r}{2} \, \, \text{for each} \, \, i \in \{1,\dots,d\} \right\}.
	\end{equation*}
The translated copy of $Q_{r}$ centered at $x \in \mathbb{R}^{d}$ is denoted by $Q_{r}(x)$. 

Throughout the paper, given two families of real numbers $(A_{\eta})_{\eta}$ and $(B_{\eta})_{\eta}$ depending on some parameter $\eta$, we write
	\begin{equation*}
		A_{\eta} \lesssim B_{\eta}
	\end{equation*}
if there is a constant $C > 0$ such that $A_{\eta} \leq C B_{\eta}$ for all values of $\eta$.  If the constant $C$ is determined by some other parameters, say, $\mu$ and $\varrho$, then we indicate this by writing $A_{\eta} \lesssim_{\mu,\varrho} B_{\eta}$.  Our usage of the symbol $\gtrsim$ is entirely analogous.  The symbol $\sim$ means that both $\lesssim$ and $\gtrsim$ hold.

We use a particular normalization of the $ H^{-1} $ norm, which for $ f \in L^2_{ loc }( \R^d ) $ is given by
	\begin{equation} \label{E: H minus one norm}
		\|f\|_{H^{-1}(U)} = \sup \left\{ \fint_{U} f v \, \mid \, v \in H^{1}_{0}(U), \, \, \fint_{U} | \nabla v |^{2} \leq 1 \right\}.
	\end{equation}

Occasionally, we abbreviate
\begin{align}\label{eqn:defn-min}
m( \F, U, g) \coloneqq \min \{ \F(v,U) ~|~ v \in g + H_{0}^{1}(U) \}.
\end{align}

Finally, we say that a function $\omega : [0,\infty) \to [0,\infty)$ is a \emph{modulus of continuity} if $\omega$ is nondecreasing, continuous at zero, and $\omega(0) = 0$.

\addtocontents{toc}{\protect\setcounter{tocdepth}{0}}
\section*{Acknowledgements}  

We thank Felix Otto for organizing a stimulating research environment at the MPI in Leipzig, where we first met and completed this project.  The first author thanks Annika Bach for suggesting he revisit this problem and acknowledges the support of NSF Grant DMS-2202715. 

\addtocontents{toc}{\protect\setcounter{tocdepth}{2}}

\tableofcontents

\section{Proof of Theorem \ref{T: homogenization theorem}}\label{section:  homogenization introduction}

In this section, we outline the proof of our main theorem on the homogenization regime, Theorem \ref{T: homogenization theorem}.  We use the same reduction as in \cite{ansini_braides_chiado-piat,morfe,cristoferi_fonseca_hagerty_popovici,marziani}, exploiting the fact that $ \Gamma $-convergence is equivalent to the convergence of the (suitably scaled) energy in a certain family of planar cell problems.  Precisely, we fix a function $q : \mathbb{R} \to \R $ satisfying
	\begin{gather} 
		\sup_{ s \in \R } | q'(s) | +  \int_{ - \infty }^{ \infty } q '( s )^2 + W(q(s)) \, ds < \infty, \label{E: finite width} \\
	 	-1 \leq q(s) \leq 1,  \quad q(s) \rightarrow 1 \quad \text{as} ~ s \rightarrow \infty, \quad { \rm and } \quad  q(s) \rightarrow -1 \quad \text{as} ~ s \rightarrow - \infty.  \label{E: finite width 2}
	\end{gather}
	Notice that the function $q(\gamma^{-1} \cdot)$ converges to $ \boldsymbol{1}_{[0,\infty)} - \boldsymbol{1}_{(-\infty,0)}$ in $L^{1}_{\text{loc}}(\mathbb{R})$ as $\gamma \to 0$. In order to prove $\Gamma$-convergence, up to rotation and translation, it suffices to prove that
	\begin{align*}
		\lim_{\epsilon \to 0} \min \left\{ \frac{1}{ \varrho^{d-1} } \mathscr{F}_{\epsilon,\delta(\epsilon)}(u;Q_{\varrho}) \, \mid \, u(y) = q(\epsilon^{-1}y) \, \, \text{for each} \, \, y \in \partial Q_{\varrho} \right\} = \bar{\sigma}(e_{1}).
	\end{align*}
Here $\bar{\sigma} : S^{d-1} \to (0,\infty)$ is the homogenized surface tension defined by the relation
	\begin{equation*}
		\bar{\sigma}(e)^{2} = \sigma_{W}^{2} \bar{\theta} e \cdot \bar{a} e.
	\end{equation*}
Above $\bar{a}$ is the homogenized matrix associated with $a$, the definition of which is recalled in Section \ref{S: homogenization preliminaries}; $\bar{\theta} = \mathbb{E}[\theta(0)]$ is the mean of $\theta$; and the constant $\sigma_{W}$ is the surface tension of the constant-coefficient Allen-Cahn functional in case $a \equiv \text{Id}$ and $\theta \equiv 1$.  It is well-known that $\sigma_{W}$ is characterized by the variational formula
	\begin{equation} \label{E: homogeneous surface tension}
		\sigma_{W} = \min \left\{ \int_{-\infty}^{\infty} \left( \frac{1}{2} u'(s)^{2} + W(u(s)) \right) \, ds \, \mid \, \lim_{s \to \pm \infty} u(s) = \pm 1 \right\}.
	\end{equation}
For a proof of this variational principle, we refer to the notes of Alberti \cite{AlbertiLectureNotes}.

In order to prove convergence, we begin by observing that the homogenized norm $\bar{\sigma}$ of \eqref{eqn:gamma-limit} is always an upper bound. Note that in contrast to Theorem \ref{T: planar homogenization theorem lower bound} below, the next result does not assume any smoothness of $W$.

	\begin{theorem} \label{T: planar homogenization theorem upper bound}  Assume that the medium $(a,\theta)$ satisfies the assumptions of Section \ref{S: assumptions} and the potential $W : \mathbb{R} \to [0,\infty)$ is a nonnegative continuous function such that $W(u) = 0$ if and only if $u \in \{-1,1\}$.  If $\epsilon \mapsto \delta(\epsilon)$ is any scaling such that $\epsilon^{-1} \delta(\epsilon) \to 0$ as $\epsilon \to 0$, then, for any $x_{0} \in \mathbb{R}^{d}$ and any $\varrho > 0$, with probability one,
		\begin{align*}
			\limsup_{\epsilon \to 0} \min \left\{ \frac{1}{  \varrho^{d-1} } \mathscr{F}_{\epsilon,\delta(\epsilon)}(u; Q_{\varrho}(x_{0})) \, \mid \, u - q(\epsilon^{-1} (\cdot - x_{0}) \cdot e_{1}) \in H^{1}_{0}(Q_{\varrho}(x_{0}) \right\} \leq \bar{\sigma}(e_{1}).
		\end{align*}\end{theorem}
		
Since this is a basic observation, applicable in both the homogenization and rare events regimes, the proof is relegated to Appendix \ref{A: upper bound}.
		
Next, we prove that $\bar{\sigma}$ is also a lower bound provided the two assumptions \eqref{E: sublinear part} and \eqref{E: averaging part} both hold.

	\begin{theorem} \label{T: planar homogenization theorem lower bound}
		Assume that the medium $(a,\theta)$ and the potential $W$ satisfy the assumptions of Section \ref{S: assumptions}. If the scale $\epsilon \mapsto \delta(\epsilon)$ is chosen so that \eqref{E: sublinear part} and \eqref{E: averaging part} both hold, then, for any $x_{0} \in \mathbb{R}^{d}$ and any $ \varrho > 0 $,
		\begin{align*}
			\bar{\sigma}(e_{1}) \leq \liminf_{\epsilon \to 0} \min \left\{ \frac{1}{  \varrho^{d-1} } \mathscr{F}_{\epsilon,\delta(\epsilon)}(u; Q_{\varrho}(x_{0})) \, \mid \, u - q(\epsilon^{-1} (\cdot - x_{0}) \cdot e_{1}) \in H^{1}_{0}(Q_{\varrho}(x_{0})) \right\}
		\end{align*}
	in probability. \end{theorem}
	
The proof of Theorem \ref{T: planar homogenization theorem lower bound} is the subject of the next two sections, Sections \ref{S: homogenization 1} and \ref{S: homogenization 2}.  As we prove in the companion paper \cite{part2}, there are counterexamples showing that the theorem fails to hold without \eqref{E: sublinear part} or \eqref{E: averaging part}.
	
For concreteness, we have fixed the direction of the transition in the results above to equal the first coordinate vector $e_{1}$.  In fact, these results remain true if we replace $e_{1}$ by any unit vector, as, indeed, our assumptions are rotationally invariant.

	\begin{prop} \label{P: rotated fields}  If the medium $(a,\theta)$ satisfies the assumptions of Section \ref{S: assumptions}, then, for any orthogonal transformation $\mathcal{O} : \mathbb{R}^{d} \to \mathbb{R}^{d}$, the rotated coefficients $(a^{\mathcal{O}},\theta^{\mathcal{O}})$ given by 
		\begin{equation*}
			a^{\mathcal{O}}(x) = a(\mathcal{O}(x)), \quad \theta^{\mathcal{O}}(x) = \theta(\mathcal{O}(x)),
		\end{equation*}
	satisfy the same assumptions with the group action $(\tau_{x})_{x \in \mathbb{R}^{d}}$ replaced by the action $(\tau^{\mathcal{O}}_{x})_{x \in \mathbb{R}^{d}}$ given by $\tau^{\mathcal{O}}_{x} = \tau_{\mathcal{O}(x)}$.  Furthermore, we have
		\begin{itemize}
			\item[(i)] $(a,\theta)$ satisfies the assumptions \eqref{E: sublinear part} and \eqref{E: averaging part} if and only if $(a^{\mathcal{O}},\theta^{\mathcal{O}})$ does.
			\item[(ii)] $\bar{a}^{\mathcal{O}} = \mathcal{O}^{-1} \bar{a} \mathcal{O}$ and $\bar{\theta}^{\mathcal{O}} = \bar{\theta}$.
		\end{itemize}
	\end{prop}
	
In view of the proposition, we have proved that, in any given direction $e$, the energy of the planar cell problem converges to $\bar{\sigma}(e)$ in probability as the size of the cell goes to infinity.  The next result asserts that this implies $\Gamma$-convergence in probability.  Toward that end, it is again convenient to work with orthogonal transformations.  Given an orthogonal transformation $\mathcal{O} \in O(d)$, define the rotated cube $Q^{\mathcal{O}}_{r}$ by 
	\begin{equation*}
		Q^{\mathcal{O}}_{r} = \mathcal{O}(Q_{r}).
	\end{equation*}
Similarly, let $Q^{\mathcal{O}}_{r}(x) = x + Q^{\mathcal{O}}_{r}$ denote the cube centered at $x$ with respect to these rotated axes.  Since the map $\mathcal{O} \mapsto \mathcal{O}(e_{1})$ maps surjectively onto $S^{d-1}$, in this way we cover all possible directions.  

\begin{theorem}\label{T: planar homogenization theorem reduction} Assume that the medium $(a,\theta)$ satisfies the assumptions of Section \ref{S: assumptions} and, in addition, that $W$ is a nonnegative continuous function such that $\{u \in \mathbb{R} \, \mid \, W(u) = 0\} = \{-1,1\}$ and satisfying the (super)quadratic growth assumption \eqref{eqn:w-growth-at-inf}. 

Let $\sigma : \mathbb{R}^{d} \to [0,\infty)$ be a positively one-homogeneous convex function. Suppose that the scale $\epsilon \mapsto \delta(\epsilon)$ is chosen in such a way that, for any $(x,\varrho,\mathcal{O}) \in \mathbb{R}^{d} \times (0,\infty) \times O(d)$, we have that 
		\begin{align} \label{E: conv in prob part of the proof}
			&\lim_{\epsilon \to 0} \min \left\{ \frac{\mathscr{F}_{\epsilon,\delta(\epsilon)}(u; Q^{\mathcal{O}}_{\varrho}(x))}{ \varrho^{d-1}} \, \mid \, u - q(\epsilon^{-1}(\cdot - x) \cdot \mathcal{O}(e_{1})) \in H^{1}_{0}(Q^{\mathcal{O}}_{\varrho}(x)) \right\} \\
			&\qquad \qquad \qquad \qquad \qquad \qquad \qquad \qquad \qquad \qquad \qquad \qquad \qquad  = \sigma(\mathcal{O}(e_{1})). \nonumber
		\end{align}
in probability.  Then, for any bounded Lipschitz open set $U \subseteq \mathbb{R}^{d}$ and any $\nu > 0$, 
		\begin{equation*}
			\lim_{\epsilon \to 0} \mathbb{P} \left\{d_{\Gamma}(\mathscr{F}_{\epsilon,\delta(\epsilon)}(\cdot \, ;U), \mathscr{E}_{\sigma}(\cdot \, ;U)) > \nu \right\} = 0,
		\end{equation*}
	where $\mathscr{E}_{\sigma}$ is the anisotropic surface energy determined by $\sigma$ defined in \eqref{E: limit functional}. \end{theorem}
	
Since this result is very similar to related results in \cite{morfe} and \cite{marziani}, its proof is deferred to Appendix \ref{sec:gamma-convergence-in-prob}.

Finally, combining these intermediate results, we arrive at the proof of our main theorem on the homogenization regime.

\begin{proof}[Proof of Theorem \ref{T: homogenization theorem}] In view of Proposition \ref{P: rotated fields}, the limit \eqref{E: conv in prob part of the proof} holds for any $(x,R,\mathcal{O}) \in \mathbb{R}^{d} \times (0,\infty) \times O(d)$ if and only if it holds for any $(x,\varrho) \in \mathbb{R}^{d} \times (0,\infty)$ with $\mathcal{O} = \text{Id}$ held fixed.  Taken together, Theorems \ref{T: planar homogenization theorem upper bound} and \ref{T: planar homogenization theorem lower bound} imply that this is indeed the case.  Therefore, Theorem \ref{T: planar homogenization theorem reduction} implies that, for any bounded Lipschitz open set $U \subseteq \mathbb{R}^{d}$,
	\begin{align*}
		\lim_{\epsilon \to 0} \mathbb{P} \left\{d_{\Gamma}(\mathscr{F}_{\epsilon,\delta(\epsilon)}(\cdot;U), \bar{\mathscr{E}}(\cdot;U)) > \nu \right\} = 0,
	\end{align*}
which is the desired conclusion.\end{proof}

\subsection{Proof of Corollary \ref{cor:upgrade-to-as-convergence}} As in the introduction, under slightly stronger assumptions, convergence in probability can be upgraded to almost-sure convergence.  Again, as is explained in the appendix, this reduces to proving almost-sure convergence of the energy in the planar cell problems.  The next corollary covers the corresponding improvement of Theorem \ref{T: planar homogenization theorem lower bound}:

	\begin{corollary} \label{C: planar homogenization corollary} Assume that the medium $(a,\theta)$ and the potential $W$ satisfy the assumptions of Section \ref{S: assumptions}. If $(\epsilon_{j})_{j \in \mathbb{N}}$ is a sequence of positive numbers converging to zero and $(\delta_{j})_{j \in \mathbb{N}}$ is such that $\epsilon_{j}^{-1} \delta_{j} \to 0$ as $j \to \infty$ and conditions \eqref{E: sublinear summable} and \eqref{E: averaging summable} both hold, then, for any $x \in \mathbb{R}^{d}$ and any $ \varrho > 0$,
		\begin{align*}
			\bar{\sigma}(e_{1}) \leq \liminf_{\epsilon \to 0} \min \left\{ \frac{1}{\varrho^{d-1}} \mathscr{F}_{\epsilon_{j},\delta_{j}}(u; Q_{\varrho}(x)) \, \mid \, u - q(\epsilon_{j}^{-1} (\cdot - x) \cdot e_{1}) \in H^{1}_{0}(Q_{\varrho}(x)) \right\}
		\end{align*}
	with probability one. \end{corollary}
	
For definiteness, the next proof explains how to deduce almost-sure $\Gamma$-convergence from the previous corollary:

	\begin{proof}[Proof of Corollary \ref{cor:upgrade-to-as-convergence}] Fix sequences $(\epsilon_{j})_{j \in \mathbb{N}}$ and $(\delta_{j})_{j \in \mathbb{N}}$ such that $\epsilon_{j} \to 0$ as $j \to \infty$, $\epsilon_{j}^{-1} \delta_{j} \to 0$ as $j \to \infty$, and for which the assumptions \eqref{E: sublinear summable} and \eqref{E: averaging summable} both hold.  Combining the results of Corollary \ref{C: planar homogenization corollary} and Theorem \ref{T: planar homogenization theorem upper bound}, we conclude that, for any $x \in \mathbb{R}^{d}$ and $\varrho > 0$,
		\begin{equation*}
			\lim_{j \to \infty} \min \left\{ \frac{1}{\varrho^{d-1}} \mathscr{F}_{\epsilon_{j},\delta_{j}}(u; Q_{\varrho}(x)) \, \mid \, u - q(\epsilon_{j}^{-1} (\cdot - x) \cdot e_{1}) \in H^{1}_{0}(Q_{\varrho}(x)) \right\} = \bar{\sigma}(e_{1})
		\end{equation*}
	with probability one.  Furthermore, by rotational invariance (Proposition \ref{P: rotated fields}), this remains true if the direction $e_{1}$ and the cube $Q_{\varrho}(x)$ are rotated.  Therefore, by Proposition \ref{P: general convergence} in the appendix (which is the almost-sure version of Theorem \ref{T: planar homogenization theorem reduction}), 
		\begin{align*}
			\lim_{j \to \infty} d_{\Gamma}(\mathscr{F}_{\epsilon_{j},\delta_{j}}(\cdot;U),\bar{\mathscr{E}}(\cdot;U)) = 0 \quad \text{with probability one.}
		\end{align*}  \end{proof}

\section{Relative Error Estimates} \label{S: homogenization 1}

This section and the next treat the proof of Theorem \ref{T: planar homogenization theorem lower bound}. As in the previous section, we begin by fixing a one-dimensional planar boundary condition $ q : \mathbb{R} \to \R $, which is only assumed to satisfy \eqref{E: finite width} and \eqref{E: finite width 2}. Writing $e_{1} = (1,0,,\dots,0) \in \mathbb{R}^{d}$ for the first standard basis vector, recall that our interest is in the analysis of the following limit:
	\begin{align} \label{E: planar limit we want}
		\lim_{\epsilon \to 0} \min \left\{ \frac{1}{\varrho^{d-1}}  \mathscr{F}_{\epsilon,\delta(\epsilon)}(u; Q_{\varrho}(x)) \, \mid \, u - q(\epsilon^{-1} (\cdot - x) \cdot e_{1}) \in H^{1}_{0}(Q_{\varrho}(x)) \right\}.
	\end{align}
The prefactor of $ \varrho^{d-1} $ is natural since heuristically, we expect that the minimizer transitions between the boundary in a neighborhood of a $ (d - 1) $-dimensional surface. Indeed, using the boundary datum as a competitor, we readily obtain an $\epsilon$-independent upper bound of order $\varrho^{d-1}$.

Our analysis begins by rewriting \eqref{E: planar limit we want} via a mesoscopic rescaling. In particular, blowing up space by a factor $ \epsilon^{-1} $, we define
\begin{equation}\label{eqn:length-scales-trafo}
R \coloneqq \frac{1}{\epsilon}
\quad { \rm and } \quad
\gamma(R) \coloneqq \frac{ \delta(\epsilon) }{ \epsilon }
\end{equation}
so that we are now interested in the limit as $R \to \infty$ and the quantity in \eqref{E: planar limit we want} becomes 
	\begin{equation} \label{E: problem of interest to us}
		\min \left\{ \frac{1}{ (\varrho R)^{d-1}} \mathscr{F}_{\gamma(R)}(u; Q_{\varrho R}(Rx)) \, \mid \, u - q((\cdot - Rx) \cdot e_{1}) \in H_{0}^{1}(Q_{\varrho R}(Rx)) \right\}.
	\end{equation}
Here $\mathscr{F}_{\gamma}$ is the functional
	\begin{equation}\label{eqn:rescaled-energy}
		\mathscr{F}_{\gamma}(u; U) = \frac{1}{2} \int_{U} a(\gamma^{-1}x) \nabla u \cdot \nabla u \, dx + \int_{U} \theta(\gamma^{-1}x) W(u) \,dx = \mathscr{F}_{1,\gamma}(u;U)
	\end{equation}
for any open set $ U \subset \R^d $.

The analysis of \eqref{E: problem of interest to us} proceeds in two steps. We begin by defining a homogenized functional $\overline{\mathscr{F}}$ as follows:
	\begin{align*}
		\overline{\mathscr{F}}(u; U) &= \frac{1}{2} \int_{U} \bar{a} \nabla u \cdot \nabla u \, dx + \int_{U} \bar{\theta} W(u) \, dx
	\end{align*}
Well-known homogenization results imply that $\mathscr{F}_{\gamma}(\cdot;U) \overset{\Gamma}{\to} \overline{\mathscr{F}}(\cdot;U)$ for any fixed bounded open set $U \subseteq \mathbb{R}^{d}$ as $\gamma \to 0$.

In the first step, contained in the rest of the present section, we establish a deterministic estimate of the difference $\mathscr{F}_{\gamma} - \overline{\mathscr{F}}$ measured using the correctors associated with the operator $-\nabla \cdot (a(x) \nabla)$ and the $H^{-1}$ norm of $\theta(\gamma^{-1} \cdot) - \bar{\theta}$.  It bears emphasizing that this part of the proof does not involve any probabilistic arguments.  For this reason, a set of distilled, deterministic assumptions on the medium $(a,\theta)$ is imposed in Section \ref{S: deterministic assumptions} that replace those in Section \ref{S: assumptions} in this section only.  Later, in Section \ref{S: homogenization 2}, we will prove that the probabilistic assumptions of Section \ref{S: assumptions} imply that those of Section \ref{S: deterministic assumptions} hold with probability one.

In the second step, carried out in Section \ref{S: homogenization 2}, we use the statistical or self-averaging properties of the medium to show that the error becomes negligible provided $R$ does not grow too fast.  At a purely qualitative level, the soft criteria \eqref{E: sublinear part} and \eqref{E: averaging part} are sufficient for this purpose.

\begin{remark} \label{R: abuse of notation for q} Throughout this section, to lighten the notation, we abuse notation by identifying $q$ with its extension $x \mapsto q(x \cdot e_{1})$ to $\mathbb{R}^{d}$.  \end{remark}

\subsection{(Deterministic) Assumptions} \label{S: deterministic assumptions} As mentioned above, throughout this section, the arguments are entirely deterministic.  No properties of the underlying probability space are used, and, in particular, probabilistic (or ergodic theoretic) arguments can be postponed until the next section. Toward that end, it is convenient to make precise exactly the properties of the medium $(a,\theta)$ that are used here.

\begin{remark} In what follows, we allow ourselves the flexibility to choose correctors depending on the domain, that is, we work with correctors $(\phi_{Q},\sigma_{Q})$ that vary with the cube $Q \subseteq \mathbb{R}^{d}$.  This is convenient in the almost-periodic setting for technical reasons, see the discussion in Section \ref{S: uniformly almost periodic}, particularly Remark \ref{R: sub remark}. \end{remark}

We assume that, in addition to the medium $(a,\theta)$ and effective coefficients $\bar{a}$ and $\bar{\theta}$, for each axis-aligned cube $Q \subseteq \mathbb{R}^{d}$ and each $\xi \in \mathbb{R}^{d}$, there are functions 
	\begin{align*}
		\phi_{Q,\xi} : Q \to \mathbb{R}, \quad \sigma_{Q,\xi} : Q \to \mathbb{R}^{d \times d}
	\end{align*}
and a constant matrix $\bar{a}(Q) \in \text{Sym}(d)$ such that the following conditions hold:

\subsubsection{Bounds on $(a,\theta,\bar{a},\bar{a}(Q),\bar{\theta})$} We assume the following pointwise bounds on $a$ and $\theta$
	\begin{equation}\label{eqn:det-bounds-a-theta}
		\lambda \text{Id} \leq a(y) \leq \Lambda \text{Id}, \quad \theta_{*} \leq \theta(y) \leq \theta^{*} \quad \text{for each} \quad y \in \mathbb{R}^{d},
	\end{equation}
together with the identical bounds on the constants $\bar{a}$, $\bar{a}(Q)$, and $\bar{\theta}$
	\begin{equation}\label{eqn:det-bounds-bara-bartheta}
		\lambda \text{Id} \leq \bar{a},\bar{a}(Q) \leq \Lambda \text{Id}, \quad \theta_{*} \leq \bar{\theta} \leq \theta^{*}.
	\end{equation}
	
\subsubsection{Helmholtz-type Decomposition} We assume $\phi_{Q, \xi} \in H^{1}(Q)$; $\sigma_{Q,\xi} \in L^{2}(Q; \mathbb{R}^{d \times d}) $ is a skew-symmetric matrix field with $\nabla \cdot \sigma_{Q,\xi} \in L^{2}(Q; \mathbb{R}^{d})$; and they combine to provide a Helmholtz-type decomposition of the vector field $a(y) \xi$:
	\begin{align}\label{eqn:a-helmholtz-decomposition}
		a(y) \xi = \bar{a}(Q) \xi - a(y) \nabla \phi_{Q,\xi}(y) + (\nabla \cdot \sigma_{Q,\xi})(y) \quad \text{for a.e.} \quad y \in Q.
	\end{align}
Above the divergence $\nabla \cdot \sigma$ of a matrix field $\sigma$ is given by $(\nabla \cdot \sigma)_{i} = \sum_{j = 1}^{d} \frac{\partial \sigma_{ji}}{\partial x_{j}}$. 

Additionally, we assume that the map $\xi \mapsto (\phi_{Q,\xi},\sigma_{Q,\xi})$ is linear in $\xi$, that is, for any $\xi_{1}, \xi_{2} \in \mathbb{R}^{d}$ and any $\alpha \in \mathbb{R}$,
	\begin{gather*}
		(\phi_{Q,\alpha \xi_{1} + \xi_{2}},\sigma_{Q,\alpha \xi_{1} + \xi_{2}}) = (\alpha \phi_{Q,\xi_{1}} + \phi_{Q,\xi_{2}},\alpha \sigma_{Q,\xi_{1}} + \sigma_{Q,\xi_{2}}) \quad \text{a.e. in} \, \, Q.
	\end{gather*} 

\subsection{${\rm Sub}(Q)$ and ${\rm Osc}(Q)$} \label{S: further notation} As in the introduction, we concatenate the correctors by introducing the vector-valued $\phi_{Q} \coloneqq (\phi_{Q,e_{1}},\dots,\phi_{Q,e_{d}})$ and tensor-valued $\sigma_{Q} \coloneqq (\sigma_{Q,e_{1}},\dots,\sigma_{Q,e_{d}})$.  We then define ${ \rm{ Sub } }(Q)$, and ${\rm { Osc } } ( Q )$ by	
	\begin{align*}
		{ \rm{ Sub } } ( Q ) &= \frac{1}{|Q|^{\frac{1}{d}}}  \left( \fint_{Q} | (\phi_{Q}(y),\sigma_{Q}(y)) - \fint_{Q} (\phi_{Q}(y'),\sigma_{Q}(y')) \, dy' |^{2} \, dy \right)^{\frac{1}{2}}, \\
		{ \rm Osc }(Q) &= \frac{1}{|Q|^{\frac{1}{d}}} \| \theta - \bar \theta \|_{ H^{-1} (Q) }, \phantom{\int} 
	\end{align*}
	where the $ H^{-1} $ norm is normalized as in \eqref{E: H minus one norm}.

Notice that the assumption \eqref{eqn:a-helmholtz-decomposition} is invariant under subtracting the mean from $ \phi_{Q} $ and $ \sigma_{Q} $. This motivates subtracting the mean in the definition of $ { \rm Sub } (Q) $.

\subsection{Invariance under Rescaling} It is important to note that the assumptions of Section \ref{S: deterministic assumptions} are invariant under rescaling.  In particular, for any $\gamma > 0$, if we define the rescaled fields $ a^{\gamma} $ and $\theta^{\gamma}$ by 
	\begin{align} \label{E: rescaled fields}
		a^{\gamma}(x) = a(\gamma^{-1} x ), \quad \theta^{\gamma}(x) = \theta(\gamma^{-1}x),
	\end{align} 
then the correctors should also be rescaled by defining $\phi^{\gamma}_{Q}$ and $\sigma^{\gamma}_{Q}$  by 
	\begin{align} \label{E: rescaled correctors}
		 \quad \phi_{Q}^{\gamma}(x) = \gamma \phi_{\gamma^{-1} Q}(\gamma^{-1} x), \quad \sigma_{Q}^{\gamma}(x) = \gamma \sigma_{\gamma^{-1} Q}(\gamma^{-1} x).
	\end{align}
An immediate computation then shows that, for each $i \in \{1,2,\dots,d\}$,
	\begin{align} \label{E: helmholtz}
		a^{\gamma}(y) e_{i} = \bar{a}(Q) e_{i} - a^{\gamma}(y) \nabla \phi_{Q,e_{i}}^{\gamma}(y) + ( \nabla \cdot \sigma_{Q,e_{i}}^{\gamma} )(y) \quad \text{for a.e.} \quad y \in Q.
	\end{align}

\subsection{Regularized Functional}\label{sec: regularized_potential_reg} In the first step of the proof, we utilize a regularized functional $\mathscr{F}^{\rm reg}_{\gamma}$ obtained from $\mathscr{F}_{\gamma}$ by replacing $W$ with a suitably regularized version $W_{\rm reg}$. We must emphasize that this is possible because the minimizers in the planar cell problems take values in $[-1,1]$. This is made precise in the next lemma.


\begin{lemma}\label{L: minus one one} Given $R > 0$, if $u \in H^{1}(Q_{R})$ is such that 
	\begin{align*}
		\mathscr{F}_{\gamma}(u; Q_{R}) = \min \left\{ \mathscr{F}_{\gamma}(v; Q_{R}) \, \mid \, v(y) = q(y \cdot e_{1}) \, \, \text{for each} \, \, y \in \partial Q_{R} \right\},
	\end{align*}
then $-1 \leq u \leq 1$ in $Q_{R}$. \end{lemma}

	\begin{proof} By the assumptions of Section \ref{S: assumptions}, the potential $W$ is strictly positive outside of the set $[-1,1]$.  Thus, $\mathscr{F}_{\gamma}(\max\{\min\{u,1\},-1\}, Q_{R}) \leq \mathscr{F}_{\gamma}(u;Q_{R})$, with strict inequality if $ | u | > 1 $ on a set of positive measure. At the same time, by assumption \eqref{E: finite width 2}, the boundary datum $q$ takes values in $[-1,1]$, so $\max\{\min\{u,1\},-1\} \in q + H^{1}_{0}(Q_{R})$.  Therefore, by minimality, $u = \max\{\min\{u,1\},-1\}$ a.e. \end{proof}

Due to the fact that $u$ takes values in $[-1,1]$, we are free to change the definition of $W$ outside of that interval without changing the energy.  In this way, it will be convenient to work with a modification of $\mathscr{F}_{\gamma}$ that restricts to a strictly convex functional at sufficiently small length scales --- or, put differently, the gradient term dominates the potential term at these scales.  A convenient way to do this is to replace $W$ by a regularized potential $W_{\text{reg}}$ such that
	\begin{gather}
		W_{\text{reg}} = W \, \, \text{in a neighborhood of} \, \, [-1,1], \nonumber \\
		[W_{\text{reg}}']_{C^{0,1}(\mathbb{R})} < \infty, \quad \liminf_{|u| \to \infty} W_{\text{reg}}(u) = \infty, \label{E: regularized_potential_reg} \\
		W_{\text{reg}}'(u) = W_{\text{reg}}'(2) \quad \text{if} \quad u \geq 2, \quad W_{\text{reg}}'(u) = W_{\text{reg}}'(-2) \quad \text{if} \quad u \leq -2. \nonumber
	\end{gather}
We define the regularized energy $\mathscr{F}^{\text{reg}}_{\gamma}$ by using $W_{\text{reg}}$ as the potential
	\begin{equation*}
		\mathscr{F}^{\text{reg}}_{\gamma}(u;U) = \frac{1}{2} \int_{U}  a^{\gamma} \nabla u \cdot \nabla u + \int_{U} \theta^{\gamma} W_{\text{reg}}(u).
	\end{equation*}
Define $\overline{\mathscr{F}}^{\text{reg}}$ similarly.  This is a convenient modification for two reasons.  On the one hand, the modified functional coincides with the original one when restricted to functions between $-1$ and $1$, that is,	
	\begin{equation*}
		\mathscr{F}_{\gamma}^{\text{reg}}(u;U) = \mathscr{F}_{\gamma}(u;U) \quad \text{if} \quad -1 \leq u \leq 1 \quad \text{in} \, \, U.
	\end{equation*}
By Lemma \ref{L: minus one one}, this applies, in particular, to minimizers in our planar cell problem. On the other hand, in the homogenization arguments given below, certain error estimates involving the potential $W$ trivialize when it is replaced by $W_{\text{reg}}$.

The existence of a $W^{\text{reg}}$ as above follows readily from the assumptions in Section \ref{S: assumptions}. See Figure \ref{F: regularized well} for a depiction of the construction.

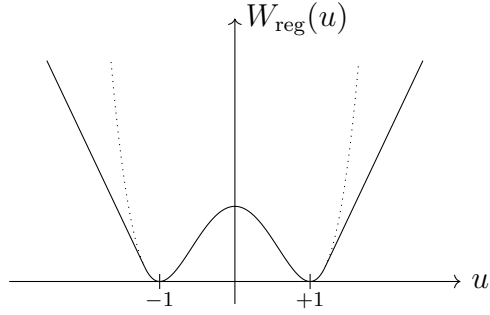
\begin{figure}

\centering

\begin{tikzpicture}

	\draw[->] (-3,0) -- node[at end, right] { $ u $ } (3, 0);
	\draw[->] (0,-0.3) -- node[at end, right] { $ W_{ \rm reg } ( u ) $ } (0, 3.5);

	\draw plot[domain = -1.2 : 1.2, samples=200] (\x, { ( 1 - \x * \x ) * ( 1 - \x * \x ) });
	\draw plot[domain = 1.2 : 2.5 ] (\x, { 2.112 * ( \x - 1.2 ) + 0.1936 });
	\draw plot[domain = -2.5 : -1.2 ] (\x, { - 2.112 * ( \x + 1.2 ) + 0.1936 });

	\draw[dotted] plot[domain = -1.6458 : -1.2, samples=200] (\x, { ( 1 - \x * \x ) * ( 1 - \x * \x ) });
	\draw[dotted] plot[domain = 1.2 : 1.6458, samples=200] (\x, { ( 1 - \x * \x ) * ( 1 - \x * \x ) });

	\draw [-] ( -1, -0.1 ) -- node[below] {\tiny $ - 1 $ } ( -1,  0.1 );
	\draw [-] ( 1, -0.1 ) -- node[below] {\tiny $ +1 $ } (  1,  0.1 );

\end{tikzpicture}

\caption{A possible regularization $ W_{ \rm reg } $ of the classical quartic well $ W(u) = (1-u^2)^2 $ (dotted) by a piecewise linear extension (line).}
\label{F: regularized well}
\end{figure}

Concerning the strict convexity of $\mathscr{F}_{\gamma}^{\text{reg}}$ and $\overline{\mathscr{F}}^{\text{reg}}$ at small scales, we use the following proposition, which is relatively well-known.

\begin{prop}\label{prop:convexity} There is an $r_{c} > 0$ depending only on $\lambda$, $ \theta^* $ and $W_{\text{reg}}$ such that if $Q$ is a cube of side length $ r \leq r_{c}$ and $u \in H^{1}(Q)$, then
		\begin{equation*}
			\mathscr{F}_{\gamma}^{ \rm reg }(v; Q) \geq \mathscr{F}^{\text{reg}}_{\gamma}(u; Q) + \langle D\mathscr{F}^{\text{reg}}_{\gamma}(u;Q), v - u \rangle + \frac{\lambda}{4} \int_{Q} |\nabla v - \nabla u|^{2} \, dx
		\end{equation*}
	for any $v \in u + H_{0}^{1}(Q)$. The same conclusion holds true for $ \overline{\F}^{ \rm reg } $.
\end{prop}


\begin{proof}
For simplicity let us write $ \F = \F_{\gamma}^{ \rm reg } (\cdot, Q) $.  It is well known that under the condition \eqref{E: regularized_potential_reg}, $ \mathscr{F} $ is twice Fréchet differentiable with
\begin{align}\label{eq:derivative-energy}
\langle D\F(u), h \rangle &= \int_{Q} a^{\gamma} \nabla u \cdot \nabla h + \int_{Q} \theta^{\gamma} W_{\text{reg}}'(u) h, \\
D^2\F (u)( h, k ) &= \int_{Q} a^{\gamma} \nabla h \cdot \nabla k + \int_{Q} \theta^{\gamma} W_{\text{reg}}''(u) h k , \nonumber
\end{align}
where $ h, k \in H^1_0(Q) $. Hence
$$
\begin{aligned}
D^2\F(u)( h, h )
&\ge \lambda \int_{Q} | \nabla h |^2 - \theta^* \sup W_{\text{reg}}'' \int_{Q} | h |^2 \\
&\ge \left( \lambda - r^2 \, C_{ \text{Poincaré} } \, \theta^* \, \sup W_{\text{reg}}'' \right) \int_{Q} | \nabla h |^2,
\end{aligned}
$$
where $ r$ is the side length of $Q$. Choosing $ r \leq r_c $ sufficiently small yields 
$$
D^2\F(u)( h, h ) \ge \frac{\lambda}{2} \int_{Q} | \nabla h |^2,
$$
which implies the desired estimate.
\end{proof}

As hinted already above, the regularized functional $\mathscr{F}^{\rm reg}_{\gamma}$ will be useful when we compare the minimizer $u$ of \eqref{eqn: T: elliptic term cube decomposition variation problem} to its two-scale expansion; see the proof of Proposition \ref{prop:homogenization-interior-qubes} below.  For now, the next result already hints at the utility of our regularization.

	\begin{prop} \label{P: uniqueness comparison} If $Q$ is a cube of side length $r_{c}$ and $g \in H^{1}(Q)$ satisfies $- 1 \leq g \leq 1$ in $Q$, then there are unique $u, \bar{u} \in H^{1}(Q)$ such that 
		\begin{align*}
			\mathscr{F}_{\gamma}(u; Q) &= \min \left\{ \mathscr{F}_{\gamma}(v; Q) \, \mid \, v \in g + H^{1}_{0}(Q) \right\}, \\
			\overline{\mathscr{F}}_{\gamma}(\bar{u}; Q) &= \min \left\{ \overline{\mathscr{F}}_{\gamma}(v; Q) \, \mid \, v \in g + H_{0}^{1}(Q) \right\}.
		\end{align*}
	Furthermore, $-1 \leq u, \bar{u} \leq 1$ in $Q$ and 
		\begin{align*}
			\frac{\max\{\lambda,\theta_{*}\}}{\min\{\Lambda,\theta^{*}\}} \leq \frac{\mathscr{F} (u;Q)}{\overline{\mathscr{F}}(\bar{u};Q)} \leq \frac{\max\{\Lambda,\theta^{*}\}}{\min\{\lambda,\theta_{*}\}}.
		\end{align*}
	\end{prop}
	
\begin{proof}
	Any two minimizers $u$ and $\bar{u}$ solve the Euler-Lagrange equation
	\begin{align}\label{eqn:prop-on-minimiality-1}
	- \nabla \cdot a^{\gamma} \nabla u + \theta^{\gamma} W'( u ) = 0 = - \nabla \cdot \bar a \nabla \bar u + \bar \theta W'( \bar u ) \quad { \rm in } ~ Q;
	\quad u = \bar u = g \quad { \rm on ~ } Q.
	\end{align}
	Since $ - 1 \leq g \leq 1 $, the proof of Lemma \ref{L: minus one one} implies that $ -1 \leq u, \bar u \leq 1 $. Hence, by \eqref{E: regularized_potential_reg}, we may replace $ W $ by $ W_{ \rm reg } $ in \eqref{eqn:prop-on-minimiality-1}. Therefore $ u $, $ \bar u $ are also minimizers of the strictly convex energies $ \mathscr{F}^{\text{reg}}_{\gamma} $ and $ \overline{\mathscr{F}}^{\text{reg}} $, and are therefore unique.
	
	Concerning the estimate on the ratio of the energies, this follows from the fact that $\lambda \text{Id} \leq a, \bar{a} \leq \Lambda \text{Id}$ and $\theta_{*} \leq \theta, \bar{\theta} \leq \theta^{*}$ by the assumptions of Section \ref{S: deterministic assumptions}.
\end{proof}

\begin{figure}

\centering

\begin{tikzpicture}
	\draw (0,0) rectangle (5 + 0.5, 5.5);
	\draw [dotted] (-1, 5.5) -- (0, 5.5);
	\draw [dotted] (-1, 0) -- (0, 0);
	\draw [|-|] (-1,0) -- node[left] { $ \varrho R $ } (-1, 5.5);

	\foreach \x in {0, ..., 10}
		\foreach \y in {0, ..., 10}
			\draw (\x / 2, \y / 2) rectangle (\x / 2 + 1 / 2, \y /2 + 1 / 2);

	\draw [dotted] (5.5, 4.5) -- (6.5, 4.5);
	\draw [dotted] (5.5, 4) -- (6.5, 4);
	\draw [|-|] (6.5, 4.5) -- node[right] { $ r_c \sim 1 $ } (6.5, 4);

	\draw (7.75 + 0.5, 3.04) -- (7.75 + 0.5, 0.45);
	\draw (6.25 + 0.5, 3.04) -- (6.25 + 0.5, 0.45);
	\draw (5.70 + 0.5, 2.50) -- (8.29 + 0.5, 2.50);
	\draw (5.70 + 0.5, 1.00) -- (8.29 + 0.5, 1.00);

	\foreach \x in {1, ..., 9}
		\draw (6.5 + \x / 10 + 0.5, 2.25) -- (6.5 + \x / 10 + 0.5, 1.25);
	\foreach \y in {1, ..., 9}
		\draw (6.5 + 0.5, 1.25 + \y / 10) -- (7.5 + 0.5, 1.25 + \y / 10);

	\draw [dotted] (7.5 + 0.5, 1.75 + 1/10) -- (9 + 0.5, 1.75 + 1/10);
	\draw [dotted] (7.5 + 0.5, 1.75) -- (9 + 0.5, 1.75);
	\draw [|-|] (9 + 0.5,1.75 + 1/10) -- node[right] { $ \gamma $ } (9 + 0.5,1.75);

	\draw[color=lightgray] (3.75 + 0.5, 1.75) circle (0.5);
	\draw[color=lightgray] (7 + 0.5, 1.75) circle (1.5);
\end{tikzpicture}

\caption{Cube decomposition \eqref{eqn:partition-of-small-cubes} and a schematic picture of the three length scales $ \varrho R \gg 1 $, $ r_c \sim 1 $, and $ \gamma \ll 1 $}
\label{fig:cube-decomposition}

\end{figure}
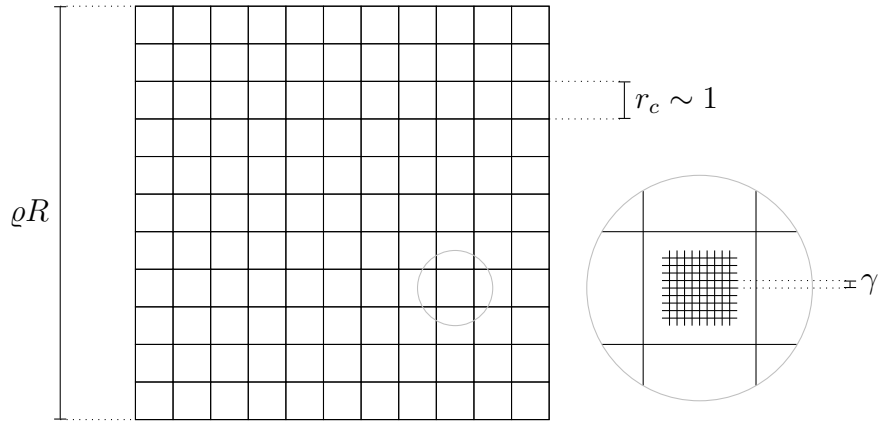

\subsection{Cube Decomposition}The length scale $ r_{c} $ of the previous result suggests a cube decomposition as depicted in Figure \ref{fig:cube-decomposition}. It involves three parameters, corresponding to the macroscopic, mesoscopic, and microscopic scales. There is the cube size $ \varrho R \gg 1 $; the mesh size $ r_c $, which is fixed from now on, so that it is of order one; and the length scale $ \gamma \ll 1 $ of the underlying medium. Our strategy to establish the closeness of $ \mathscr{F}_{ \gamma } $ and $ \overline{ \mathscr{F} } $ involves homogenization error estimates employed on each subcube. The main difficulty comes from the surface scaling $ (\varrho R)^{d-1} $ of the energy in \eqref{E: problem of interest to us}. Since we use a cover consisting of approximately $ r_{c}^{-d} ( \varrho R )^{d} $ cubes, it is crucial that these estimates are relative to the energy in the cube under consideration. Otherwise, their sum leads to the wrong scaling.

Before giving the precise statements, let us fix some notation regarding the decomposition. Given the cube $Q_R $ with $ R = (2K+1) r_{c} $ for some $K \in \mathbb{N} $, we observe that we can write $Q_R $ as a union of subcubes of side length $r_{c}$:
	\begin{equation}\label{eqn:partition-of-small-cubes}
		Q_R = \bigcup_{ z \in \mathbb{Z}^{d} \cap [-K,K]^{d} } Q_{r_{c}}(r_{c} z)
	\end{equation}
The estimate below, comprising the first step of the proof of Theorem \ref{T: planar homogenization theorem lower bound}, shows that the macroscopic error is controlled by the maximum of the mesoscopic averaging errors in each subcube. As suggested already above, this is a relative estimate: the error equals a small multiple of the total energy.  In particular, it scales like $R^{d-1}$.  

	\begin{theorem} \label{T: elliptic term cube decomposition} There is a modulus of continuity $\omega : [0,\infty) \to [0,\infty)$ with the following property: given any $K \in \mathbb{N}$, if $R = ( 2K + 1) r_{c} $ and if $ u \in H^{1}(Q_R) $ is such that 
		\begin{align}\label{eqn: T: elliptic term cube decomposition variation problem} 
			\mathscr{F}_{\gamma}(u; Q_R) = \min \left\{ \mathscr{F}_{\gamma}(v; Q_R) \, \mid \, v - q \in H^{1}_{0}(Q_R) \right\}, \quad u - q \in H^{1}_{0}(Q_R),
		\end{align}
	then there is a $\tilde{u} \in q + H^{1}_{0}(Q_R) $ such that
		\begin{equation*}
		\begin{aligned}
			& \frac{ |\mathscr{F}_{\gamma}(u; Q_R) - \overline{\mathscr{F}}(\tilde{u}; Q_R)| }{ \mathscr{F}_{\gamma}(u ; Q_{R} ) } \\
			&\qquad \leq \max_{ z \in \mathbb{Z}^{d} \cap [-K,K]^{d}} \omega \big( { \rm Sub } (Q_{r_{c}\gamma^{-1}}(r_{c}z)) + { \rm Osc}(Q_{r_{c}\gamma^{-1}}(r_{c}z)) + | \bar a  ( Q_{r_{c}\gamma^{-1}}(r_{c}z) ) - \bar a | \big),
		\end{aligned}
		\end{equation*}
	where ${ \rm Sub }(\cdot) $ is defined by \eqref{eqn:defn-sub} and $ { \rm Osc }(\cdot) $ by \eqref{eqn:defn-osc}.
	\end{theorem}
	
Above we abuse the notation for $q$, following Remark \ref{R: abuse of notation for q}.

In the theorem above, it really is necessary to consider the maximum of the mesoscopic homogenization errors $ { \rm Sub }_{ r_c z } \left( \gamma^{-1} r_c \right) $ and $  { \rm Osc }_{ r_c z } \left( \gamma^{-1} r_c \right) $ --- the counterexamples of the companion paper \cite{part2} show that homogenization may fail if the medium is far from its average even in a finite string of mesoscopic cubes. Of course, in order to control the suprema above in full generality, it is necessary to ensure that the macroscopic length $R =  ( 2 K + 1 )  r_{c}$ does not grow too fast relative to the inverse of the microscopic parameter $\gamma$. This is the reason why homogenization may fail if $\delta \epsilon^{-1}$ decays too slowly with $\epsilon$.

With Theorem \ref{T: elliptic term cube decomposition} in hand, all that is left in the proof of homogenization is to establish that if $K$ does not grow too fast relative to $\gamma^{-1}$, then 
	\begin{equation} \label{E: what we want to prove averaging}
		\sup_{z \in \mathbb{Z}^{d} \cap [-K,K]^d } { \rm Osc } ( Q _{r_{c} \gamma^{-1}} (r_{c} z) ) + { \rm Sub } ( Q_{r_{c} \gamma^{-1}} ( r_{c} z ) ) + | \bar{a}(Q_{r_{c} \gamma^{-1}}(r_{c} z) - \bar{a}| \to 0
	\end{equation}
in probability. The details will be made precise in Section \ref{S: homogenization 2}.

\subsection{Properties of Minimizers on Subcubes} In the proof of Theorem \ref{T: elliptic term cube decomposition}, we will need some elementary properties of and estimates on the minimizer $ u $ of \eqref{eqn: T: elliptic term cube decomposition variation problem} restricted to the subcubes in the decomposition \eqref{eqn:partition-of-small-cubes}. For the reader's convenience, we collect them in this subsection.

First, our proof relies on the following PDE lemma, which shows that $ u $ is uniformly (in the parameter $ \gamma $) Hölder continuous at mesoscopic scales.

\begin{lemma}\label{lemma:dg-n-m}
Suppose $ u \in H^1(Q_R) $ with $ - 1 \le u \le 1 $ solves the equation $ - \nabla \cdot a^{\gamma}  \nabla u + \theta^{\gamma} W'(u) = 0 $ in $ Q_R $; $ u(x) = q(x \cdot e_{1})$ on $ \partial Q_R $. Then there exist constants $ 0 < \alpha < 1 $ and $ C_{ \text{Hölder} } > 0 $, which depend only on $ q $, $ \lambda $, $ \Lambda $, $ \theta_{*} $, $ \theta^{*} $, $ W $ and the dimension $ d $, such that
$$
[ u ]_{ C^{0,\alpha} (Q) } \leq C_{ \text{Hölder} }
$$
for every cube $ Q \subset Q_{ R }(x) $ of side length $ r_c $.
\end{lemma} 

The lemma is an application of the classical De Giorgi-Nash-Moser theorem. Its proof is given in Section \ref{section:pde-estimates}.

In our proof of Theorem \ref{T: elliptic term cube decomposition}, we will use the lemma above together with local and global energy bounds that we state precisely for the reader's convenience in the next proposition.

\begin{prop}\label{prop:assumptions-homogenization-interior-qubes}
	Let $ u $ be a minimizer of the variational problem \eqref{eqn: T: elliptic term cube decomposition variation problem}. Then $ u $ satisfies the bound
	\begin{equation*}
	\quad \frac{\lambda}{2} \int_{ Q_R } | \nabla u |^2 + \theta_{*} \int_{Q_{R}} W(u)
	\lesssim R^{d-1}.
	\end{equation*}
	Furthermore, on every cube $ Q $ of side length $ r_c $
	\begin{equation}\label{eqn:hom-int-cubes-assumption}
	-1 \leq u \leq 1 \quad \text{in} \, \, Q,
	\quad [ u ]_{ C^{0,\alpha}( \overline{Q} ) } \le C_{ \text{Hölder} },
	\quad \int_{ Q } | \nabla u |^2 + \int_{Q} W(u) \leq C_{ \text{Energy} },
	\end{equation}
	where the constants $ C_{ \text{Hölder} } $ and $ C_{ \text{Energy} } $ only depend $ q $, $ \lambda $, $ \Lambda $, $ \theta_{*} $, $ \theta^{*} $, $ W $ and $ r_c $.
\end{prop}

	\begin{proof} Considering the function $v(x) = q(x \cdot e_{1})$ as a competitor, one deduces that $\mathscr{F}_{\gamma}(u;Q_{R}) \leq R^{d-1} \left( \frac{ \Lambda }{ 2 } \int_{ - \infty }^{ \infty } | q' |^2 + \theta^* \int_{ - \infty }^{ \infty } W'(q) \right)$, the constant in the parentheses being finite by assumption \eqref{E: finite width}.  
	
The $L^{\infty}$ and H\"{o}lder bounds are direct consequences of Lemmas \ref{L: minus one one} and \ref{lemma:dg-n-m}, and the local energy bound follows from an application of Caccioppoli's inequailty, see, e.g., \eqref{eq:interior-caccioppoli} and \eqref{eq:boundary-caccioppoli} below.
\end{proof}

In addition to the basic estimates of the previous proposition, we need some higher integrability of $\nabla u$ which is provided by the next lemma.

\begin{remark} To simplify the notation in our estimates, we adapt the convention that for any cube $ Q \subseteq \mathbb{R}^{d}$, we denote by $ \varrho Q $ the cube that has the same center as $ Q $ but with the radius multipled by $ \varrho $. \end{remark}

\begin{lemma}[Meyer's estimate]\label{lemma:meyers-estimate}
	Let $ u \in H^1(Q_R) $ be a minimizer of \eqref{eqn: T: elliptic term cube decomposition variation problem} and denote by $ Q $ a cube of side length $ r_c $. There exists some $ p = p(d) > 2 $ such that:

	\begin{enumerate}
	\item If $ 2 Q \subset Q_R $, then
	$$
	\left( \int_{Q} |\nabla u|^p \right)^{\frac{1}{p}}
	\lesssim \left( \int_{2 Q} |\nabla u|^2 + \int_{2Q} W(u) \right)^{\frac{1}{2}} .
	$$

	\item If $ Q \subset Q_R $ and $ \partial Q \cap \partial Q_R \neq \emptyset $, then
	$$
	\begin{aligned}
	\left( \int_{Q \cap Q_R} |\nabla u|^p \right)^{\frac{1}{p}}
	&\lesssim \left( \int_{2 Q \cap Q_R} |\nabla u|^2 + \int_{2Q \cap Q_R} W(u) \right)^{\frac{1}{2}} \\
	&\quad\quad + \left( \int_{2 Q \cap Q_R} |\nabla q|^p \right)^{\frac{1}{p}}.
	\end{aligned}
	$$
	\end{enumerate}
\end{lemma}

Once again, in the previous lemma, we abuse the notation for $q$ as in Remark \ref{R: abuse of notation for q}.

Of importance for us are only estimates in terms of energetic quantities on the right hand side. On a technical level this boils down to replacing $ W'(u) $ terms (that come from standard linear PDE estimates) by $ W(u) $. For the classical double-well potential $ W(u) = (1 - u^2)^2 $ one expects
$$
\left( \int | W'(u) |^p \right)^{\frac{1}{p}} \approx \left( \int | W(u) |^{\frac{p}{2}} \right)^{\frac{1}{p}}
$$
away from the transition of $ u $. This step is made rigorous in the proof of Lemma \ref{lemma:meyers-estimate} in Section \ref{section:pde-estimates} by using the so-called \emph{clearing-out property}, which is well known in the literature,~cf.~Theorem 6 in \cite{Bethuel}, and is adapted to our setting in the next lemma. 

\begin{lemma}[Clearing-out property]\label{lemma:clearing-out}
Let $ Q $ be a cube of side length $ r_c $. For every $ \varepsilon > 0 $ there exists a $ \delta > 0 $ such that, for every $ u \in H^1(Q)  $ that satisfies \eqref{eqn:hom-int-cubes-assumption}, the following implication holds:
$$
\int_{Q} | \nabla u |^2 + \int_{Q} W(u) < \delta
\quad\Rightarrow\quad \min \{ \sup_{Q} | u - 1 |, \sup_{Q} | u + 1 | \} < \varepsilon.
$$
Moreover, the same conclusion holds true for minimizers $ \bar u $ of the energy functional $ \overline{ \mathscr{F} } (\cdot ; Q) $ on $ u + H^1_0(Q) $.
\end{lemma}

The clearing out property will turn out to be helpful due to the following observation. By \eqref{eqn:w-non-degeneracy} both $ W $ and $ W' $ look like a polynomial near the minima of $ W $. Indeed, let us focus on the minimum at $ u = 1 $, where
\begin{align*}
W(u) = W^{ ( 2 \kappa ) }(1) ( u - 1 )^{ 2 \kappa } + r(u) ( u - 1 )^{ 2 \kappa }
\quad { \rm with } \quad \lim_{ u \rightarrow 1} r(u) = 0,
\end{align*}
so that
\begin{align}\label{eqn:w-poly}
\frac{ W(u) }{ W^{ ( 2 \kappa ) }(1) | u - 1 |^{ 2 \kappa } } \sim 1, \quad
\frac{ W'(u) }{ W^{ ( 2 \kappa ) }(1) | u - 1 |^{ 2 \kappa -1 } } \sim 1 \quad { \rm near } \quad u = 1.
\end{align}
In the small energy regime, the clearing-out lemma shows that we can appeal to \eqref{eqn:w-poly}.

Lastly, when comparing $ \F_{\gamma} $ to $ \overline{\F} $ using homogenization techniques, we will also need some estimates for the solution of the homogeneous problem on each subcube.

\begin{lemma}[Meyer's estimate]\label{lemma:homogeneous-calderon-zygmund-estiamte}
	Let $ Q $ be a cube of side length $ r_c $ and suppose $ \bar u \in H^1(Q) $ minimizes $ \overline{\mathscr{F}}( \cdot ; Q) $ on $ u + H^1_0(Q) $, where the boundary datum $ u $ satisfies \eqref{eqn:hom-int-cubes-assumption}. Then
	$$
	\left( \int_Q | \nabla \bar u |^p \right)^{\frac{1}{p}} \lesssim \left( \int_{Q} | \nabla \bar u |^2 + \int_{Q} W(\bar u) \right)^{\frac{1}{2}} + \left( \int_{Q} |\nabla u |^p \right)^{\frac{1}{p}}
	$$
	for some $ p = p(d) > 2 $.
\end{lemma}

Let us remark that the above estimate is the same statement as the Meyer's estimate in Lemma \ref{lemma:meyers-estimate} for the heterogeneous medium. In fact, the same proof applies. We comment on the details in Section \ref{section:pde-estimates}.

Finally, we will need interior Schauder estimates for minimizers $\bar{u}$ of the homogenized energy $\overline{\mathscr{F}}$.  Once again, the interest here is our bounds are controlled by the energy itself.

\begin{lemma}[Interior Schauder estimates]\label{lemma:homogeneous-interior-schauder-estimate}
	Let $ Q $ be a cube of side length $ r_c $ and suppose $ \bar u \in H^1(Q) $ minimizes $ \overline{\mathscr{F}}( \cdot ; Q) $ on $ u + H^1_0(Q) $, where the boundary datum $ u $ satisfies \eqref{eqn:hom-int-cubes-assumption}.  Then
	$$
	\sup_{\varrho Q} | \nabla^2 \bar u | + | \nabla \bar u | \lesssim_{\varrho} \left( \int_{Q} | \nabla \bar u |^2 + \int_{Q} W(\bar u) \right)^{\frac{1}{2}}
	$$
	for all $ 0 < \varrho < 1 $.
\end{lemma}

As remarked earlier, it is crucial to have an estimate with energetic quantities on the right hand side. The r.h.s.~of the classical Schauder estimate for the Euler-Lagrange equation depends on the norm $ \Vert W'(u) \Vert_{ C^{0,\alpha} } $. Via the De Giorgi-Nash-Moser estimate this can be replaced by\footnote{Here $ L^{\frac{d}{2}+} $ is supposed to denote $ L^{ \frac{d}{2} + \varepsilon } $ for some arbitrary small $ \varepsilon > 0 $.} $ \| W'(u) \|_{ L^{ \frac{d}{2} + } }  $. For the classical quartic well this expression scales like a $ L^{ \frac{d}{2} + } $-norm while $ ( \int_{Q} W(u) )^{\frac{1}{2}} $ scales like a $ L^2 $-norm near the constant states $ \pm 1 $. This mismatch is resolved below by using a variant of Moser iteration.

\subsection{Relative Error Estimates on Subcubes} In this section, we sketch the main ingredient for the proof of Theorem \ref{T: elliptic term cube decomposition}, a slightly quantified homogenization result on the subcubes of the cube decomposition \eqref{eqn:partition-of-small-cubes}.

In the rest of this section we fix three constants: the length scale $ r_c $ coming from Proposition \ref{prop:convexity} and the energetic and H\"{o}lder bounds $ C_{ \text{Energy} } $ and $C_{ \text{H\"{o}lder} }$ from Lemma \ref{prop:assumptions-homogenization-interior-qubes}. Since these constants are fixed once and for all, in what follows, we suppress the dependence of further constants upon them.

Our goal for the remainder of this section is to prove the following qualitative homogenization result that will be used to derive Theorem \ref{T: elliptic term cube decomposition}.

\begin{prop}\label{prop:homogenization-interior-qubes}
	Let $ Q $ be a cube of side length $ r_c $ and suppose $ u \in H^1(Q) $ is a minimizer of $ \mathscr{\F}_{\gamma} $ subject to its own boundary conditions, which satisfies \eqref{eqn:hom-int-cubes-assumption} and $ u \in W^{1,p}(Q) $ for some $ p > 2 $. Denote by $ \bar u \in H^1(Q) $ a minimizer of $  \overline{\F}( \cdot , Q ) $ on $ u + H^1_0(Q) $. Then for every $ \alpha > 0 $ there exists a $ \delta > 0 $ such that if
	$$
	| \bar { a } (Q)- \bar a | 
	+ \left( \fint_{Q} | ( \phi_{Q}^{\gamma}, \sigma_{Q}^{\gamma} ) - \fint_{Q} ( \phi_{Q}^{\gamma}, \sigma_{Q}^{\gamma} ) |^2  \right)^{ \frac{1}{2} }+ \| \theta^{\gamma} - \bar \theta \|_{ H^{-1}(Q) } < \delta,
	$$ 
then
	$$
	\begin{aligned}
	| \overline{\mathscr{F}}(\bar u; Q) - \F_{\gamma}(u; Q) | 
	\leq \alpha \left( \F_{\gamma}(u; Q) + \left( \int_{Q} | \nabla u |^{p} \right)^{\frac{2}{p}} \right).
	\end{aligned}
	$$
\end{prop}	

Before we sketch the proof of Proposition \eqref{prop:homogenization-interior-qubes}, let us recall from Lemma \ref{lemma:meyers-estimate} that there is no loss in assuming $ u \in W^{1,p}(Q) $ as long as $ p $ is sufficiently close to two.

The proof of Proposition \ref{prop:homogenization-interior-qubes} is based on PDE estimates for the cut-off two-scale expansion that are summarized in the next lemma.

\begin{lemma}\label{lemma:h1-estimate-cut-off-two-scale-expansion}
	Let $ u $ and $ \bar u $ be as in Proposition \ref{prop:homogenization-interior-qubes}. Furthermore let $ \eta \in C^{\infty}(\R^d) $ be a smooth cut-off function with support $ \supp{\eta} $ compactly contained in $ Q $, $ 0 \le \eta \le 1 $ and fix some exponent $ p > 2 $ sufficiently close to two. For every $ \alpha > 0 $ there exists a $ \delta > 0 $ such that if $ \fint_{Q} ( \phi_{Q}^{\gamma}, \sigma_{Q}^{\gamma} ) = 0 $ and
	$$
	| \bar { a } (Q) - \bar a | + \left( \fint_{Q} | ( \phi_{Q}^{\gamma}, \sigma_{Q}^{\gamma} ) |^2 \right)^{ \frac{1}{2} } + \| \theta^{\gamma} - \bar \theta \|_{ H^{-1}(Q) } \le \delta,
	$$
	then
	$$
	\begin{aligned}
		&\int_{Q} | \nabla ( u - (1 + \eta \phi_{Q,e_i}^{\gamma} \partial_i) \bar u ) |^2  \\
		&\qquad \leq \alpha \F_{\gamma}(u; Q) + \F_{\gamma}(u; Q) | \bar {a} (Q) - \bar a |^{2} \phantom{\fint} \\
		&\qquad\quad + C_{\eta} \F_{\gamma}(u; Q) \left(  \fint_{Q} | ( \phi_{Q}^{\gamma}, \sigma_{Q}^{\gamma} ) |^2 + \| \theta^{\gamma} - \bar \theta \|_{ H^{-1}(Q) }^2  \right) \\
		&\qquad\quad + C \left( \F_{\gamma}(u; Q) + \left( \int_{Q} | \nabla u |^{p} \right)^{\frac{2}{p}} \right) \left( \int_{Q} (1 - \eta)^{ \frac{2p}{p-2} } \right)^{\frac{p-2}{p}}
	\end{aligned}
	$$
	where $ C $, resp.~$ C_{\eta} $, denotes a constant, which is independent of, resp.~depends on, $ \eta $.\end{lemma}

A proof is given further down in this section. The estimate from Lemma \ref{lemma:h1-estimate-cut-off-two-scale-expansion} is upgraded to a statement about energies by a convexity argument.  Here is where it is convenient to work with the regularized functionals, since the two-scale expansion of $ \bar u $ that appeared in Lemma \ref{lemma:h1-estimate-cut-off-two-scale-expansion} does not necessarily have values in between $ - 1 $ and $ 1 $. Indeed, notice that Proposition \ref{prop:convexity} implies that
\begin{equation}\label{eqn:ee-diff-convexity}
| \F^{\text{reg}}_{\gamma}(u; Q) - \F^{\text{reg}}_{\gamma}(u^*; Q) |
\le \langle D\F^{\text{reg}}_{\gamma}(u; Q) - D\F^{\text{reg}}_{\gamma}(u^* ; Q), u - u^* \rangle
\end{equation}
whenever $ u $ and $ u^* $ share the same boundary conditions and $ u $ is a critical point. Before we make this precise, let us continue to explain the proof strategy.

Lemma \ref{lemma:h1-estimate-cut-off-two-scale-expansion} is reminiscent of \cite[Proposition 2.1, (18)]{JosienOtto}. As therein its proof is based on PDE estimates for the Euler-Lagrange equation that we collect in Lemmas \ref{lemma:homogeneous-calderon-zygmund-estiamte} and \ref{lemma:homogeneous-interior-schauder-estimate}. To employ them it is worth noting that the right hand sides are of the same order as the one in Proposition \ref{prop:homogenization-interior-qubes}.

Homogenization will enter the proof of Lemma \ref{lemma:h1-estimate-cut-off-two-scale-expansion} in a rather soft way via the following lemma, which is inspired by \cite[Proposition 2.1, (17)]{JosienOtto}.

\begin{lemma}\label{lemma:l2-homogenization-estimate}	
	Let $ u $ and $ \bar u $ be as in Proposition \ref{prop:homogenization-interior-qubes}. For every $ \alpha > 0 $, there exists a $ \delta > 0 $ such that if
	\begin{align}\label{eqn:l2-hom-cond}
	| \bar { a } (Q) - \bar a |
	+ \left( \fint_{Q} | (\phi_{Q}^{\gamma}, \sigma_{Q}^{\gamma}) - \fint_{Q} ( \phi_{Q}^{\gamma}, \sigma_{Q}^{\gamma} )  |^2 \right)^{ \frac{1}{2} }
	+ \| \theta^{\gamma} - \theta \|_{ H^{-1}(Q) } \le \delta,
	\end{align}
	then
	\begin{equation}\label{eqn:l2-hom}
	\int_{Q} | u - \bar u |^2 \leq \alpha \mathscr{F}_{ \gamma } (u; Q).
	\end{equation}
\end{lemma}

The above lemma may be seen as a weak form of Lemma \ref{lemma:h1-estimate-cut-off-two-scale-expansion}; rather then comparing $ u $ and (the two-scale expansion of) $ \bar u $ in the strong $ H^1 $-norm, we begin with the weaker $ L^2 $-norm. Its proof, as presented in Section \ref{sec:proofs-deterministic-homogenization}, is based on a compactness argument. At this point we rely on the local energy bound \eqref{eqn:hom-int-cubes-assumption}.

We will continue by first proving Theorem \ref{T: elliptic term cube decomposition} based on the homogenization result in Proposition \ref{prop:homogenization-interior-qubes}, see Section \ref{section: T: elliptic term cube decomposition}. Next, we prove the homogenization statements above, see Section \ref{sec:proofs-deterministic-homogenization}. Finally, we collect the proofs of our PDE estimates in Section \ref{section:pde-estimates}. 

\subsection{Proof of Theorem \ref{T: elliptic term cube decomposition}}\label{section: T: elliptic term cube decomposition} Taking for granted Proposition \ref{prop:homogenization-interior-qubes} and Lemma \ref{lemma:meyers-estimate} for now, here is the proof of the relative energy error estimate, Theorem \ref{T: elliptic term cube decomposition}.

	\begin{proof}[Proof of Theorem \ref{T: elliptic term cube decomposition}] The theorem follows directly from Proposition \ref{prop:homogenization-interior-qubes} applied to the subcubes in \eqref{eqn:partition-of-small-cubes}, after accounting for the rescaling.  Toward that end, recall the definitions of the rescaled medium $(a_{Q}^{\gamma},\theta_{Q}^{\gamma})$, see \eqref{E: rescaled fields},  and correctors $(\phi_{Q}^{\gamma},\sigma_{Q}^{\gamma})$, see \eqref{E: rescaled correctors}. Observe that, given a cube $Q = Q_{r_{c}}(z)$ of radius $r_{c}$ centered at some point $z$, we can write
		\begin{align*}
			&\left( \fint_{Q_{r_{c}}(z)} |(\phi_{Q_{r_{c}}(z)}^{\gamma},\sigma_{Q_{r_{c}}(z)}^{\gamma}) - \fint_{Q_{r_{c}}(z)} (\phi_{Q_{r_{c}}(z)}^{\gamma},\sigma_{Q_{r_{c}}(z)}^{\gamma})|^{2} \right)^{\frac{1}{2}} \\
			&\quad = \gamma \left( \fint_{Q_{r_{c} \gamma^{-1}}(\gamma^{-1}z)} |(\phi,\sigma)_{Q_{r_{c} \gamma^{-1}}(\gamma^{-1}z)} - \fint_{Q_{r_{c} \gamma^{-1}}(\gamma^{-1} z)} (\phi,\sigma)_{Q_{r_{c} \gamma^{-1}}(\gamma^{-1}z)}|^{2} \right)^{\frac{1}{2}} \\
			&\quad \leq  r_c { \rm Sub } (Q_{r_{c}\gamma^{-1}}(r_{c}z))
		\end{align*} 
and, similarly, by \eqref{E: H minus one norm},
		\begin{equation*}
			\|\theta^{\gamma} - \bar{\theta}\|_{H^{-1}(Q_{r_{c}}(z))} = \gamma \|\theta - \bar{\theta}\|_{H^{-1}(Q_{r_{c} \gamma^{-1}}(r_{c} \gamma^{-1}z))} \leq r_c { \rm Osc } (Q_{r_{c}\gamma^{-1}}(r_{c}z)).
		\end{equation*}
	
	Let $u \in H^{1}(Q_{R})$ be a minimizer of \eqref{eqn: T: elliptic term cube decomposition variation problem}.  Given a subcube $Q = Q_{r_{c}}(r_{c} z)$ for some $z \in \mathbb{Z}^{d} \cap [-K,K]$, let $\bar{u}_{Q}$ be the minimizer of $\overline{\mathscr{F}}(\cdot; Q)$ in $u + H^{1}_{0}(Q)$.  We define $\tilde{u} \in H^{1}(Q_{R})$ by gluing together these functions, that is,
		\begin{align*}
			\tilde{u} = \sum_{z \in \mathbb{Z}^{d} \cap [-K,K]} \bar{u}_{Q_{r_{c}}(r_{c} z)} \boldsymbol{1}_{Q_{r_{c}}(r_{c} z)}.
		\end{align*}
	
	Now, by Proposition \ref{prop:homogenization-interior-qubes}, there is a modulus of continuity $\omega : [0,\infty) \to [0,\infty)$ such that, for any $z \in \mathbb{Z}^{d} \cap [-K,K]^{d}$,
		\begin{align*}
			&| \overline{\mathscr{F}}(\tilde{u}; Q_{r_{c}}(r_{c} z)) - \F_{\gamma}(u; Q_{r_{c}}(r_{c} z)) | \\
			&\quad \leq \omega \! \left( { \rm Sub } (Q_{r_{c}\gamma^{-1}}(r_{c}z)) +{ \rm Osc } (Q_{r_{c}\gamma^{-1}}(r_{c}z)) +  | \bar {a} ( Q_{r_{c} \gamma^{-1} }(r_{c} z) ) - \bar a | \right) \\
			&\quad\qquad \times \left( \F_{\gamma}(u; Q_{r_{c}}(r_{c} z)) + \left(\int_{Q_{r_{c}}(r_{c} z)} |\nabla u|^{p} \right)^{\frac{2}{p}} \right).
		\end{align*}
	Let $\mathcal{E}_{K}(\gamma)$ be the maximum error obtained in this approximation:
		\begin{equation*}
			\mathcal{E}_{K}(\gamma) = \max_{ z \in \mathbb{Z}^{d} \cap [-K,K]^{d} } \omega \left( { \rm Sub } (Q_{r_{c}\gamma^{-1}}(r_{c}z)) + { \rm Osc } (Q_{r_{c}\gamma^{-1}}(r_{c}z)) +  | \bar {a} ( Q_{r_{c} \gamma^{-1}}(r_{c} z) ) - \bar a | \right) .
		\end{equation*}
	After summing over cubes, the previous estimate becomes
		\begin{align*}
			& | \overline{\mathscr{F}}(\tilde{u}; Q_{R}) - \F_{\gamma}(u; Q_{R}) | \\
			&\quad \leq \mathcal{E}_{K}(\gamma)  \left( \F_{\gamma}(u; Q_{R}) + \sum_{ z \in \mathbb{Z}^{d} \cap [-K,K]^{d} } \left(\int_{Q_{r_{c}}(r_{c} z)} |\nabla u|^{ p} \right)^{\frac{2}{p}} \right).
		\end{align*}
	Invoking Lemma \ref{lemma:meyers-estimate}, the additional gradient term can be bounded via 
		\begin{align*}
			 \sum_{ z \in \mathbb{Z}^{d} \cap [-K,K]^{d} } & \left(\int_{Q_{r_{c}}(r_{c} z)} |\nabla u|^{p} \right)^{\frac{2}{p}} \\
			& \lesssim \sum_{ z \in \mathbb{Z}^{d} \cap [-K,K]^{d} } \F_{\gamma}(u; Q_{2r_{c}}(r_{c} z) \cap Q_{R}) + \sum_{ \substack{ z \in \mathbb{Z}^{d} \\ | z |_{ \infty } = K } }  \left( \int_{ Q_{2r_{c}}(r_{c}z) \cap Q_R } |\nabla q|^{p} \right)^{\frac{2}{p}}  \\
			& \lesssim \F_{\gamma}(u;Q_{S}) + R^{ d - 1 } \sup | q'  |^{ 2 } .
		\end{align*}
	In view of the fact that $\theta_{*} \leq \theta$ and $\lambda \text{Id} \leq a$, we have a lower bound
		\begin{align*}
			\F_{\gamma}(u;Q_{R}) &\geq \min \left\{ \frac{\lambda}{2} \int_{Q_{R}} |\nabla v|^{2} + \theta_{*} \int_{Q_{R}} W(v) \, \mid \, v(x) = q(x) \, \, \text{for each} \, \, x \in \partial Q_{R} \right\} \\
					&\gtrsim R^{d - 1},
		\end{align*}
	where in the last line we used $\Gamma$-convergence of the constant-coefficient Allen-Cahn functional to obtain the lower bound (cf.\ \cite[Theorem 3.7]{ansini_braides_chiado-piat}).  Concatenating the previous estimates, we observe that the additional gradient error terms can be absorbed into $\F_{\gamma}(u,Q_{R})$, leading to
		\begin{align*}
			& | \overline{\mathscr{F}}(\tilde{u}; Q_{R}) - \F_{\gamma}(u; Q_{R}) | \lesssim \mathcal{E}_{K}(\gamma)  \F_{\gamma}(u; Q_{R})
		\end{align*}
	Up to multiplying $\omega$ by a constant, this is precisely the desired conclusion.\end{proof}

\subsection{Proof of the Homogenization Statements.}\label{sec:proofs-deterministic-homogenization} 

We are now going to prove Proposition \ref{prop:homogenization-interior-qubes}. We start with the weak approximation statement in Lemma \ref{lemma:l2-homogenization-estimate} and show the upgrade to the $ H^1 $-statement later on.  While the argument is inspired by Proposition 2.1 in \cite{JosienOtto} for the linear problem, the argument here is slightly more elaborate since the PDE is nonlinear.

\begin{proof}[Proof of Lemma \ref{lemma:l2-homogenization-estimate}.]  The following proof is completely deterministic in the sense that we use only the assumptions of Section \ref{S: deterministic assumptions}, most importantly, the decomposition \eqref{eqn:a-helmholtz-decomposition}, and the growth condition \eqref{eqn:l2-hom-cond} as an input to obtain the error estimate \eqref{eqn:l2-hom}. In particular, as throughout this section, no probabilistic arguments are involved here.

\textit{Step 1 (Contradictive assumption).} We prove the error estimate \eqref{eqn:l2-hom} by contradiction. To this end, let us notice that by translation and dilation invariance of the statement, we may restrict our contradictive proof to a fixed cube $ Q $. Therefore, it is enough to suppose that $\{(a_{k},\theta_{k})\}_{k \in \mathbb{N}}$ are sequences of coefficients satisfying the assumptions of Section \ref{S: deterministic assumptions}.  Let $\{(\bar{a}_{k}(Q),\bar{a}_{k}, \bar{\theta}_{k})\}_{k \in \mathbb{N}}$ and $\{(\phi_{Q,k},\sigma_{Q,k})\}_{k \in \mathbb{N}}$ be the associated constants and correctors.  As in the assumptions of the present lemma, suppose that $\{(u_k, \bar u_k)\}_{k \in \mathbb{N}}$ are functions in $H^{1}(Q)$ such that $ u_k $ satisfies \eqref{eqn:hom-int-cubes-assumption} and minimizes $ \mathscr{F}_k(v ; Q) \coloneqq \frac{1}{2} \int_{Q} a_k \nabla v \cdot \nabla v + \int_{Q} \theta_{k} W(v) $ on $ u_k + H^1_0(Q) $ (i.e., subject to its own boundary conditions), and $ \bar u_k $ minimizes $ \overline{\mathscr{F}}_{k} (v ; Q) \coloneqq \frac{1}{2} \int_{Q} \bar a_k \nabla v \cdot \nabla v + \int_{Q} \bar{ \theta }_{k} W(v) $ on $ u_k + H^1_0(Q) $. In particular, $ (u_k, \bar u_k ) $ solves
\begin{align*}
- \nabla \cdot a_k \nabla u_k + \theta_k W'(u_k) & = 0 = - \nabla \cdot \bar a_k \nabla \bar u_k + \bar\theta_{k} W'( \bar u_k) & \quad { \rm in } ~ Q \\
 u_k &= \bar u_k & \quad { \rm on } ~ \partial Q,
\end{align*}
To see that the error estimate \eqref{eqn:l2-hom} must hold, it suffices to establish that if
\begin{align}\label{eqn:l2-hom-contra-assumption-1}
| \bar a_{ k } (Q) - \bar a_k | + \left( \fint_{Q} | (\phi_{Q,k}, \sigma_{Q,k}) - \fint_{Q} (\phi_{Q,k}, \sigma_{Q,k}) |^2\right)^{\frac{1}{2}}
+ \| \theta_k - \bar \theta \|_{ H^{-1}(Q) } \rightarrow 0
\end{align}
as $ k \rightarrow \infty $, then
\begin{align*}
	\lim_{k \to \infty} \F_{k}(u_{k} ; Q)^{-1} \int_{Q} | u_k - \bar u_k |^2 = 0.
\end{align*}
Toward that end, we argue by contradiction: Specifically, we fix an $\alpha^{*} > 0$ and assume that (after passing to another subsequence)
	\begin{align} \label{eqn:l2-hom-contra-assumption-2}
		\int_{Q} | u_k - \bar u_k |^2 \geq \alpha^{*} \F_{k}(u_{k} ; Q)  \quad \text{for each} \quad k \in \mathbb{N}.
	\end{align}
Below we show that this leads to a contradiction.  Before we do this, let us collect some properties of $ (u_k, \bar u_k) $ that follow from the above assumptions.

First, by \eqref{eqn:hom-int-cubes-assumption} and Proposition \ref{P: uniqueness comparison}, we have
$$
\sup_{ k \in \N } \int_{Q} ( u_k^2 + \bar u_k^2 + | \nabla u_k |^2 + | \nabla \bar u_k |^2 ) < \infty.
$$
Thus, up to passing to a subsequence, we can assume that there are functions $u, \bar{u} \in H^{1}(Q)$ such that $\bar{u} - u \in H^{1}_{0}(Q)$ and
$$
\nabla u_k \rightharpoonup \nabla u,
\quad \nabla \bar u_k \rightharpoonup \nabla \bar u,
\quad u_k \rightarrow u,
\quad \bar u_k \rightarrow \bar u
\quad \text{in} ~ L^2(Q)
$$
as $ k \rightarrow \infty $. Due to the uniform boundedness assumption on the constants $ ( \bar{a}_k, \bar{\theta}_{k} ) $ from Section \ref{S: deterministic assumptions}, we may select a further subsequence to ensure that
$$
\bar a_k \rightarrow \bar a 
\quad { \rm and } \quad
\bar \theta_k \rightarrow \bar \theta
$$
as $ k \rightarrow \infty $ for some constant matrix $ \bar a $ and positive number $ \bar \theta $. The last two assertions are enough to show that both $ u $ and $ \bar u $ solve
\begin{equation}\label{eqn:l2-hom-pf06}
- \nabla \cdot \bar a \nabla \bar v +  \bar\theta W'(\bar{v}) = 0 \quad \text{in} ~ Q
\end{equation}
with boundary conditions $ \bar v = u = \bar u $ on $ \partial Q $; see the next paragraph for the details.

Here comes the argument for \eqref{eqn:l2-hom-pf06}. For $ \bar u $, the statement almost immediately follows from the above convergence statements. As in \cite{JosienOtto}, one can show that
$$
\lim_{ k \rightarrow \infty } 
\int_{Q} | (\phi_{Q,k}, \sigma_{Q,k}) - \fint_{Q} (\phi_{Q,k}, \sigma_{Q,k}) |^2 = 0
~~ \Rightarrow ~~
a_k \nabla u_k - \bar {a}_{k} (Q) \nabla u_k \rightharpoonup 0
\quad \text{in} ~ L^2(Q).
$$
as $ k \rightarrow \infty $. Since we arranged everything such that $ \bar a_k \rightarrow \bar a $, we know by assumption \eqref{eqn:l2-hom-contra-assumption-1} that also $ \bar a_{k} (Q) \rightarrow \bar a $. This implies
\begin{equation}\label{eqn:l2-hom-pf06a}
a_k \nabla u_k = \bar a_k (Q) \nabla u_k + ( a_k - \bar {a}_{k} (Q) ) \nabla u_k  \rightharpoonup \bar a \nabla u
\quad \text{in} ~ L^2(Q).
\end{equation}
Furthermore, we may write the well terms as
$$
\theta_k W'(u_k) - \bar \theta W' ( u )
= ( \theta_k - \bar \theta ) W'( u )
+ \theta_k ( W'(u_k) - W'( u ) ).
$$
Since $ W'(u) \in H^1(Q) $ and $ \theta_* \leq \theta_k \leq \theta^* $, this implies
\begin{equation}\label{eqn:l2-hom-pf06b}
\theta_k W'(u_k) - \bar \theta W' ( u )
\rightharpoonup 0
\quad { \rm in } ~ L^{2}(Q).
\end{equation}
Together \eqref{eqn:l2-hom-pf06a} and \eqref{eqn:l2-hom-pf06b} imply that $u$ satisfies \eqref{eqn:l2-hom-pf06} distributionally.

In the remainder of the proof, we let $\delta^{*} > 0$ be a small positive constant to be determined below and consider separately two cases: (i) $\mathscr{F}_{k}(u_{k}; Q) \geq \delta^{*}$ for each $k$ and (ii) $\mathscr{F}_{k}(u_{k}; Q) \leq \delta^{*}$ for each $k$.  (Up to passing to yet another subsequence, these two cases are exhaustive.)  We refer to these two cases as the \emph{large} and \emph{small energy regimes}, respectively.

\textit{Step 2 (Large energy regime).} Assume that $\mathscr{F}_{k}(u_{k};Q) \geq \delta^{*}$ for each $k$.  Since $ -1 \leq u, \bar u \leq 1 $ on $ Q $, which is a cube of side length $r_{c}$, Proposition \ref{P: uniqueness comparison} implies that equation \eqref{eqn:l2-hom-pf06} has at most one solution. Therefore,
$$
0 = \int_{Q} | u - \bar u |^2
= \lim_{k \rightarrow \infty} \int_{Q} | u_k - \bar u_k |^2
\geq \alpha^* \inf_{ k \in \N } \mathscr{F}_{ k } ( u_k ; Q )
\geq \alpha^* \delta^* > 0,
$$
which yields the desired contradiction in the large energy regime.

\textit{Step 3 (Small energy regime).} It only remains to consider the case when we have $\mathscr{F}_{k}(u_{k};Q) \leq \delta^{*}$ for each $k$.  Here we use the clearing-out lemma to deduce that $u_{k}$ and $\bar{u}_{k}$ are uniformly close to $1$ or $-1$, and then the regularity assumptions on $W$ allow us to pass to a linearized equation.

For technical reasons, we need to separately consider the cases when $\kappa = 1$ ($W$ is approximately quadratic near its minima) and $\kappa > 1$ ($W$ is superquadratic near its minima).

\textit{Step 3.1 (Linearization for locally quadratic wells).} Assume that the parameter $\kappa$ of Section \ref{S: assumptions} is equal to one.  Since $\mathscr{F}_{k}(u_{k}; Q) \leq \delta^{*}$ for each $k$, we apply the clearing-out lemma (Lemma \ref{lemma:clearing-out}) to deduce that $ u_k $ is close to either of the minima of $ W $,~w.l.o.g.
\begin{equation}\label{eqn:l2-hom-pf05}
\sup_{ k \in \N }  \| u_k - 1 \|_{ L^{ \infty }(Q) }  \leq \varepsilon^*,
\end{equation}
where $ \varepsilon^* \rightarrow 0 $ as $ \delta^* \rightarrow 0 $. In view of Proposition \ref{P: uniqueness comparison}, the energy of the functions $(\bar{u}_{k})_{k \in \mathbb{N}}$ is also uniformly small, hence, up to decreasing $\delta^{*}$ by a constant factor, we also have $\|\bar{u}_{k} - 1\|_{L^{\infty}(Q)} \leq \varepsilon^{*}$ for each $k$. (Recall that $ u_k $ and $ \bar u_k $ share the same boundary conditions. Hence $ \bar u_k $ has to be close to the same minimum of $ W $ as $ u_k $, see \eqref{eqn:l2-hom-pf05}.) We claim that this leads to a contradiction.

To see this, we begin by defining the functions $\{R_{k}\}_{k \in \mathbb{N}}$ by
$$
R_k \coloneqq \frac{ W'(u_k) - W'(1) - W''(1) ( u_k - 1 ) }{  u_k - 1 }
=  \frac{ W'(u_k) - W''(1) ( u_k - 1 ) }{  u_k - 1 }
$$
and $ \{\bar R_k\}_{k \in \mathbb{N}}$ analogously, so that
\begin{equation}\label{eqn:l2-hom-pf03}
\| R_k \|_{ L^{\infty }(Q) } + \| \bar R_k \|_{ L^{\infty }(Q) }  \leq \omega( \varepsilon^* ),
\end{equation}
where $ \omega : [0,\infty) \to [0,\infty)$ is a modulus of continuity determined by $W$. Now let $ w_k = \F_{ k }(u_k ; Q)^{-\frac{1}{2}} ( 1 - u_k ) $ and $ \bar w_k = \F_{ k }(u_k ; Q)^{-\frac{1}{2}} ( 1 - \bar u_k ) $, with the same normalization on both terms. Note that
$$
- \nabla \cdot a_k \nabla w_k + \theta_k W''(1) w_k
= f_k ,
\quad
- \nabla \cdot \bar a_k \nabla \bar w_k + \bar \theta_k W''(1) \bar w_k
= \bar f_k
\quad \text{in} ~ Q
$$
where $ f_k  = - \theta_k R_k w_k $, $ \bar f_k =  - \bar \theta_k \bar R_k \bar w_k $. Moreover, the lower bound \eqref{eqn:l2-hom-contra-assumption-2} becomes
\begin{equation}\label{eqn:l2-hom-pf02b}
\int_{Q} | w_k - \bar w_k |^2 \ge  \alpha^* > 0.
\end{equation}

Since in this step of the proof we restrict ourselves to $\kappa = 1$ in \eqref{eqn:w-non-degeneracy}, it is straight forward to see by using \eqref{eqn:w-poly} and \eqref{eqn:l2-hom-pf05} that, for sufficiently small $ \varepsilon^{*}$, $ W $ looks like a quadratic so that from assumption \eqref{eqn:hom-int-cubes-assumption} on $ u_k $ and the minimiality of $ \bar u_k $, cf. Proposition \ref{P: uniqueness comparison}, we obtain
\begin{equation}\label{eqn:l2-hom-pf04}
\sup_{k \in \N} \int_{Q} ( | \nabla w_k |^2 + | \nabla \bar w_k |^2 + w_k^2 + \bar w_k^2 ) \lesssim 1 < \infty.
\end{equation}
Hence we may as well assume that there are functions $w, \bar{w} \in H^{1}(Q)$ such that
$$
\nabla w_k \rightharpoonup \nabla w,
\quad \nabla \bar w_k \rightharpoonup \nabla \bar w,
\quad w_k \rightarrow w,
\quad \bar w_k \rightarrow \bar w
\quad \text{in} ~ L^2(Q)
$$
as $ k \rightarrow \infty $. Furthermore, due to \eqref{eqn:l2-hom-pf03}, we can similarly assume there are $f,\bar{f} \in L^{2}(Q)$ such that
$$
f_k \rightharpoonup f,
\quad \bar f_k \rightharpoonup \bar f
\quad \text{in} ~ L^2(Q).
$$
Combining these convergence statements with the arguments from Step 1, we obtain the limiting equation
$$
- \nabla \cdot \bar a \nabla (w - \bar w) + \bar \theta W''(1) (w - \bar w)
= \bar f - f
\quad \text{in} ~ Q.
$$
Furthermore, by \eqref{eqn:l2-hom-pf03}, \eqref{eqn:l2-hom-pf04}, and weak convergence,
$$
\int_{Q} ( | f  |^2 + | \bar f |^2) \leq \liminf_{ k \rightarrow \infty } \int_{Q} ( | f_k  |^2 + | \bar f_k |^2)  \lesssim \omega ( \varepsilon^* )^2.
$$
Since the constant $\kappa$ in Section \ref{S: assumptions} is equal to one by assumption, we know that $W''(1) > 0$.  Thus, if we test the equation above with $w - \bar{w}$ and invoke \eqref{eqn:l2-hom-pf05}, we find
$$
\lambda \int_{Q} |\nabla w - \nabla \bar{w}|^{2} + \frac{W''(1)}{2} \int_{Q} | w - \bar w |^2
\lesssim  \omega ( \varepsilon^* )^2,
$$
a contradiction to \eqref{eqn:l2-hom-pf02b} as soon as $ \delta^* $ (hence also $ \varepsilon^* $) is sufficiently small.

\textit{Step 3.2 (Linearization for locally superquadratic wells.)} We now sketch the necessary modifications for the last argument from \eqref{eqn:l2-hom-pf03} onwards, for the case $ \kappa > 1 $ in \eqref{eqn:w-non-degeneracy}. As in the previous step, we define $ ( R_k ) $ so that the following identity holds
$$
R_k(u_k - 1) = W'(u_k) - W''(1) (u_k - 1) = W'(u_k).
$$
By assumption \eqref{eqn:w-non-degeneracy},
$$
| R_k | \lesssim (u_k - 1)^{ 2\kappa - 2 }.
$$
Defining $ w_k = \F_{ k }(u_k ; Q)^{-\frac{1}{2}} ( 1 - u_k ) $ and $ \bar w_k = \F_{ k }(u_k ; Q)^{-\frac{1}{2}} ( 1 - \bar u_k ) $ as above, they become almost $ a_k $-harmonic, resp.~ $ \bar a_k $-harmonic, in the sense that
$$
- \nabla \cdot a_k \nabla w_k = - \theta_k R_k w_k,
\quad - \nabla \cdot \bar a_k \nabla \bar w_k = - \bar\theta_k \bar R_k \bar w_k
\quad \text{in} ~ Q.
$$
Let us now write $ f_k = - \theta_k R_k w_k $, resp.~$ \bar f_k = - \bar\theta_k \bar R_k \bar w_k $. In view of \eqref{eqn:w-poly}, as soon as $ \varepsilon^* $ is small enough,
\begin{align} 
	\int_{Q} |u_{k} - 1|^{2\kappa - 2} w_{k}^{2} \lesssim \mathscr{F}_{ k } (u_{k} ; Q)^{-1} \int_{Q} W(u_{k}) \leq 1, \nonumber \\
	\int_{Q} |\nabla w_{k}|^{2} \lesssim \mathscr{F}_{ k } (u_{k} ; Q)^{-1} \int_{Q} |\nabla u_{k}|^{2} \leq 1. \label{E: slightly annoying gradient bound}
\end{align}
Using the first of the above equations, we can estimate
$$
\int_{Q} | f_k |^2 
\lesssim \int_{Q} | u_k - 1 |^{2 ( 2\kappa - 2 ) } w_k^2
\lesssim \| u_k - 1 \|_{ L^{\infty}(Q) } ^{2\kappa - 2}.
$$
The same computation on $ \bar w_k $ together with \eqref{eqn:l2-hom-pf05} yields
\begin{equation}\label{eqn:l2-hom-pf07}
\sup_{k \in \N} \int_{Q} ( | f_k |^2 + | \bar f_k |^2 )
\lesssim \| u_k - 1 \|_{ L^{\infty}(Q) } ^{2\kappa - 2} + \| \bar{u}_{k} - 1\|_{L^{\infty}(Q)}^{2 \kappa - 2}
\lesssim 1.
\end{equation}
Therefore, by compactness, we may assume that
$$
f_k \rightharpoonup f,
\quad \bar f_k \rightharpoonup \bar f
\quad \text{in} ~ L^2(Q)
$$
along a subsequence as $ k \rightarrow \infty $.

Replacing $w_{k}$ and $\bar{w}_{k}$ by $ v_k \coloneqq w_{k} - \fint_{Q} w_{k}$ and $ \bar v_k \coloneqq \bar{w}_{k} - \fint_{Q} w_{k}$ (the same average on both terms), the relation \eqref{eqn:l2-hom-pf02b} becomes
\begin{align}\label{eqn:l2-hom-pf02c}
\int_{Q} | v_k - \bar v_k |^2 \ge  \alpha^* > 0,
\end{align}
and the equations stay the same,~i.e.,
$$
- \nabla \cdot a_k \nabla v_k = f_k,
\quad - \nabla \cdot \bar a_k \nabla \bar v_k = \bar f_k
\quad \text{in} ~ Q.
$$
Note that $ v_k - \bar v_k \in H^1_0(Q) $ since we subtracted the same constant.

First, since by definition we have $\int_{Q} v_{k} = 0$, the estimate \eqref{E: slightly annoying gradient bound} and Poincaré together yield
$$
\int_{Q} v_{k}^{2}  \lesssim \int_{Q} |\nabla v_{k}|^{2} \lesssim 1.
$$
Next, after subtracting the equations for $w_{k}$ and $\bar{w}_{k}$, we have
$$
- \nabla \bar a_k \nabla ( v_k - \bar v_k ) = f_k - \bar f_k - \nabla \cdot ( \bar a_k - a_k ) \nabla v_k
\quad \text{in} ~ Q
$$
so that, by the energy estimate on this equation (recall $v_{k} - \bar{v}_{k} \in H^{1}_{0}(Q)$),
$$
\sup_{ k \in \N } \int_{Q} ( | \nabla v_k |^2 + | \nabla \bar v_k |^2 + v_k^2 + \bar v_k^2 ) \lesssim 1.
$$
From all this, we may conclude that
$$
\nabla v_k \rightharpoonup \nabla v,
\quad \nabla \bar v_k \rightharpoonup \nabla \bar v,
\quad v_k \rightarrow v,
\quad \bar v_k \rightarrow \bar v
\quad \text{in} ~ L^2(Q)
$$
along a subsequence as $ k \rightarrow \infty $. All this is enough to pass to the limiting equations, i.e.,
$$
- \nabla \cdot \bar a \nabla v = f,
\quad - \nabla \cdot \bar a \nabla \bar v = \bar f
\quad \text{in} ~ Q,
$$
the same conclusion as we reached in the previous steps of the proof. In particular, we can conclude that
$$
- \nabla \cdot \bar a ( \nabla v - \nabla \bar v ) = f - \bar f
\quad \text{in} ~ Q.
$$
The energy estimate (and Poincaré inequality) for this equation, together with \eqref{eqn:l2-hom-pf07}, yields
$$
\int_{Q} | v - \bar v |^2
\lesssim \int_{Q} ( | f |^2 + | \bar f |^2 )
\lesssim \liminf_{k \rightarrow \infty}  \int_{Q} ( | f_k |^2 + | \bar f_k |^2 )
\lesssim \sup_{k \in \N} \| u_k - 1 \|_{ L^{\infty}(Q) } ^{2\kappa - 2},
$$
which, in view of \eqref{eqn:l2-hom-pf05}, contradicts \eqref{eqn:l2-hom-pf02c}.
\end{proof}

We now give the proof of Lemma \ref{lemma:h1-estimate-cut-off-two-scale-expansion}. The strategy is quite standard in the literature, see, for example, Proposition 2.1 in \cite{JosienOtto}. Therein a local estimate is proved for the gradient of the two-scale expansion on some small ball relative to the $ H^1 $ norm on a larger ball. It turns out that for us it is more natural to introduce the cut-off directly in the two-scale expansion rather than in the estimate; this leaves the boundary datum unchanged.

\begin{proof}[Proof of Lemma \ref{lemma:h1-estimate-cut-off-two-scale-expansion}.]
We start with some useful identities that culminate in what is known as the intertwining property of the (cut-off) two-scale expansion,~cf.~\eqref{eq:cut-off-two-scale-intertwining-property}. As for the standard two-scale expansion we compute
$$
\nabla (1 + \eta \phi_{Q,e_{i}}^{\gamma} \partial_i) \bar u
= (e_i + \nabla \phi_{Q,e_{i}}^{\gamma}) \eta \partial_i \bar u
+ \phi_{Q,e_{i}}^{\gamma} \nabla ( \eta \partial_i \bar u )
+ (1 - \eta) \nabla \bar u
$$
and
$$
\begin{aligned}
& a^{\gamma} \nabla (1 + \eta \phi_{Q,e_{i}}^{\gamma} \partial_i) \bar u \\
&\quad = a^{\gamma} (e_i + \nabla \phi_{Q,e_{i}}^{\gamma}) \eta \partial_i \bar u
+ \phi_{Q,e_{i}}^{\gamma} a^{\gamma} \nabla ( \eta \partial_i \bar u )
+ (1 - \eta) a^{\gamma} \nabla \bar u.
\end{aligned}
$$
Invoking the Helmholtz-type decomposition \eqref{E: helmholtz}, we rewrite this in the form
\begin{equation}\label{eq:cut-off-two-scale-flux}
\begin{aligned}
& a^{\gamma} \nabla (1 + \eta \phi_{Q,e_{i}}^{\gamma} \partial_i) \bar u \\
&\quad = \bar { a } (Q) \nabla \bar u
+ ( \nabla \cdot \sigma_{Q,e_{i}}^{\gamma} ) \eta \partial_i \bar u 
+ \phi_{Q,e_{i}}^{\gamma} a^{\gamma} \nabla ( \eta \partial_i \bar u )
+ (1 - \eta) ( a^{\gamma} - \bar { a } (Q ) ) \nabla \bar u \\
&\quad = \bar { a } (Q) \nabla \bar u
+ ( \phi_{Q,e_{i}}^{\gamma} a^{\gamma}  - \sigma_{Q,e_{i}}^{\gamma} ) \nabla ( \eta \partial_i \bar u )
+ (1 - \eta) ( a^{\gamma} - \bar { a } (Q) ) \nabla \bar u
+ \nabla \cdot \eta (\partial_i \bar u) \sigma_{Q,e_{i}}^{\gamma}.
\end{aligned}
\end{equation}
Note that the last term in the above computation is divergence free, so that
\begin{equation}\label{eq:cut-off-two-scale-intertwining-property}
\begin{aligned}
& \nabla \cdot a^{\gamma} \nabla (1 + \eta \phi_{Q,e_{i}}^{\gamma} \partial_i) \bar u 
= \nabla \cdot \bar a \nabla \bar u \\
& \quad + \nabla \cdot ( \phi_{Q,e_{i}}^{\gamma} a^{\gamma}  - \sigma_{Q,e_{i}}^{\gamma} ) \nabla ( \eta \partial_i \bar u )
+ \nabla \cdot (1 - \eta) ( a^{\gamma} - \bar {a} (Q) ) \nabla \bar u
+ \nabla \cdot ( \bar { a } (Q) - \bar a ) \nabla \bar u.
\end{aligned}
\end{equation}
Hence $ u - (1 + \eta \phi_{Q,e_{i}}^{\gamma} \partial_i) \bar u $ solves the equation
\begin{equation}\label{eq:cut-off-two-scale-difference-equation}
\begin{aligned}
& - \nabla \cdot a^{\gamma} \nabla ( u - (1 + \eta \phi_{Q,e_{i}}^{\gamma} \partial_i) \bar u ) + \theta^{\gamma} W'(u) - \bar \theta W'(\bar u) \\
& \quad =
\nabla \cdot ( \phi_{Q,e_{i}}^{\gamma} a^{\gamma} - \sigma_{Q,e_{i}}^{\gamma} ) \nabla ( \eta \partial_i \bar u )
+ \nabla \cdot (1 - \eta) ( a^{\gamma} - \bar {a} (Q) ) \nabla \bar u \\
&\quad \qquad
+ \nabla \cdot ( \bar {a} (Q) - \bar a ) \nabla \bar u .
\end{aligned}
\end{equation}
Having the identity $ \theta^{\gamma} W'(u) - \bar \theta W'(\bar u) = ( \theta^{ \gamma} - \bar \theta ) W'( \bar u ) + \theta^{ \gamma} ( W'(u) - W'( \bar u ) ) $ in mind, one can see that the energy estimate for \eqref{eq:cut-off-two-scale-difference-equation} implies
$$
\begin{aligned}
& \int_{Q} | \nabla ( u - (1 + \eta \phi_{Q,e_{i}}^{\gamma} \partial_i) \bar u ) |^2 \\
& \qquad \lesssim \int_{Q} ( W'(u) - W'(\bar u) )^2 + \int_{Q} ( \theta^{\gamma} - \bar \theta ) W'( \bar u )  ( u - (1 + \eta \phi_{Q,e_{i}}^{\gamma} \partial_i) \bar u ) \\
& \quad \qquad + \int_{Q} | ( \phi_{Q,e_{i}}^{\gamma} a^{\gamma}  - \sigma_{Q,e_{i}}^{\gamma} ) \nabla ( \eta \partial_i \bar u ) |^2
+ \int_{Q} | (1 - \eta) ( a^{\gamma} - \bar {a} (Q) ) \nabla \bar u |^2 \\
&\quad \qquad
+ | \bar {a} (Q) - \bar a |^{ 2 } \int_{Q} | \nabla \bar u |^2 
\end{aligned}
$$
so that, by our assumptions in Section \ref{S: assumptions}, we can estimate
\begin{equation}\label{eqn:proof-h1-estimate-cut-off-two-scale-expansion-1}
\begin{aligned}
& \int_{Q} | \nabla ( u - (1 + \eta \phi_{Q,e_{i}}^{\gamma} \partial_i) \bar u ) |^2 \\
& \qquad \lesssim \|W''\|^{2}_{L^{\infty}([-1,1])} \int_{Q} ( u - \bar u )^2 + \int_{Q} ( \theta^{\gamma} - \bar \theta ) W'( \bar u )  ( u - (1 + \eta \phi_{Q,e_{i}}^{\gamma} \partial_i) \bar u )  \\
& \quad \qquad + \sup_{Q} | \nabla ( \eta \partial_i \bar u ) |^2 \int_{Q} | ( \phi_{Q}^{\gamma}, \sigma_{Q}^{\gamma} ) |^2
+ \left( \int_{Q} (1 - \eta)^{2q} \right)^{\frac{1}{q}} \left( \int_{Q} | \nabla \bar u |^{2p} \right)^{\frac{1}{p}} \\
& \quad \qquad + | \bar a_{Q} - \bar a |^{ 2 } \int_{Q} | \nabla \bar u |^2 ,
\end{aligned}
\end{equation}
where $ \frac{1}{p} + \frac{1}{q} = 1 $.

Let us now handle the second term on the r.h.s.. To this end, we split the integral using the cut-off $ \eta $ to obtain
\begin{align*}
&\int_{Q} ( \theta^{\gamma} - \bar \theta ) W'( \bar u )  ( u - (1 + \eta \phi_{Q,e_{i}}^{\gamma} \partial_i) \bar u ) \\
&\quad \lesssim \| \theta^{\gamma} - \bar \theta \|_{ H^{-1}(Q) } \left(  \int_{Q} | \nabla ( \eta  W'( \bar u )  ( u - (1 + \eta \phi_{Q,e_{i}}^{\gamma} \partial_i) \bar u ) ) |^2  \right)^{ \frac{1}{2} } \\
&\qquad +  \int_{Q} ( 1 - \eta ) ( \theta^{\gamma} - \bar \theta ) W'( \bar u )  ( u - (1 + \eta \phi_{Q,e_{i}}^{\gamma} \partial_i) \bar u ) \\
&\quad \lesssim  \|  \theta^{\gamma} - \bar \theta \|_{ H^{-1}(Q) } \left( \int_{Q} | \nabla ( \eta W'( \bar u )  ( u - (1 + \eta \phi_{Q,e_{i}}^{\gamma} \partial_i) \bar u ) ) |^2 \right)^{ \frac{1}{2} } \\
&\qquad + \left( \int_{Q} ( ( 1 - \eta ) W'( \bar u ) )^2 \right)^{ \frac{1}{2} } \left( \int_{Q}  ( u - (1 + \eta \phi_{Q,e_{i}}^{\gamma} \partial_i) \bar u )^2 \right)^{ \frac{1}{2} } .
\end{align*}
As before, we may use the Poincaré inequality and Lemma \ref{lemma:homogeneous-interior-schauder-estimate}, also in form of \eqref{eqn:moser-conclusion}, together with the energetic minimality of $ \bar u $ in form of the estimate in Proposition \ref{P: uniqueness comparison}, to estimate
\begin{align*}
& \left( \int_{Q} | \nabla ( \eta W'( \bar u )  ( u - (1 + \eta \phi_{Q,e_{i}}^{\gamma} \partial_i) \bar u ) ) |^2 \right)^{ \frac{1}{2} } \\
& \quad \lesssim \left(  \sup_{ Q } | \eta W'( \bar u ) | + | \nabla ( \eta W'(\bar u) ) ) | \right) \left( \int_{Q} | \nabla ( u - (1 + \eta \phi_{Q,e_{i}}^{\gamma} \partial_i) \bar u ) |^2 \right)^{ \frac{1}{2} } \\
& \quad \leq C_{ \eta } \mathscr{F}_{ \gamma } ( u ; Q )^{ \frac{1}{2} } \left( \int_{Q} | \nabla ( u - (1 + \eta \phi_{Q,e_{i}}^{\gamma} \partial_i) \bar u ) |^2 \right)^{ \frac{1}{2} }.
\end{align*}
Furthermore, by Sobolev inequality, used as in \eqref{eqn:meyer02} below, we can estimate
\begin{align*}
\left( \int_{Q} ( ( 1 - \eta ) W'( \bar u ) )^2 \right)^{ \frac{1}{2} }
&\leq \left( \int_{Q} ( 1 - \eta )^{2q} \right)^{ \frac{1}{2q} } \left( \int_{Q} ( W'( \bar u ) )^{2p}  \right)^{ \frac{1}{2p} } \\
&\lesssim \left( \int_{Q} ( 1 - \eta )^{2q} \right)^{ \frac{1}{2q} }  \left( \int_{Q} | \nabla \bar u |^2 + \int_{Q} W(\bar u) \right)^{ \frac{1}{2} }
\end{align*}
for $ p $ close to one. Inserted in the last three equations into each other implies
\begin{equation}\label{eqn:proof-h1-estimate-cut-off-two-scale-expansion-2}
\begin{aligned}
& \int_{Q} ( \theta^{\gamma} - \bar \theta ) W'( \bar u )  ( u - (1 + \eta \phi_{Q,e_{i}}^{\gamma} \partial_i) \bar u ) \\
& \qquad \lesssim \Big(  C_{ \eta } \| \theta^{\gamma} - \bar \theta \|_{ H^{ -1 } ( Q ) }
+  \left( \int ( 1 - \eta )^{2q}  \right)^{ \frac{1}{2q} } \Big) \mathscr{F}_{ \gamma }(u ; Q)^{ \frac{1}{2} }  \\
& \qquad\quad \times \left( \int_{Q} | \nabla ( u - (1 + \eta \phi_{Q,e_{i}}^{\gamma} \partial_i) \bar u ) |^2 \right)^{ \frac{1}{2} }.
\end{aligned}
\end{equation}
Using Lemmas \ref{lemma:homogeneous-calderon-zygmund-estiamte}, \ref{lemma:homogeneous-interior-schauder-estimate}, and \ref{lemma:l2-homogenization-estimate} together with the energetic minimality of $ \bar u $ in form of the estimate in Proposition \ref{P: uniqueness comparison}, on the equations \eqref{eqn:proof-h1-estimate-cut-off-two-scale-expansion-1} and \eqref{eqn:proof-h1-estimate-cut-off-two-scale-expansion-2}, we obtain
$$
\begin{aligned}
&\int_{Q} | \nabla ( u - (1 + \eta \phi_{Q,e_{i}}^{\gamma} \partial_i) \bar u ) |^2 \\
&\qquad \leq o ( 1 ) \F_{\gamma}(u;Q) 
+ C_{\eta} \F_{\gamma}(u;Q) \left( \int_{Q} | ( \phi_{Q}^{\gamma}, \sigma_{Q}^{\gamma} ) |^2 +  \| \theta^{\gamma} - \bar \theta \|_{ H^{-1}(Q) }^2 \right) \\
&\qquad\quad + C \left( \int_{Q} (1 - \eta)^{2q} \right)^{\frac{1}{q}}  \left( \F_{\gamma}(u;Q)
+ \left( \int_{Q} | \nabla u |^{2p} \right)^{\frac{1}{p}} \right) 
+ | \bar {a} (Q) - \bar a |^{2} \F_{\gamma}(u;Q) ,
\end{aligned}
$$
where $ o(1) \rightarrow 0 $ as $ \int_{Q} | ( \phi_{Q}^{\gamma}, \sigma_{Q}^{\gamma} ) |^2 + \| \theta^{\gamma} - \bar{\theta} \|_{H^{-1}(Q)}^{2} + | \bar{a}(Q) - \bar{a}|^{2} \rightarrow 0 $. That yields the claim.
\end{proof}

We can connect the lemmas proven above to show Proposition \ref{prop:homogenization-interior-qubes}.

\begin{proof}[Proof of Proposition \ref{prop:homogenization-interior-qubes}.]
Since the assumptions in Section \ref{S: deterministic assumptions}, in particular the {Helm\-holtz} decomposition \eqref{eqn:a-helmholtz-decomposition}, are unchanged if we subtract a constant from $(\phi_{Q},\sigma_{Q})$, we assume without loss of generality that $ \fint_{Q} ( \phi_{Q}^{\gamma}, \sigma_{Q}^{\gamma}) = 0 $. The proof itself will be given in several steps. Throughout them $ \eta \in C^{\infty}(\R^d) $ denotes a smooth cut-off with $ \supp{\eta} $ compactly contained in $ Q $ and $ 0 \le \eta \le 1 $ that we will choose in the end. Constants depending on $ \eta $ are tracked by a subscript.

Before we start the actual proof, the reader may wish to revisit the definitions of $ W_{ \rm reg } $ and $ \mathscr{F}^{ \rm reg } $ in Section \ref{sec: regularized_potential_reg}. It will be useful to keep in mind that these coincide with $ W $, $ \mathscr{F} $ on functions with values between $ - 1 $ and $ 1 $, but differ when the two-scale correction is added. 

\textit{Step 1 (Convexity).} We start the argument by showing that for every $ \alpha > 0 $, there exists a $ \delta_1 = \delta_1 ( \alpha, \eta ) > 0 $ such that
\begin{align*}
| \bar { a } (Q) - \bar a | + \left( \int_{Q} | ( \phi_{Q}^{\gamma}, \sigma_{Q}^{\gamma}) |^2 \right)^{ \frac{1}{2} }+ \| \theta^{\gamma} - \bar \theta \|_{ H^{-1}(Q) } < \delta_1
\end{align*}
implies
\begin{equation}\label{eq:cut-off-two-scale-energy-difference}
\begin{aligned}
& | \F_{\gamma}^{ \rm reg } ((1 + \eta \phi_{Q,e_i}^{\gamma} \partial_i) \bar u ; Q) - \F_{\gamma}^{ \rm reg } (u ; Q) | \phantom{\fint} \\
&\quad \leq \alpha \F_{\gamma} (u ; Q)
 + C_{\eta} \F_{\gamma} (u ; Q) \left( \int_{Q} | ( \phi_{Q}^{\gamma}, \sigma_{Q}^{\gamma} ) |^2 + \| \theta^{\gamma} - \bar \theta \|_{ H^{-1} (Q) } ^ { 2 }  \right) \\
&\quad\quad + C \left( \F_{\gamma} (u ; Q) + \left( \int_{Q} | \nabla u |^{2p} \right)^{\frac{1}{p}} \right) \left( \int_{Q} (1 - \eta)^{2q} \right)^{\frac{1}{q}}
+ \F_{\gamma}(u; Q) | \bar a_Q - \bar a |^{ 2 } .
\end{aligned}
\end{equation}
That is, the convexity,~cf.~Proposition \ref{prop:convexity}, enables us to lift the $ H^1 $-bound on the (cut-off) two-scale expansion from Lemma \ref{lemma:h1-estimate-cut-off-two-scale-expansion} to the energies.

We want to appeal to Proposition \ref{prop:convexity}. For convenience, we introduce $\delta u = (1 + \eta \phi_{Q,e_i}^{ \gamma } \partial_{i}) \bar{u} - u$. Using the explicit form for $ D \mathscr{F}_{\gamma}^{ \rm reg }  $ that we derived \eqref{eq:derivative-energy}, the r.h.s.~in \eqref{eqn:ee-diff-convexity} becomes
$$
\begin{aligned}
& \langle D\F_{\gamma}^{ \rm reg } ((1 + \eta \phi_{Q,e_i}^{\gamma} \partial_i) \bar u ; Q) - D\F_{\gamma}^{ \rm reg } (u ; Q), (1 + \eta \phi_{Q,e_i}^{\gamma} \partial_i) \bar u - u \rangle \phantom{\fint} \\
& \qquad = \int_{Q} a^{\gamma} \nabla ( (1 + \eta \phi_{Q,e_i}^{\gamma} \partial_i) \bar u ) \cdot \nabla \delta u
+ \int_{Q} \theta^{\gamma} W_{ \rm reg }'( (1 + \eta \phi_{Q,e_i}^{\gamma} \partial_i) \bar u ) \delta u \\
&\qquad\quad -\int_{Q} a^{\gamma} \nabla u \cdot \nabla \delta u
- \int_{Q}  \theta^{\gamma} W_{ \rm reg } '( u ) \delta u ,
\end{aligned}
$$
which by Young's and Poincaré inequality is estimated  by (recall also the global Lipschitz bound on $ W_{\rm reg}' $ in \eqref{E: regularized_potential_reg})
$$
\begin{aligned}
& \langle D\F_{\gamma}^{ \rm reg } ((1 + \eta \phi_{Q,e_i}^{\gamma} \partial_i) \bar u ; Q) - D\F_{\gamma}^{ \rm reg } (u ; Q), (1 + \eta \phi_{Q,e_i}^{\gamma} \partial_i) \bar u - u \rangle \phantom{\fint}  \\
& \qquad \lesssim \int_{Q} | \nabla ( (1 + \eta \phi_{Q,e_i}^{\gamma} \partial_i) \bar u - u ) |^2 
+ \int_{Q} | \nabla \delta u |^2 \\
& \qquad \quad + \int_{Q} | W_{ \rm reg }' ( (1 + \eta \phi_{Q,e_i}^{\gamma} \partial_i) \bar u )  - W_{ \rm reg }'(u) |^2
+ \int | \delta u |^2 \\
& \qquad \lesssim \int_{Q} | \nabla \delta u |^2. 
\end{aligned}
$$
This, together with Lemma \ref{lemma:h1-estimate-cut-off-two-scale-expansion} and estimate \eqref{eqn:ee-diff-convexity}, implies the claimed \eqref{eq:cut-off-two-scale-energy-difference}.

\textit{Step 2 (Boundary layer estimate).} We now argue that the contributions along the boundary layer are negligible. More precisely, we show
\begin{equation}\label{eq:boundary-layer-energy}
\overline{\mathscr{F}}^{\rm reg} (\bar u ; Q \setminus \varrho Q)
\lesssim | Q \setminus \varrho Q |^{ 1 - \frac{1}{p} } \left( \left( \int_{Q} | \nabla \bar u |^{2p} \right)^{\frac{1}{p}} + \overline{\mathscr{F}} (\bar u ; Q) \right) .
\end{equation}
The same statement holds true with $ \bar u $ and $ \overline{\mathscr{F}} $ replaced by $ u $ and $ \mathscr{F}_{\gamma} $.

First, observe that
\begin{equation}\label{eqn:en-hom-01}
\int_{Q \setminus \varrho Q} | \nabla \bar u |^2
\leq | Q \setminus \varrho Q |^{ 1 - \frac{1}{p} } \left( \int_{Q \setminus \varrho Q} | \nabla \bar u |^{2p} \right)^{\frac{1}{p}}
\end{equation}
for $ p > 1 $. To show a similar estimate for the well term, we distinguish the cases $ \overline{\mathscr{F}} (\bar u ; Q) \ge \delta $ and $  \overline{\mathscr{F}} (\bar u ; Q) < \delta $, where $ \delta > 0 $ is chosen such that the clearing-out property, i.e.~the conclusion of Lemma \ref{lemma:clearing-out}, holds.

If $  \overline{\mathscr{F}} (\bar u ; Q) < \delta $, we may assume w.l.o.g.~that $ | \bar u - 1 | \ll 1 $, so that by \eqref{eqn:w-poly} we have
$$
W_{ \rm reg } (\bar u) \sim | \bar u - 1|^{ 2 \kappa } \quad \text{in} ~ Q.
$$
Hence
$$
\int_{Q \setminus \varrho Q} W_{ \rm reg } (\bar u) 
\lesssim | Q \setminus \varrho Q |^{ 1 - \frac{1}{p} } \left( \int_{Q} | ( \bar u - 1 )^{\kappa} |^{2 p}  \right)^{\frac{1}{p}}
$$
for any $ p > 1 $. If $ p $ is sufficiently close to one, we can appeal to the Sobolev inequality to obtain
$$
\begin{aligned}
\int_{Q \setminus \varrho Q} W_{ \rm reg } (\bar u) 
&\lesssim| Q \setminus \varrho Q |^{ 1 - \frac{1}{p} } \left( \int_{Q} | \nabla ( \bar u - 1 )^{\kappa} |^2 + \int_{Q} | ( \bar u - 1 )^{\kappa} |^2 \right).
\end{aligned}
$$
Since $ | \bar u - 1 | \ll 1 $ this implies
$$
\begin{aligned}
\int_{Q \setminus \varrho Q} W_{ \rm reg } (\bar u) 
&\lesssim | Q \setminus \varrho Q |^{ 1 - \frac{1}{p} } \left( \int_{Q} | \nabla \bar u |^2 + \int_{Q} W_{ \rm reg } (\bar u) \right).
\end{aligned}
$$
Note that this estimate also holds true in the regime  $ \overline{\mathscr{F}} (\bar u ; Q) \ge \delta $ since $ W_{ \rm reg} (\bar u) = W (\bar u) $ is bounded. With the gradient estimate \eqref{eqn:en-hom-01} this combines to \eqref{eq:boundary-layer-energy}. Again, to reiterate, the same argument would also work with $ \bar u $ replaced by $ u $ and $ \overline{\mathscr{F}} $ by $ \mathscr{F}_{\gamma} $.

\textit{Step 3 (Homogenization on the level of energies).} We now argue, that the cut-off two-scale expansion is also a good approximation on the level of energies. To make this precise, we assume from this point on that the cut-off $ \eta $ is chosen such that $ \supp{\eta} \subset \varrho Q $ (for $ 0 < \varrho < 1 $ as in Step 2) and $ 0 \le \eta \le 1 $. For convenience, let us also assume that $ \supp{\eta} $ is convex. We claim that there exists a $ \delta_2  > 0 $ such that if
\begin{align*}
\left( \int_{Q} | ( \phi_{Q}^{\gamma}, \sigma_{Q}^{\gamma} ) |^2 \right)^{\frac{1}{2}} + \| \theta^{\gamma} - \bar{\theta} \|_{H^{-1}(Q)} + | \bar{a}(Q) - \bar{a} | < \delta_2,
\end{align*}
then
\begin{equation}\label{eq:cut-off-two-scale-expansion-energy}
\begin{aligned}
& | \F_{\gamma}^{ \rm reg } ((1 + \eta \phi_{Q,e_i}^{\gamma} \partial_i) \bar u ;  \varrho Q) - \overline{\mathscr{F}}^{ \rm reg } (\bar u ; \varrho Q) |  \\
&\qquad \le C_{\eta} \overline{\mathscr{F}} (\bar u ; Q) \left( \int_{Q} | (\phi_{Q}^{\gamma}, \sigma_{Q}^{\gamma}) |^2 + \| \theta^{\gamma} - \bar \theta \|_{ H^{-1}(Q) }^2  + | \bar a_{Q} - \bar a |^2 \right)^{\frac{1}{2}} \\
& \qquad\quad + \left( \int_{\varrho Q} | 1 - \eta|^{q} \right)^{\frac{1}{q}} \left(  C_{\varrho} \overline{\mathscr{F}} (\bar u ; Q) + \left( \int_{Q} | \nabla \bar u |^{2p} \right)^{\frac{1}{p}} \right),
\end{aligned}
\end{equation}
where $ \frac{1}{q} + \frac{1}{p} = 1 $ with $ q $ sufficiently large. We will show the above estimate for the gradient and well part separately.

\textit{Step 3.1 (Gradient term).} Let us begin with the gradient term that splits into
\begin{equation}\label{eq:two-scale-expansion-gradient-energy}
\begin{aligned}
& \int_{\varrho Q} a^{\gamma} \nabla ( (1 + \eta \phi_{Q,e_i}^{\gamma} \partial_i ) \bar u ) \cdot \nabla ( (1 + \eta \phi_{Q,e_j}^{\gamma} \partial_j ) \bar u ) \\
&\quad\quad= \int_{\varrho Q} a^{\gamma} \nabla ( (1 + \eta \phi_{Q,e_i}^{\gamma} \partial_i ) \bar u ) \cdot \nabla \bar u
+ \int_{\varrho Q} a \nabla ( (1 + \eta \phi_{Q,e_i}^{\gamma} \partial_i ) \bar u ) \cdot \nabla ( \eta \phi_{Q,e_j}^{\gamma} \partial_j \bar u )
\end{aligned}
\end{equation}
We also argue for the two terms on the r.h.s.~of \eqref{eq:two-scale-expansion-gradient-energy} separately.

On the first term on the r.h.s.~in \eqref{eq:two-scale-expansion-gradient-energy}, we decompose the flux $a^{\gamma} \nabla(1 + \eta \phi_{Q,e_i} \partial_{i}) \bar{u}$ using identity \eqref{eq:cut-off-two-scale-flux} to obtain
$$
\begin{aligned}
&\int_{\varrho Q} a^{\gamma} \nabla ( (1 + \eta \phi_{Q,e_i} \partial_i ) \bar u ) \cdot \nabla \bar u \\
&\quad\quad = \int_{\varrho Q} \bar a \nabla \bar u \cdot \nabla \bar u
+ \int_{\varrho Q} \phi_{Q,e_i}^{\gamma} a^{\gamma} \nabla ( \eta \partial_i \bar u ) \cdot \nabla \bar u 
+ \int_{\varrho Q} (\nabla \cdot \sigma^{\gamma}_{Q,e_i}) \cdot (\eta \partial_{i} \bar{u} \nabla \bar{u}) \\
&\quad \quad \quad + \int_{\varrho Q} (1 - \eta) ( a^{\gamma} - \bar {a} (Q) ) \nabla \bar u \cdot \nabla \bar u
+ \int_{ \varrho Q } ( \bar {a} (Q) - \bar a ) \nabla \bar u \cdot \nabla \bar u .
\end{aligned}
$$
After integrating-by-parts to take the derivative off of $\sigma_{Q,i}$ and invoking Lemma \ref{lemma:homogeneous-interior-schauder-estimate} to control the derivatives of $\bar{u}$ that appear, this becomes
\begin{equation}\label{eqn:en-hom-02}
\begin{aligned}
& \left| \int_{\varrho Q} a^{\gamma} \nabla ( (1 + \eta \phi_{Q,e_i} \partial_i ) \bar u ) \cdot \nabla \bar u - \int_{\varrho Q} \bar a \nabla \bar u \cdot \nabla \bar u \right| \\
&\quad\quad \lesssim C_{\eta} \left( \int_{Q} | (\phi_{Q}^{\gamma}, \sigma_{Q}^{\gamma} ) |^2 \right)^{\frac{1}{2}} \overline{\mathscr{F}} (\bar u ; Q)
+ \left( \int_{\varrho Q} | 1 - \eta|^{q} \right)^{\frac{1}{q}} \left( \int_{Q} | \nabla \bar u |^{2p} \right)^{\frac{1}{p}} \\
&\quad\qquad + | \bar { a } (Q) - \bar a | \overline{\mathscr{F}} (\bar u ; Q) . \phantom{\int}
\end{aligned}
\end{equation}
where $ \frac{1}{p} + \frac{1}{q} = 1 $. Note that all these terms appear in the right-hand side of \eqref{eq:cut-off-two-scale-expansion-energy}.

On the second term in \eqref{eq:two-scale-expansion-gradient-energy}, we also appeal to \eqref{eq:cut-off-two-scale-intertwining-property} to obtain
$$
\begin{aligned}
\int_{\varrho Q} a \nabla ( (1 + \eta \phi_{Q,e_i} \partial_i ) \bar u ) \cdot \nabla ( \eta \phi^{\gamma}_{Q,e_j} \partial_j \bar u )
&= - \int_{\varrho Q} \bar\theta \, W_{ \rm reg }'(\bar u) \eta \phi^{\gamma}_{Q,e_i} \partial_i \bar u\\
&+ \int_{\varrho Q} ( \phi_{Q,e_i}^{\gamma} a^{\gamma} - \sigma_{Q,e_i}^{\gamma} ) \nabla ( \eta \partial_i \bar u ) \cdot \nabla ( \eta \phi^{\gamma}_{Q,e_j} \partial_j \bar u ) \\
& + \int_{\varrho Q} (1 - \eta) ( a^{\gamma} - \bar {a} (Q) ) \nabla \bar u \cdot \nabla (\eta \phi^{\gamma}_{Q,e_j} \partial_j \bar u) \\
& + \int_{\varrho Q} ( \bar { a } (Q) - \bar a ) \nabla \bar u \cdot \nabla (\eta \phi^{\gamma}_{Q,e_j} \partial_j \bar u) . 
\end{aligned}
$$
Note that by Lemma \ref{lemma:homogeneous-interior-schauder-estimate} and Caccioppoli's estimate (applicable since $ \supp{\eta} \subset \varrho Q $), applied to the gradient of $ \phi_{Q,e_i}^{\gamma} + ( x - x_0 ) \cdot e_i $, where $ x_0 $ denotes the center of $ Q $,
$$
\begin{aligned}
\int_{\varrho Q} | \nabla (\eta \phi^{\gamma}_{Q,e_j} \partial_j \bar u) |^2
&\leq C_{\eta} \overline{\mathscr{F}} (\bar u ; Q) \left( \int_{Q} | \eta \nabla \phi_{Q}^{\gamma} |^2  + \int_{Q} | \phi_{Q}^{\gamma} |^2 
\right) \\
&\leq C_{\eta} \overline{\mathscr{F}} (\bar u ; Q) \left( 1 + \int_{Q} | \phi_{Q}^{\gamma} |^2 \right)
\end{aligned}
$$
and
$$
\int_{\varrho Q} ( \eta \phi^{\gamma}_{Q,e_j} \partial_j \bar u)^2
\le C_{\eta} \overline{\mathscr{F}} (\bar u ; Q) \int_{Q} | \phi_{Q}^{\gamma} |^2.
$$
Combining this with another application of Lemma \ref{lemma:homogeneous-interior-schauder-estimate}, the above identity is estimated by
$$
\begin{aligned}
&\left| \int_{\varrho Q} a \nabla ( (1 + \eta \phi^{\gamma}_{Q,e_i} \partial_i ) \bar u ) \cdot \nabla ( \eta \phi^{\gamma}_{Q,e_j} \partial_j \bar u ) \right| \\
&\qquad \le C_{\eta} \overline{\mathscr{F}} (\bar u ; Q)^{\frac{1}{2}} \left( \int_{Q} W_{ \rm reg }'(\bar u)^2 \right)^{\frac{1}{2}} \left( \int_{Q} | (\phi_{Q}^{\gamma}, \sigma_{Q}^{\gamma}) |^2 \right)^{\frac{1}{2}} \\
&\qquad\quad + C_{\eta} \overline{\mathscr{F}} (\bar u ; Q) \left( \int_{Q} | ( \phi_{Q,e_i}^{\gamma}, \sigma_{Q,e_i}^{\gamma} )|^2 + | \bar a_{Q} - \bar a |^2 \right)^{\frac{1}{2}} \left( 1 + \int_{Q} | \phi_{Q}^{\gamma} |^2 \right)^{ \frac{1}{2} }  \\
&\qquad\quad + \left| \int_{\varrho Q} (1 - \eta) ( a^{\gamma} - \bar a ) \nabla \bar u \cdot \nabla (\eta \phi^{\gamma}_{Q,e_j} \partial_j \bar u) \right|.
\end{aligned}
$$
On the last term, we need to be more careful. Note that
$$
\nabla (\eta \phi^{\gamma}_{Q,e_j} \partial_j \bar u)
= \phi^{\gamma}_{Q,e_j} \nabla ( \eta \partial_j \bar u ) + \eta ( \partial_j \bar u ) \nabla \phi^{\gamma}_{Q,e_j}
$$
so that Lemma \ref{lemma:homogeneous-interior-schauder-estimate} shows
$$
\begin{aligned}
& \left| \int_{\varrho Q} (1 - \eta) ( a^{\gamma} - \bar a ) \nabla \bar u \cdot \nabla (\eta \phi_{Q,e_j} \partial_j \bar u) \right| \\
&\quad\quad \leq C_{\eta} \left( \int_{Q} | \phi_{Q}^{\gamma} |^2 \right)^{\frac{1}{2}} \overline{\mathscr{F}} (\bar u ; Q) \\
&\quad\quad\quad + C_{\varrho} \left( \int_{\varrho Q} | \nabla \phi_{Q}^{\gamma} |^2 \right)^{\frac{1}{2}} \left( \int_{\varrho Q} | 1 - \eta |^2 \right)^{\frac{1}{2}} \overline{\mathscr{F}} (\bar u ; Q).
\end{aligned}
$$
We want to remark that the constant in the last line depends only on $\varrho$, not on $ \eta $, since we may apply Lemma \ref{lemma:homogeneous-interior-schauder-estimate} on $ \varrho Q \subset Q $ after neglecting the cut-off $ \eta $. By another application of Caccioppoli's estimate, the last three equations combine to
$$
\begin{aligned}
&\left| \int_{\varrho Q} a \nabla ( (1 + \eta \phi_{Q,e_i} \partial_i ) \bar u ) \cdot \nabla ( \eta \phi_{Q,e_j} \partial_j \bar u ) \right| \\
&\qquad \le C_{\eta} \overline{\mathscr{F}} (\bar u ; Q)^{\frac{1}{2}} \left( \int_{Q} W_{ \rm reg }'(\bar u)^2 \right)^{\frac{1}{2}} \left( \int_{Q} | (\phi_{Q}^{\gamma}, \sigma_{Q}^{\gamma}) |^2 \right)^{\frac{1}{2}} \\
&\qquad\quad + C_{\eta} \overline{\mathscr{F}} (\bar u ; Q) \left( \int_{Q} | ( \phi_{Q}^{\gamma}, \sigma_{Q}^{\gamma} )|^2 + | \bar {a} (Q) - \bar a |^2 \right)^{\frac{1}{2}} \left( 1 + \int_{Q} | \phi_{Q}^{\gamma} |^2 \right)^{ \frac{1}{2} }  \\
&\quad\quad\quad + C_{\varrho} \left( 1 + \left( \int_{Q} | \phi_{Q}^{\gamma} |^2 \right)^{\frac{1}{2}} \right) \left( \int_{\varrho Q} | 1 - \eta |^2 \right)^{\frac{1}{2}} \overline{\mathscr{F}} (\bar u ; Q).
\end{aligned}
$$
The well term $ W_{ \rm reg }' $ is handled via a Sobolev inequality like in \eqref{eqn:meyer02} so that overall we obtain the estimate
$$
\begin{aligned}
&\left| \int_{\varrho Q} a \nabla ( (1 + \eta \phi_{Q,e_i} \partial_i ) \bar u ) \cdot \nabla ( \eta \phi_{Q,e_j} \partial_j \bar u ) \right| \\
&\qquad \le C_{\eta} \overline{\mathscr{F}} (\bar u ; Q) \left( \int_{Q} | (\phi_{Q}^{\gamma}, \sigma_{Q}^{\gamma}) |^2 + | \bar a_{Q} - \bar a |^2 \right)^{\frac{1}{2}} \left( 1 + \int_{Q} | (\phi_{Q}^{\gamma}, \sigma_{Q}^{\gamma}) |^2 \right)^{ \frac{1}{2} } \\
&\qquad\quad + C_{\varrho} \left( 1 + \int_{Q} | ( \phi_{Q}^{\gamma}, \sigma_{Q}^{\gamma} ) |^2 \right)^{\frac{1}{2}} \left( \int_{\varrho Q} | 1 - \eta |^2 \right)^{\frac{1}{2}} \overline{\mathscr{F}} (\bar u ; Q).
\end{aligned}
$$
Assuming $ \delta_2 < 1 $, these terms are exactly the errors appearing in \eqref{eq:cut-off-two-scale-expansion-energy}. By combining this with \eqref{eqn:en-hom-02}, we finish the estimate for \eqref{eq:two-scale-expansion-gradient-energy}. This contains all contributions of the gradient part to \eqref{eq:cut-off-two-scale-expansion-energy}.

\textit{Step 3.2 (Potential term).}  On the well part of \eqref{eq:cut-off-two-scale-expansion-energy}, we appeal to another convexity argument. We start with the observation
$$
|W_{ \rm reg } ( \bar u + \eta \phi_{Q,e_i}^{\gamma} \partial_i \bar u) - W_{ \rm reg } ( \bar u ) -  W_{ \rm reg } '(\bar u) \eta \phi_{Q,e_i}^{\gamma} \partial_i \bar u |
\leq \frac{1}{2} \sup | W_{ \rm reg }'' | ( \eta \phi_{Q,e_i}^{\gamma} \partial_i \bar u )^2.
$$
It is crucial to work with a second order Taylor estimate here, to obtain an estimate with the right scaling on the r.h.s.. Indeed, by Lemma \autoref{lemma:homogeneous-interior-schauder-estimate},
\begin{equation}\label{eqn:pf-hom-well-potb}
\begin{aligned}
& \left| \int_{ \varrho Q } \theta^{\gamma} W_{ \rm reg } ( \bar u + \eta \phi_{Q,e_i}^{\gamma} \partial_i \bar u) - \int_{ \varrho Q } \theta^{\gamma} W_{ \rm reg } ( \bar u ) \right| \\
& \qquad \leq C_{\eta} \left( \int_{ Q } | W_{ \rm reg } '(\bar u) |^2 \right)^{\frac{1}{2}} \left( \int_{Q} | \phi_{Q}^{\gamma} |^2 \right)^{\frac{1}{2}} \overline{\mathscr{F}} (\bar u ; Q)^{\frac{1}{2}} \\
&\qquad\quad  + C_{\eta} \left( \int_{ Q } | \phi_{Q}^{\gamma} |^2 \right)^{\frac{1}{2}} \overline{\mathscr{F}} (\bar u ; Q)
\end{aligned}
\end{equation}
As above, the $ W' $ term may be handled via a Sobolev inequality, cf.~\eqref{eqn:meyer02}.

In addition to \eqref{eqn:pf-hom-well-potb}, we need to control $ ( \theta^{\gamma} - \bar \theta ) W_{\rm reg} ( \bar u ) $, which we split up in a boundary and interior contribution. Since $ \supp \eta \subset \varrho Q \subset Q $, we have
\begin{equation}\label{eqn:pf-hom-well-pot}
\begin{aligned}
& \left| \int_{ \varrho Q } ( \theta^{\gamma} - \bar \theta ) \eta W_{ \rm reg } ( \bar u ) \right| \\
&\qquad \leq C_{ \eta } \| \theta^{\gamma} - \bar \theta \|_{ H^{-1}(Q) } \left( \int_{ \supp \eta } | W_{ \rm reg } (\bar u) |^2 + \int_{ \varrho Q } | \nabla W_{ \rm reg } (\bar u) |^2   \right)^{ \frac{1}{2} }.
\end{aligned}
\end{equation}
Since we assume that $ \supp{ \eta } $ is convex, we can apply the Poincaré inequality on $ \supp{\eta} $ to obtain
$$
\begin{aligned}
\left( \int_{ \supp \eta } | W_{ \rm reg } ( \bar u ) |^2 \right)^{ \frac{1}{2} }
&\leq \left( \int_{ \supp \eta } | W_{ \rm reg } ( \bar u ) - \fint_{ \supp \eta } W_{ \rm reg } ( \bar u ) |^2 \right)^{ \frac{1}{2} } + C_{ \eta }  \fint_{ \supp \eta } W_{ \rm reg } ( \bar u ) \\
&\leq C_{ \eta } \left( \left( \int_{ \supp \eta } | \nabla W_{ \rm reg } (\bar u) |^2 \right)^{ \frac{1}{2} } + \int_{ Q } W_{ \rm reg } ( \bar u ) \right).
\end{aligned}
$$
Since $ -1 \leq \bar u \leq 1 $, we may use the previous estimate to post-process \eqref{eqn:pf-hom-well-pot} to
\begin{equation} \label{eqn:pf-hom-well-potd}
\begin{aligned}
& \left| \int_{ \varrho Q } ( \theta^{\gamma} - \bar \theta ) \eta W_{ \rm reg } ( \bar u ) \right| \\
& \qquad \leq C_{ \eta } \| \theta^{\gamma} - \bar \theta \|_{ H^{-1}(Q) } \left( \overline{ \mathscr{F} } ( u ; Q ) + \left( \int_{ \supp \eta } | \nabla W (\bar u) |^2 \right)^{ \frac{1}{2} } \right).
\end{aligned}
\end{equation}
We can estimate
$$
\left( \int_{ \supp \eta  } | \nabla W(\bar u) |^2 \right)^{ \frac{1}{2} }
= \left( \int_{ \supp \eta  }  | W'( \bar u ) \nabla \bar u |^{ 2 } \right)^{ \frac{1}{2} }
\lesssim ( \sup_{ \supp \eta  } | W'( \bar u ) | ) \left( \int_{ Q } | \nabla \bar u |^2 \right)^{ \frac{1}{2} }.
$$
By virtue of \eqref{eqn:moser-conclusion}, in the proof of Lemma \ref{lemma:homogeneous-interior-schauder-estimate}, we may conclude
\begin{equation}\label{eqn:pf-hom-well-potc}
\left| \int_{ \varrho Q } ( \theta^{\gamma} - \bar \theta ) \eta W_{ \rm reg } ( \bar u ) \right|
\leq C_{ \eta } \| \theta^{\gamma} - \bar \theta \|_{ H^{-1}(Q) } \overline{ \mathscr{F} } ( \bar u ; Q ),
\end{equation}
whenever $ \overline{ \mathscr{F} } ( \bar u ; Q ) \ll 1 $. In the opposite regime $ \overline{ \mathscr{F} } ( \bar u ; Q ) \gtrsim 1 $, we don't need to use \eqref{eqn:moser-conclusion} to conclude the above inequality since $ W'_{\rm reg} ( \bar u ) $ is uniformly bounded.

Lastly, we may argue by Sobolev's inequality (as done in Step 2) that
\begin{align}\label{eqn:pf-hom-well-potc-2}
\left| \int_{ \varrho Q } ( \theta^{\gamma} - \bar \theta ) ( 1 - \eta ) W_{ \rm reg }( \bar u ) \right|
&\lesssim \left( \int_{ \varrho Q } ( 1 - \eta )^{ q } \right)^{ \frac{1}{q} } \overline{ \mathscr{F} } ( \bar u ; Q )
\end{align}
for some large $ q \gg 1 $. Together, \eqref{eqn:pf-hom-well-potb}, \eqref{eqn:pf-hom-well-potc}, and \eqref{eqn:pf-hom-well-potc-2} deal with the well part of \eqref{eq:cut-off-two-scale-expansion-energy}.

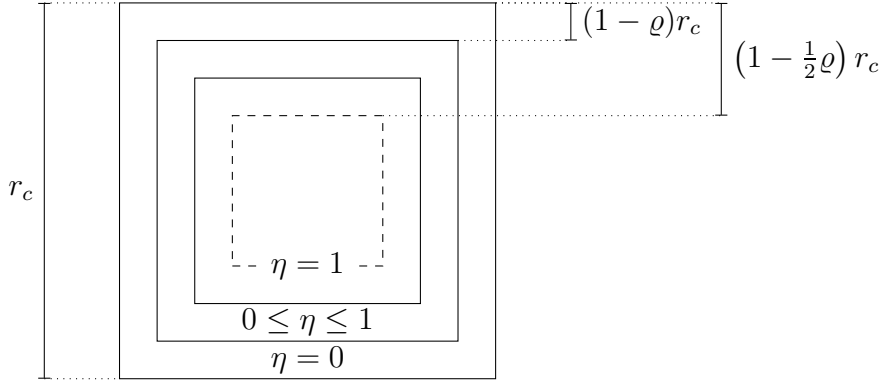
\begin{figure}

\centering

\begin{tikzpicture}
	\draw (0,0) rectangle (5,5);
	\draw [dotted] (-1, 5) -- (0, 5);
	\draw [dotted] (-1, 0) -- (0, 0);
	\draw [|-|] (-1,0) -- node[left] { $ r_c $ } (-1,5);

	\draw (0.5, 0.5) rectangle (4.5, 4.5);
	\draw [dotted] (4.5, 4.5) -- (6, 4.5);
	\draw [dotted] (5, 5) -- (6, 5);
	\draw [|-|] (6, 4.5) -- node[right] { $ ( 1 - \varrho ) r_c $ } (6, 5);

	\draw [ dashed, opacity = 0.2 ] (1.5, 1.5) rectangle (3.5, 3.5);
	\draw [dotted] (3.5, 3.5) -- (8, 3.5);
	\draw [dotted] (5, 5) -- (8, 5);
	\draw [|-|] (8, 3.5) -- node[right] { $ \left( 1 - \frac{ 1 }{ 2 } \varrho \right) r_c $ } (8, 5);

	\draw (1, 1) rectangle (4, 4);
	\node [ fill = white ] at (2.5, 1.5) {$ \eta = 1 $};
	\node at (2.5, 0.75) {$ 0 \le \eta \le 1 $};
	\node at (2.5, 0.25) {$ \eta = 0 $};
\end{tikzpicture}

\caption{Cut-off strategy used in proof of Proposition \autoref{prop:homogenization-interior-qubes}: The boundary layer of width $ 1 - \varrho $ (where $ \eta = 0 $) is energetically negligible while on the set $ \{ \eta = 1 \} $ the two-scale expansion is a good approximation in energy. The transition layer $ \{ 0 < \eta < 1 \} $, which is squeezed into $ \varrho Q \setminus \frac{\varrho}{2} Q $, appears as a technical artifact of the proof and is also shown to be negligible as well.}

\end{figure}

\textit{Step 5 (Conclusion).} We now prove the proposition. Given $ \alpha > 0 $, we show that there exists some $ \delta > 0 $ and a constant $ C_{\alpha} > 0 $ such that if
\begin{align*}
 | \bar {a} (Q) - \bar a | + \left( \int_{Q} | (\phi_{Q}^{\gamma}, \sigma_{Q}^{\gamma}) |^2 \right)^{ \frac{1}{2} } + \| \theta^{\gamma} - \bar \theta \|_{ H^{-1}(Q) } < \delta,
\end{align*}
then
$$
\begin{aligned}
& | \overline{\mathscr{F}}^{ \rm reg } (\bar u ; Q) - \F_{\gamma}^{ \rm reg } (u ; Q) | \\
& \quad \leq \alpha \left( \F_{\gamma} (u ; Q) + \left( \int_{Q} | \nabla u |^{2p} \right)^{\frac{1}{p}} \right) \\
& \quad \qquad + C_{\alpha} \F_{\gamma} (u ; Q) \left( \int_{Q} | (\phi_{Q}^{\gamma}, \sigma_{Q}^{\gamma}) |^2 + \| \theta^{\gamma} - \bar \theta \|_{ H^{-1}(Q) }^2 + | \bar a_{Q} - \bar a |^2 \right)^{\frac{1}{2}}. 
\end{aligned}
$$
Since $ - 1 \leq u, \bar u \leq 1 $, we replace $ ( \overline{\F}^{\rm reg}, \F_{\gamma}^{ \rm reg } ) $ by $ ( \overline{\F}, \F_{\gamma} ) $ on the l.h.s., so that the above claim implies the proposition.

Here comes the argument. First, we choose $ \varrho > 0 $ close to one such that the boundary layer is negligible. That is, we appeal to the estimate \eqref{eq:boundary-layer-energy} in combination with Meyer's estimate, cf.~ Lemma \ref{lemma:homogeneous-calderon-zygmund-estiamte}, to get
\begin{equation}\label{eqn:en-hom-04}
\begin{aligned}
&\overline{\mathscr{F}}^{ \rm reg } (\bar u ; Q \setminus \varrho Q)
+ \F_{\gamma}^{ \rm reg } (\bar u ; Q \setminus \varrho Q) \\
&\qquad \lesssim | Q \setminus \varrho Q |^{ 1 - \frac{1}{p} } \left( \F_{\gamma} (u ; Q) + \left( \int_{Q} | \nabla u |^{2p} \right)^{\frac{1}{p}} \right) \\
&\qquad \leq \frac{\alpha}{4} \left( \F_{\gamma} (u ; Q) + \left( \int_{Q} | \nabla u |^{2p} \right)^{\frac{1}{p}} \right).
\end{aligned}
\end{equation}
provided $ 0 < \varrho < 1 $ is sufficiently close to one.
Simultaneously, by choosing $ \varrho $ even closer to one if necessary, we may achieve that for every cut-off $ 0 \leq \eta \leq 1 $ with $ \eta |_{ \frac{\varrho}{2} Q  } = 1 $, we have
\begin{equation}\label{eqn:en-hom-03}
C \left( \int_{\varrho Q} (1 - \eta)^{2q}  \right)^{\frac{1}{q} } \leq C \left| Q \setminus \frac{\varrho}{2} Q \right|^{ \frac{1}{q} }\leq \frac{ \alpha }{ 2 } ,
\end{equation}
where $ C $ and $ q $ are chosen as in \eqref{eq:cut-off-two-scale-energy-difference}. Estimate \eqref{eqn:en-hom-03} will be important later on. For now, we combine \eqref{eqn:en-hom-04} with the fact that the cut-off two-scale expansion approximates the energy inside $ \varrho Q $.  This is made precise by \eqref{eq:cut-off-two-scale-expansion-energy}. Therefore, the last estimates and \eqref{eqn:en-hom-04} and \eqref{eq:cut-off-two-scale-expansion-energy} combine to
$$
\begin{aligned}
& | \overline{\mathscr{F}}^{ \rm reg } (\bar u ; Q) - \F_{\gamma}^{ \rm reg } ((1 + \eta \phi_{Q,e_i}^{\gamma} \partial_i) \bar u ; Q) | \\
&\qquad \leq \frac{\alpha}{4} \left( \F_{\gamma} (u ; Q) + \left( \int_{Q} | \nabla u |^{2p} \right)^{\frac{1}{p}} \right)\\
&\qquad\quad + C_{\eta} \overline{\mathscr{F}} (\bar u ; Q)  \left( \int_{Q} | (\phi_{Q}^{\gamma}, \sigma_{Q}^{\gamma}) |^2 + \| \theta^{\gamma} - \bar \theta \|_{ H^{-1}(Q) }^2 + | \bar {a} (Q) - \bar a |^2  \right)^{\frac{1}{2}}  \\
&\qquad\quad + \left( \int_{\varrho Q} | 1 - \eta |^q \right)^{\frac{1}{q}} \left( C_{\varrho}  \overline{\mathscr{F}} (\bar u; Q) + \left( \int_{Q} | \nabla \bar u |^{2p} \right)^{\frac{1}{p}} \right).
\end{aligned}
$$
By choosing the transition layer of $ \eta = \eta(\varrho) $ sufficiently thin in $ \varrho Q \setminus \frac{\varrho}{2} Q $, and appealing to Meyer's estimate, cf.~Lemma \ref{lemma:homogeneous-calderon-zygmund-estiamte}, and minimality of $ \bar u $ in form of the estimate provided in Proposition \ref{P: uniqueness comparison}, we obtain
\begin{equation}\label{eq:prop-cut-off-two-scale-conclusion-i}
\begin{aligned}
& | \overline{\mathscr{F}}^{ \rm reg } (\bar u; Q) - \F_{\gamma}^{ \rm reg } ((1 + \eta \phi_{Q,e_i}^{\gamma} \partial_i) \bar u; Q) | \\
&\qquad \leq \frac{\alpha}{2} \left( \F_{\gamma} (u; Q) + \left( \int_{Q} | \nabla u |^{2p} \right)^{\frac{1}{p}} \right) \\
&\qquad\quad + C_{\eta} \mathscr{F}_{\gamma} (u; Q)  \left( \int_{Q} | (\phi_{Q}^{\gamma}, \sigma_{Q}^{\gamma}) |^2 + \| \theta^{\gamma} - \bar \theta \|_{ H^{-1}(Q) }^2 + | \bar {a} (Q) - \bar a |^2  \right)^{\frac{1}{2}}.
\end{aligned}
\end{equation}
It is now left to approximate the energy of $ u $ by the cut-off two-scale expansion. More precisely, if $ \delta $ is made sufficiently small, by our choice of $ \eta $ and \eqref{eqn:en-hom-03}, we may use \eqref{eq:cut-off-two-scale-energy-difference} to get
\begin{equation}\label{eq:prop-cut-off-two-scale-conclusion-ii}
\begin{aligned}
& | \F_{\gamma}^{ \rm reg } ((1 + \eta \phi_{Q,e_i}^{\gamma} \partial_i) \bar u ; Q) - \F_{\gamma}^{ \rm reg } (u ; Q) | \\
&\qquad \leq \frac{\alpha}{2} \left( \F_{\gamma} (u ; Q) + \left( \int_{Q} | \nabla u |^{2p} \right)^{\frac{1}{p}} \right) \\
&\qquad\quad + C_{\eta} \mathscr{F}_{\gamma} (u ; Q)  \left( \int_{Q} | (\phi_{Q}^{\gamma}, \sigma_{Q}^{\gamma}) |^2 + \| \theta^{\gamma} - \bar \theta \|_{ H^{-1}(Q) }^2  + | \bar {a} (Q) - \bar a |^2 \right)^{\frac{1}{2}}.
\end{aligned}
\end{equation}
The estimates \eqref{eq:prop-cut-off-two-scale-conclusion-i} and \eqref{eq:prop-cut-off-two-scale-conclusion-ii} imply the desired conclusion.
\end{proof}

\subsection{Proof of PDE Estimates}\label{section:pde-estimates}

In this section we prove the main PDE estimates. We start with the proof of Lemma \ref{lemma:dg-n-m}, which is a consequence of the DeGeorgi-Nash-Moser theorem,~cf.~Theorem 8.24 in \cite{GilbargTrudinger}.

\begin{proof}[Proof of Lemma \ref{lemma:dg-n-m}]
We partition the cube $ Q_R (x)$ in sub cubes of size $ r_c $ as in \eqref{eqn:partition-of-small-cubes}. There is no loss of generality in assuming that $ Q $ is one of these cubes. We distinguish two cases: (i) $ \partial Q $ is contained in the interior of $ Q_R(x) $ or (ii) $ \partial Q $ intersects the boundary of $ Q_R(x) $.

In case $ \partial Q $ is contained in the interior of $ Q_R(x) $, we may directly appeal to Theorem 8.24 in \cite{GilbargTrudinger}, that yields
\begin{equation}\label{eqn:dg-n-m-02}
[ u ]_{ C^{0, \alpha}(Q) } \lesssim_{q} \left( \int_{ 2 Q }  | u |^2 \right)^{ \frac{1}{2} } + \left( \int_{ 2Q } | W'(u) |^q \right)^{ \frac{1}{q} }
\quad { \rm for } \quad q > \frac{d}{2}
\end{equation}
and some $ \alpha $ depending only on $ r_c $, the space dimension $ d $, the ellipticity constants $ \lambda, \Lambda $, and $ \theta^* $. The r.h.s.~of the this equation is uniformly bounded since $ -1 \le u \le 1 $ by Lemma \ref{L: minus one one}.

In case $ \partial Q $ intersects the boundary of $ Q_R(x) $, we use the boundary version of the DeGeorgi-Nash-Moser theorem, cf.~Theorem 8.29 in \cite{GilbargTrudinger} that comes with the estimate
\begin{equation}\label{eqn:dg-n-m}
[ u ]_{ C^{0, \alpha}(Q) } \lesssim_{q} \sup_{ 2 Q \cap Q_R(x) }  | u | + \left( \int_{ 2Q } | W'(u) |^q \right)^{ \frac{1}{q} } + \| q \|_{ C^{0,\alpha}( \R ) } 
\quad { \rm for } \quad q > \frac{d}{2},
\end{equation}
where $ \alpha $ has the same dependencies as above.
\end{proof}

Lemma  \ref{lemma:dg-n-m}, see also the condition \eqref{eqn:hom-int-cubes-assumption}, enters all proofs only implicitly through the clearing-out lemma, which we prove next.

\begin{proof}[Proof of Lemma \ref{lemma:clearing-out}]
Our argument is based on the fact that $ u $ satisfies \eqref{eqn:hom-int-cubes-assumption}. Upon changing the constants, the same estimate holds true for any minimizer $ \bar u $ of $ \overline{\mathscr{F}}(\cdot, Q) $ on $ u + H^1_0(Q) $. Indeed, we may appeal to Theorem 8.29 in \cite{GilbargTrudinger} to deduce an estimate similar to \eqref{eqn:dg-n-m} for $ \bar u $. Hence the following argument applies to both $ \bar u $ and $ u $. For simplicity let us restrict to the latter.

Let us fix $ \varepsilon > 0 $ and set $ A_{ \varepsilon } = \{ | u - 1 | \ge \varepsilon ~ \text{and} ~ | u + 1 | \ge \varepsilon \} \subset Q $. Due to the uniform Hölder continuity of the functions under consideration, there exists a radius $ r = r(\varepsilon) > 0 $ (that depends on the function only through the upper bound on the Hölder norm) such that
$$
B_r + ( Q \setminus A_{ \frac{ \varepsilon }{ 2 } } ) \subset \{ | u - 1 | < \varepsilon ~ \text{or} ~ | u + 1 | < \varepsilon \} \subset Q.
$$
From the continuity of $ W $, and the assumption that $ W $ attains its minimum only at $ -1 $ and $ +1 $, we know that
$$
| A_{ \frac{\varepsilon}{2} } | 
\lesssim_{\varepsilon} \int_{Q} W(u) < \delta .
$$
Hence if we choose $ \delta \le \delta_0 (\varepsilon) $, $  A_{ \frac{\varepsilon}{2} } $ contains no ball of radius $ r $. Phrased differently, that means that every ball of radius $ r $ intersects $ Q \setminus A_{ \frac{\varepsilon}{2} } $, so that
$$
Q
\subset B_r + \left( Q \setminus A_{ \frac{ \varepsilon }{ 2 } } \right)
\subset \{ | u - 1 | < \varepsilon ~ \text{or} ~ | u + 1 | < \varepsilon \},
$$
which by virtue of the continuity of $ u $ implies the claim.
\end{proof}

Next, we are going to prove Lemma \ref{lemma:meyers-estimate}, which corresponds to Meyer's estimate in our setting. We follow the classical strategy of establishing a reverse Hölder inequality, which enables us to appeal to the Gehring-type lemma of \cite[Theorem 3.1]{GM86} to conclude.

\begin{proof}[Proof of Lemma \ref{lemma:meyers-estimate}.]
For notational simplicity, we will denote by $ Q $ a cube of some size $ r > 0 $. In the conclusion (Step 3), we specialize to $ r = r_c $.

\textit{Step 1 (Caccioppoli's estimate).} We will need two slightly different versions of Caccioppoli's estimate: First, let us consider some cube $ Q $ of width $ r > 0 $ with $ 2Q \subset Q_R $. Then
\begin{equation}\label{eq:interior-caccioppoli}
\int_{Q} | \nabla u |^2
\lesssim \frac{1}{r^2} \int_{2Q} ( u - c )^2 + r^2 \int_{2Q} W'(u)^2
\end{equation}
for every constant $ c > 0 $. Additionally, we need a boundary version of \eqref{eq:interior-caccioppoli}. To this end, we consider a cube $ Q $ of width $ r > 0 $ such that $ 2Q $ is not fully contained in $ Q_R $. Then
\begin{equation}\label{eq:boundary-caccioppoli}
\int_{Q} \boldsymbol{1}_{Q_R} | \nabla u |^2
\lesssim \frac{1}{r^2} \int_{2 Q} \boldsymbol{1}_{Q_R} ( u - q )^2 + \int_{2 Q} \boldsymbol{1}_{Q_R} | \nabla q |^2 + r^2 \int_{2 Q} \boldsymbol{1}_{Q_R} W'(u)^2.
\end{equation}
Let us recall that $ q $ denotes the boundary value of $ u $, i.e.~$ u = q $ on $ \partial Q_R $.

The proofs of \eqref{eq:interior-caccioppoli} and \eqref{eq:boundary-caccioppoli} are rather standard. For the readers convenience, we will give an argument for \eqref{eq:boundary-caccioppoli}: Let $ \eta \ge 0 $ be a smooth cutoff that is to be chosen later. Observe that $ (u - q) \eta^2 $ vanishes on $ \partial Q_R $. Therefore we can test the Euler-Lagrange equation for $ u $ and obtain
\begin{equation}\label{eqn:cacc01}
\int_{ Q_R } a^{\gamma} \nabla u \cdot \nabla (u - q) \eta^2
= - \int_{ Q_R } \theta^{\gamma} W'(u) (u - q) \eta^2.
\end{equation}
The right hand side is estimated by
\begin{equation}\label{eqn:cacc02}
- \int_{ Q_R } \theta^{\gamma} W'(u) (u - q) \eta^2
\leq \frac{\theta^*}{2 r^2} \int_{ Q_R } (u - q)^2 \eta^2 + \frac{ \theta^* r^2}{2} \int_{Q_R} W'(u)^2 \eta^2,
\end{equation}
while on the left hand side we compute by Cauchy-Schwarz combined with Young's inequality
$$
\begin{aligned}
&\int_{ Q_R } a^{\gamma} \nabla u \cdot \nabla (u - q) \eta^2 \\
&\qquad = \int_{ Q_R } ( a^{\gamma} \nabla u \cdot \nabla u ) \eta^2
- \int_{ Q_R } ( a^{\gamma} \nabla u \cdot \nabla q ) \eta^2
+ 2 \int_{ Q_R } ( a^{\gamma}  \nabla u \cdot \nabla \eta ) (u - q) \eta \\
&\qquad \geq \frac{1}{2} \int_{ Q_R } ( a^{\gamma} \nabla u \cdot \nabla u ) \eta^2
- \frac{1}{2} \int_{ Q_R } ( a^{\gamma} \nabla q \cdot \nabla q ) \eta^2 \\
&\quad\qquad - 2 \int_{ Q_R } ( a^{\gamma} \nabla u \cdot \nabla u)^{\frac{1}{2}} (a^{\gamma} \nabla \eta \cdot \nabla \eta )^{\frac{1}{2}}  | u - q | \eta \\
&\qquad \ge \frac{1}{4} \int_{ Q_R } ( a^{\gamma} \nabla u \cdot \nabla u ) \eta^2
- \frac{1}{2}\int_{ Q_R } ( a^{\gamma} \nabla q \cdot \nabla q ) \eta^2
- 4 \int_{ Q_R } ( a^{\gamma} \nabla \eta \cdot \nabla \eta ) ( u - q )^2.
\end{aligned}
$$
Together with \eqref{eqn:cacc01} and \eqref{eqn:cacc02}, the last equation implies
$$
\begin{aligned}
\int_{ Q_R }  | \nabla u |^2 \eta^2
&\lesssim \int_{ Q_R }  \left( \frac{1}{ r^2}  \eta^2 + | \nabla \eta |^2 \right) (u - q)^2
+ \int_{ Q_R }  | \nabla q |^2  \eta^2 \\
&\quad\quad + r^2 \int_{ Q_R }  W'(u)^2 \eta^2.
\end{aligned}
$$
Choosing $ \eta $ with $ \eta = 1 $ on $ Q $, $ \eta = 0 $ outside $ 2 Q $ and $ | \nabla \eta | \lesssim \frac{1}{r} $ yields \eqref{eq:boundary-caccioppoli}. 

\textit{Step 2 (Reverse Hölder inequality).} We now use \eqref{eq:interior-caccioppoli} and \eqref{eq:boundary-caccioppoli} to show two kinds of reverse Hölder inequalities: On the one hand, given a cube $ Q $ of width $ r > 0 $ such that $ 2Q \subset Q_R $, we claim that
\begin{equation}\label{eq:interior-reverse-hoelder}
\left( \fint_{ Q } | \nabla u |^2 \right)^{\frac{1}{2}}
\lesssim \left( \fint_{ 2Q } |\nabla u|^p \right)^{\frac{1}{p}}
+ r \left( \fint_{ 2Q } W'(u)^2 \right)^{\frac{1}{2}} , 
\quad { \rm where } \quad p = \frac{2d}{d + 2} .
\end{equation}
On the other hand, if $ 2Q $ is not fully contained in $ Q_R $, we instead claim that
\begin{equation}
\begin{aligned}\label{eq:boundary-reverse-hoelder}
\left( \fint_{Q} \boldsymbol{1}_{ Q_R } | \nabla u |^2 \right)^{\frac{1}{2}}
&\lesssim \left( \fint_{2 Q} \boldsymbol{1}_{ Q_R } | \nabla u |^p \right)^{\frac{1}{p}}
+ \left( \fint_{2 Q} \boldsymbol{1}_{ Q_R } | \nabla q |^2 \right)^{\frac{1}{2}} \\
&\quad\quad + r \left( \fint_{2 Q} \boldsymbol{1}_{ Q_R } W'(u)^2 \right)^{\frac{1}{2}} ,
\quad { \rm where } \quad p = \frac{2d}{d + 2} . 
\end{aligned}
\end{equation}
As in the last step, we only give an argument for the second claim. The first inequality follows similarly.

Let us recall the following version of the scale-invariant Poincar\'{e}-Sobolev inequality: if $p = \frac{2d}{d + 2}$ and $f \in W^{1,p}(2Q)$ is such that $|\{f = 0\}| \geq \alpha |2Q|$ for some $\alpha \in (0,1)$, then
$$
\left( \int_{2 Q} f^2 \right)^{\frac{1}{2}}
\lesssim_{\alpha} \left( \int_{2 Q} |\nabla f|^p \right)^{\frac{1}{p}}.
$$
(See, e.g., \cite[pp.~153]{giaquinta_book}.)  Applying this estimate to the function $f = (u - q) \mathbf{1}_{Q_{R}}$, which satisfies $|\{f = 0\}| \geq |(2Q) \setminus Q_{R}| \geq 2^{-1} |Q|$, we find
$$\left( \int_{2Q} \boldsymbol{1}_{Q_{R}} (u - q)^{2} \right)^{\frac{1}{2}} \lesssim \left( \int_{2Q} \boldsymbol{1}_{Q_{R}} |\nabla u - \nabla q|^{p}  \right)^{\frac{1}{p}}.
$$

It is now left to observe that by our choice of $ p $,
$$
r^{-1} r^{ - \frac{d}{2} } = r^{- \frac{2 + d}{2}} = r^{- \frac{d}{p} }
\quad \text{and} \quad p < 2
$$
so that the above inequality and \eqref{eq:boundary-caccioppoli} combine to yield \eqref{eq:boundary-reverse-hoelder}.

\textit{Step 3 (Conclusion).} To conclude, we now use the version of Gehring's lemma from \cite[Theorem 3.1]{GM86}. Again, we only present the argument for the boundary estimate.

First note that \eqref{eq:boundary-reverse-hoelder}, as stated above, holds for every cube. Indeed, \eqref{eq:interior-reverse-hoelder} implies \eqref{eq:boundary-reverse-hoelder} for interior cubes. Therefore, Gehring's lemma yields the estimate
\begin{equation}\label{eqn:meyer01}
\begin{aligned}
\left( \fint_{Q} \boldsymbol{1}_{ Q_R } | \nabla u |^p \right)^{\frac{1}{p}}
&\lesssim \left( \fint_{2Q} \boldsymbol{1}_{ Q_R } |\nabla u |^2 \right)^{\frac{1}{2}}
+ \left( \fint_{2Q} \boldsymbol{1}_{ Q_R } |\nabla q|^p \right)^{\frac{1}{p}} \\
&\quad\quad + \left( \fint_{2Q} \boldsymbol{1}_{ Q_R } |W'(u)|^p \right)^{\frac{1}{p}}
\end{aligned}
\end{equation}
for some $ p = p(d) > 2 $ for any cube $ Q $ of size $ r_c $. We apply this to a cube that satisfies the assumptions of Lemma \ref{lemma:meyers-estimate}. That is, $ Q $ has size $ r_c $ and is either $ 2 Q \subset Q_R $ or $ Q \subset Q_R $ with $ \partial Q \cap \partial Q_R \neq \emptyset $. Since we only work with cubes of size $ r_c $ from now on, we can drop the averages in \eqref{eqn:meyer01}.

To conclude the lemma, we need to post-process the last term on the r.h.s.~of \eqref{eqn:meyer01}. To this end, we distinguish the two cases $ \F_{\gamma}(u,2Q \cap Q_R ) \ge \delta $ and $ \F_{\gamma}(u,2Q \cap Q_R ) < \delta $ where $ \delta > 0 $ is chosen according to the clearing-out lemma, cf.~Lemma \ref{lemma:clearing-out}.

If $ \F_{\gamma}(u,2Q \cap Q_R ) < \delta $, we may assume w.l.o.g.~that $ | u - 1 | \ll 1 $ in $ 2 Q \cap Q_R $. In particular, by virtue of the non-degeneracy assumption \eqref{eqn:w-non-degeneracy} in form of \eqref{eqn:w-poly}, we know that
$$
W(u) \sim | u - 1 |^{2 \kappa }
\quad \text{and} \quad
W'(u) \sim | u - 1 |^{2 \kappa - 1}
\quad { \rm in } ~ 2 Q \cap Q_R 
$$
if $ \delta $ is chosen small enough. Furthermore, we may assume that $ p > 2 $ is close to two. Hence, by Sobolev embedding applied to $ ( u - 1 )^{2 \kappa - 1} $, we get
$$
\left( \int_{2Q \cap Q_R} | W'(u) |^p \right)^{\frac{1}{p}}
\lesssim_{ p } \left( \int_{2Q \cap Q_R} | (u-1)^{2 \kappa -1} |^2 + | \nabla (u-1)^{2 \kappa - 1} |^2 \right)^{\frac{1}{2}}.
$$
But since $ | u - 1 | \ll 1 $ and $ 2(2 \kappa - 1) \ge 2 \kappa $, we can conclude
\begin{equation}\label{eqn:meyer02}
\left( \int_{2Q \cap Q_R } | W'(u) |^p \right)^{\frac{1}{p}}
\lesssim \left( \int_{2Q \cap Q_R } | \nabla u |^2 + W(u) \right)^{\frac{1}{2}}
\lesssim \mathscr{F}_{ \gamma }( u, 2Q \cap Q_R ) ^{\frac{1}{2}}.
\end{equation}
Note that if $ \F_{\gamma}(u,2Q \cap Q_R ) \ge \delta $, we may appeal to the boundedness of $ W'(u) $ to see that the above estimate is still true. In combination with \eqref{eqn:meyer01}, \eqref{eqn:meyer02} yields our claim. 
\end{proof}

\begin{proof}[Proof of Lemma \ref{lemma:homogeneous-calderon-zygmund-estiamte}.]
As announced earlier, the same proof as for Lemma \ref{lemma:meyers-estimate} applies: the Caccioppli inequalites \eqref{eq:interior-caccioppoli} and \eqref{eq:boundary-caccioppoli} also hold for $ (u,q) $ replaced by $ (\bar u, u) $, so that the conclusions from Step 2 and 3 within the proof of Lemma \ref{lemma:meyers-estimate} also apply.
\end{proof}

Finally, we establish local Schauder estimates for the constant-coefficient equation. Note that the assumption on the boundary conditions is redundant but simplifies the proof. We use it to appeal to the clearing-out property, see Lemma \ref{lemma:clearing-out}.

\begin{proof}[Proof of Lemma \ref{lemma:homogeneous-interior-schauder-estimate}.]
Fix $ 0 < \varrho < \varrho' < \varrho'' < 1 $. The De Giorgi-Nash-Moser estimate, cf.~Theorem 8.24 in \cite{GilbargTrudinger}, on $ \bar u - \fint_{Q} \bar u $ yields
$$
\begin{aligned}
\sup_{\varrho' Q} | \bar u - \fint_{Q} \bar u | + [\bar u]_{C^{0,\alpha}(\varrho' Q)}
&\lesssim_{\varrho', \varrho''} \left( \int_{ \varrho'' Q} | \bar u - \fint_{Q} \bar u |^2 \right)^{\frac{1}{2}} + \sup_{ \varrho'' Q } | W'(\bar u) | \\
&\lesssim \left( \int_{Q} | \nabla \bar u |^2 \right)^{\frac{1}{2}} + \sup_{ \varrho'' Q } | W'(\bar u) |
\end{aligned}
$$
for some $ 0 < \alpha < 1 $. Furthermore, the classical Schauder estimate, cf.~Corollary 6.3 in \cite{GilbargTrudinger}, applied to $ \bar u - \fint_{Q} \bar  u $ reads
$$
\sup_{ \varrho Q } ( | \nabla \bar u | + | \nabla^2 \bar u | )
\lesssim_{\varrho, \varrho'} \sup_{\varrho' Q} ( | \bar u - \fint_{Q} \bar u | + | W'(\bar u) | ) + [ W'(\bar u) ]_{C^{0,\alpha}(\varrho' Q) }.
$$
Since $ W $ is $ C^2 $ in a neighborhood of $ [-1,1] $, the estimates combine to
\begin{equation}\label{eqn:schauder01}
\sup_{ \varrho Q } ( | \nabla \bar  u | + | \nabla^2 \bar  u | )
\lesssim_{\varrho, \varrho', \varrho''} \left( \int_{Q} | \nabla \bar u |^2 \right)^{\frac{1}{2}} + \sup_{\varrho'' Q} | W'(\bar u) |.
\end{equation}
To conclude Lemma \ref{lemma:homogeneous-interior-schauder-estimate} we will perform a bootstrap argument on $ W'(u) $ to obtain the estimate
\begin{equation}\label{eqn:moser01}
\sup_{\varrho'' Q} | W'(\bar u) | \lesssim_{\varrho''} \left( \int_{Q} | \nabla \bar u |^2 + W(\bar u) \right)^{\frac{1}{2}}.
\end{equation}
Evidently, \eqref{eqn:schauder01} and \eqref{eqn:moser01} imply the claim.

\medskip

We now set up the iteration that leads to \eqref{eqn:moser01}: Let us first assume that $ \F_{\gamma}(\bar u,Q) \ll 1 $. Hence we can further assume w.l.o.g.~that
\begin{equation}\label{eqn:moser04}
W'(\bar u) \sim | \bar u - 1 |^{ 2 \kappa - 1 }
\quad { \rm in } ~ Q,
\end{equation}
see \eqref{eqn:w-poly}. We select a finite number of radii $ \varrho'' = r_{N+1} < r_{N} < \hdots < r_1 < r_0 = 1 $ where $ N $ denotes the integer satisfying $ N \leq \frac{d}{2} < N + 1 $. Recall the interior Calder\'{o}n-Zygmund estimate, cf.~Theorem 3.7 in \cite{T87},
$$
\left( \int_{r_k Q} | \nabla \bar u |^p \right)^{\frac{1}{p}}
\lesssim_{p, k} \left( \int_{r_{k-1} Q} | \nabla \bar u |^2 + \bar u^2 \right)^{\frac{1}{2}} + \left( \int_{r_{k-1} Q} | W'(\bar u) |^p \right)^{\frac{1}{p}}
$$
for $ 2 < p < \infty $. Since our equation is in divergence form, we may pass to $ u - \fint_{r_{k-1} Q} u $ and obtain the estimates
\begin{equation}\label{eqn:moser02}
\left( \int_{r_{k} Q} | \nabla \bar u |^{ {p^*}^{(k-1)} } \right)^{\frac{1}{ {p^*}^{(k-1)} }}
\lesssim_{p, k} \left( \int_{Q} | \nabla \bar u |^2 \right)^{\frac{1}{2}} + \left( \int_{r_{k-1} Q} | W'(\bar u) |^{ {p^*}^{(k-1)} } \right)^{\frac{1}{ {p^*}^{(k-1)} }}.
\end{equation}
Here the sequence of exponents ${ p^* }^{ (k) }$ are determined recursively so that the exponent $ { p^* }^{ (k) } $ equals the critical Sobolev exponent for $ { p^* }^{ (k - 1) } $ when $ k > 0 $ and the initial value ${ p^* }^{ (0) }  = p \in \R \setminus \mathbb{Q} $ is chosen such that $ \frac{2d}{d+2} < p < 2 $ (the former meaning that $ { p^* }^{ (1) } > 2 $) and $ p (N + 1) > d $ (which is possible since $ d < 2 (N+1) $). Note that $ { p^* }^{(k)} = \frac{ pd }{ d - pk } $.

We observe that since $ - 1 \le \bar u \le 1 $, the Sobolev inequality implies
\begin{equation}\label{eqn:moser03}
\begin{aligned}
& \left( \int_{r_k Q} ( | \bar u - 1 |^{2 \kappa - 1} )^{ {p^*}^{(k)} } \right)^{\frac{1}{ {p^*}^{(k)} }} \\
& \qquad \lesssim_{k} \left( \int_{r_{ k } Q} | \nabla \bar u |^{ \max \{  {p^*}^{(k-1)}, 2 \} } + \int_{r_{ k } Q}  ( | \bar u - 1 |^{2 \kappa - 1} )^{ \max \{  {p^*}^{(k-1)}, 2 \} } \right)^{\frac{1}{ \max \{  {p^*}^{(k-1)}, 2 \} }}.
\end{aligned}
\end{equation}
Since $ p^{*(k)} > 2 $ for each $k \geq 1$, the equations \eqref{eqn:moser02} and \eqref{eqn:moser03} combine with $ | W'(u) | \sim | u - 1 |^{ 2 \kappa - 1 } $, see \eqref{eqn:moser04},  to
\begin{align*}
&\left( \int_{r_k Q} ( | \bar u - 1 |^{ 2 \kappa - 1 } )^{ {p^*}^{(k)} } \right)^{\frac{1}{ {p^*}^{(k)} }} \\
&\qquad \lesssim_{k} \left( \int_{Q} | \nabla \bar u |^2 \right)^{\frac{1}{2}} + \left( \int_{r_{k-1} Q}  ( | \bar u - 1 |^{ 2 \kappa - 1 } )^{ \max \{  {p^*}^{(k-1)}, 2 \}  } \right)^{\frac{1}{ \max \{  {p^*}^{(k-1)}, 2 \}  }}.
\end{align*}
We iterate this equation finitely many times (hence the $ k $ dependent constant does not blow up) and use the condition $ p(N+1) > d $ (on the first inequality) and \eqref{eqn:moser02} (on the second inequality) to obtain
$$
\begin{aligned}
\sup_{r_{N+1} Q}  | \bar u - 1 |^{ 2 \kappa - 1 }
&\lesssim_{ p, N }  \left( \int_{r_{N+1} Q} | \nabla \bar u |^{ {p^*}^{(N)} } + \int_{r_{N+1} Q}  ( | \bar u - 1 |^{ 2 \kappa - 1 } )^{ {p^*}^{(N)} } \right)^{\frac{1}{ {p^*}^{(N)} }} \\
&\lesssim_N \left( \int_{Q} | \nabla \bar u |^{ 2 } \right)^{\frac{1}{2}} + \left( \int_{r_N Q}  ( | \bar u - 1 |^{ 2 \kappa - 1 } )^{ {p^*}^{(N)} } \right)^{\frac{1}{ {p^*}^{(N)} }} \\
&\lesssim \left( \int_{Q} | \nabla \bar u |^{ 2 } \right)^{\frac{1}{2}} + \left( \int_{Q}  ( | \bar u - 1 |^{ 2 \kappa - 1 }  )^{ 2 } \right)^{\frac{1}{2}}.
\end{aligned}
$$
Since $ 2 ( 2 \kappa - 1 ) \ge 2 \kappa  $, and $ W(u) \sim | u - 1 |^{ 2 \kappa  } $ by virtue of \eqref{eqn:moser04}, we conclude
\begin{equation}\label{eqn:moser-conclusion}
\begin{aligned}
\sup_{\varrho Q} | W'(\bar u) |
&\lesssim \left( \int_{Q} | \nabla \bar u |^{ 2 } \right)^{\frac{1}{2}} + \left( \int_{Q} | \bar u - 1 |^{ 2 \kappa } \right)^{\frac{1}{2}} \\
&\lesssim \left( \int_{Q} | \nabla \bar u |^{ 2 } + \int_{Q} W(\bar u) \right)^{\frac{1}{2}}.
\end{aligned}
\end{equation}
As usual, we observe that this estimate is also true if $ \F_{\gamma}(u,Q) \ge C $ for some uniform constant $ C > 0 $ since the left hand side is always bounded.
\end{proof}

\section{Averaging} \label{S: homogenization 2}

This section completes the proof of Theorem \ref{T: planar homogenization theorem lower bound}.  In this random setting, independent of the cube $Q$, we define 
	\begin{equation}
		(\phi_{Q},\sigma_{Q},\bar{a}(Q)) = (\phi,\sigma,\bar{a}),
	\end{equation} 
where $(\phi,\sigma)$ are the infinite-volume correctors and $\bar{a}$ is the homogenized matrix, as in the introduction. 

The results of the previous section showed that the energy of a minimizer $u$ in the plane-like cell problem \eqref{E: planar limit we want} can be approximated by the homogenized energy of a suitable competitor provided we are able to control the sublinear growth of the correctors $(\phi,\sigma)$ associated with the operator $-\nabla \cdot (a \nabla )$ and the oscillations of $\theta$.  In view of our definition of $(\phi_{Q},\sigma_{Q})$ and $\bar{a}(Q)$ in this context, this control is phrased in terms of the stationary fields ${ \rm Sub }(\cdot)$ and ${ \rm Osc }(\cdot)$ defined by 
	\begin{align}
		{ \rm Sub }_{x}(r) &= \sup_{ \varrho \ge r } \frac{1}{\varrho} \left( \fint_{Q_{\varrho}(x)} | (\phi(y),\sigma(y)) - \fint_{Q_{ \varrho } ( x ) } (\phi, \sigma) |^{2} \, dy \right)^{\frac{1}{2}}, \label{E: sub defined again} \\
		{ \rm Osc }_{x}(r) &= \sup_{ \varrho \ge r }  \varrho^{-1} \|\theta - \bar{\theta}\|_{H^{-1}(Q_{\varrho}(x))}. \label{E: osc defined again}
	\end{align}  
In this section, we prove that the needed control over these quantities is indeed furnished by the assumptions \eqref{E: sublinear part} and \eqref{E: averaging part}.  
	
In addition to completing the proof of the lower bound in Theorem \ref{T: homogenization theorem}, we briefly recall the optimal decay rate of ${ \rm Sub }(\cdot)$ and ${ \rm Osc }(\cdot)$ for the random checkerboard.  This last part is expository and intended for readers who may not be experts in quantitative stochastic homogenization.

\subsection{Preliminaries from Homogenization Theory} \label{S: homogenization preliminaries} For the reader's convenience, we recall the relevant results from homogenization theory that will be needed in what follows.  In particular, as we review next, the assumptions of Section \ref{S: assumptions} imply that, on an event of probability one, all of the structural assumptions of Section \ref{S: deterministic assumptions} hold.  Further, we recall the fact that the stationary fields ${ \rm Sub }(R)$ and ${ \rm Osc}(R)$ vanish as $R \to \infty$.  

\subsubsection{Helmholtz-type Decomposition.} First, the assumptions of Section \ref{S: assumptions} imply that there is a matrix $\bar{a}$, a random vector field $\phi$, and a random $3$-tensor $\sigma$ such that, with probability one, for any $\xi \in \mathbb{R}^{d}$, the function $\phi_{\xi} = \xi_{i} \phi_{i}$ and skew-symmetric matrix field $\sigma_{\xi} = \xi_{i} \sigma_{i}$ relate to $a(y) \xi$ via the formula
	\begin{equation} \label{E: helmholtz again}
		a(y) \xi = \bar{a} \xi - a(y) \nabla \phi_{\xi}(y) + (\nabla \cdot \sigma_{\xi})(y).
	\end{equation} 
Furthermore, the gradient fields $\nabla \phi$ and $\nabla \sigma$ are stationary with mean zero.  For a more precise statement and proof, we refer the reader to the short proof-sketch in \cite[Section 1.2]{JosienOtto} or, alternatively, the proof of \cite[Lemma 1]{gloria_neukamm_otto}.


\subsubsection{Bounds on $\bar{a}$ and $\bar{\theta}$.} \label{S: bound on homogenized matrix} Concerning the constant $\bar{a}$, recall that, for any $\xi \in \mathbb{R}^{d}$,
	\begin{align*}
		\xi \cdot \bar{a} \xi &= \mathbb{E}[ a(y) (\xi + \nabla \phi_{\xi}(y)) \cdot (\xi + \nabla \phi_{\xi}(y)) \\
			&= \min_{\nabla \psi} \mathbb{E}[ a(y) (\xi + \nabla \psi(y)) \cdot (\xi + \nabla \psi(y)) ],
	\end{align*}
where the minimum is over stationary gradient fields; see, for instance, \cite[Section 7.2]{jikov_kozlov_oleinik}.  Since $\lambda \text{Id} \leq a \leq \Lambda \text{Id}$ pointwise by the assumptions of Section \ref{S: assumptions}, and the minimizer is $\nabla \psi \equiv 0$ if $a$ is replaced by a constant, we deduce that $\lambda \text{Id} \leq \bar{a} \leq \Lambda \text{Id}$.

Recall from the introduction that we define $\bar{\theta} = \mathbb{E}[\theta(y)]$.  Since the pointwise bounds $\theta_{*} \leq \theta \leq \theta^{*}$ are imposed in Section \ref{S: assumptions}, it follows that $\theta_{*} \leq \bar{\theta} \leq \theta^{*}$. 

\subsubsection{Sublinearity and Oscillations} Next, let us recall that, since the action $(\tau_{x})_{x \in \mathbb{R}^{d}}$ is stationary and ergodic, we know that, for any $\nu > 0$,
	\begin{gather}
		\lim_{r \to \infty} \mathbb{P} \left\{ { \rm Sub }_{0}(r) > \nu \right\} = 0, \label{E: sublinearity averaging section} \\
		\lim_{r \to \infty} \mathbb{P} \left\{ { \rm Osc }_{0}(r) > \nu \right\} = 0. \label{E: averaging theta averaging section}
	\end{gather}
The first is classical and follows from the fact that $\nabla \phi$ is stationary and mean-zero, see, for instance, \cite[Section 3.4]{gloria_neukamm_otto} for the proof.  The second follows directly from the fact that, by the ergodic theorem, $\theta(\gamma^{-1} \cdot) \rightharpoonup \bar{\theta}$ weakly in $L^{2}_{\text{loc}}(\mathbb{R}^{d})$ as $\gamma \to 0$.  For the fact that weak convergence follows from the ergodic theorem, see \cite[Section 7.1]{jikov_kozlov_oleinik}.  Alternatively, both this and the (strong) $H^{-1}$-convergence are stated and proved in detail in \cite[Corollary 1.9]{armstrong_kuusi_book}.

\subsection{Dilatational and Rotational Invariance}\label{section:dil-rot-invariance} 

In Section \ref{S: homogenization 1}, we proved that the energy in the planar cell problems \eqref{E: planar limit we want} converges provided the phase transition occurs across the hyperplane with normal vector $e_{1}$.  In fact, the argument applies if $e_{1}$ is replaced by any unit vector $e \in S^{d-1}$.  To see that, it is necessary to argue that the assumptions \eqref{E: sublinear part} and \eqref{E: averaging part} remain true after rotation, as asserted in Proposition \ref{P: rotated fields}.

It will first be useful to verify that the assumptions are invariant under dilations rather than rotations.  That is the aim of the next proposition.

\begin{prop} \label{P: dilational invariance} If \eqref{E: sublinear part} and \eqref{E: averaging part} hold, then, for any $\alpha, \nu > 0$, 
		\begin{align}
			\lim_{\epsilon \to 0} \epsilon^{-d} \mathbb{P} \left\{ {\rm Sub }_{0} \left( \frac{\alpha \epsilon}{\delta(\epsilon)} \right) > \nu \right\} &= 0, \label{E: sublinear part dilated}\\
			\lim_{\epsilon \to 0} \epsilon^{-d} \mathbb{P} \left\{ {\rm Osc}_{0} \left( \frac{\alpha \epsilon}{\delta(\epsilon)} \right) > \nu \right\} &= 0. \label{E: averaging part dilated}
		\end{align}
	\end{prop}
	
		\begin{proof} If $\alpha \geq 1$, then this is trivial since, for any function $F$, $\max_{r \geq S} F(r)$ decreases as a function of $S$.  Thus, fix $\alpha < 1$ from here on.
		
		Given any $r > 0$, observe that 
			\begin{align*}
				\frac{1}{ (\alpha r)^{2} } \fint_{Q_{\alpha r}} |(\phi,\sigma ) - \fint_{ Q_{ \alpha r } } (\phi,\sigma) |^{2}
				&\leq \frac{1}{ (\alpha r)^{2} } \fint_{Q_{\alpha r}} |(\phi,\sigma ) - \fint_{ Q_r } (\phi,\sigma) |^{2} \\
				&\leq \frac{1}{\alpha^{2 + d} r^{2}} \fint_{Q_{r}} | (\phi, \sigma) - \fint_{ Q_r } (\phi,\sigma) |^{2},
			\end{align*}
		which after taking the supremum $ r \geq R $ shows $ { \rm Sub }_0 ( \alpha R ) \leq \alpha^{ - \frac{d + 2}{2} } { \rm Sub }_0 ( R ) $.	
		Thus, setting $ R = \delta^{-1} \epsilon$, we find
			\begin{equation*}
				\mathbb{P} \left\{ {\rm Sub}_{0} \left( \frac{\alpha \epsilon}{\delta} \right) > \nu \right\} \leq \mathbb{P} \left\{ {\rm Sub}_{0} \left( \frac{\epsilon}{\delta} \right) > \alpha^{\frac{d + 2}{2}} \nu \right\},
			\end{equation*}
		which proves \eqref{E: sublinear part dilated} follows from \eqref{E: sublinear part}.
		
		A completely analogous argument applies to $ { \rm Osc }_0 $.  Indeed, observe that
			\begin{equation*}
				\|f\|_{H^{-1}(Q_{\alpha r})} \leq \alpha^{-\frac{d}{2}} \|f\|_{H^{-1}(Q_{r})} \quad \text{for each} \quad f \in L^2_{ \rm loc } ( \R^d).
			\end{equation*}
		To see this, let us fix $v \in H^{1}_{0}(Q_{\alpha r})$, so that the function $ \tilde{v} = \alpha^{ - \frac{d}{2} } v \boldsymbol{1}_{Q_{\alpha r}} $ is in $H^{1}_{0}(Q_{r})$ and satisfies
		\begin{align*}
		\fint_{Q_{r}} | \nabla \tilde{v} |^{2} = \fint_{Q_{\alpha r}} | \nabla v |^{2}
		\quad { \rm and } \quad
		\fint_{ Q_{ \alpha r } } f v = \alpha^{ - \frac{d}{2} } \fint_{ Q_r } f \tilde v.
		\end{align*}
		Thus,
		\begin{align*}
		\| f \|_{ H^{-1} ( Q_{ \alpha r } ) } \leq \alpha^{ - \frac{d}{2} } \| f \|_{ H^{-1}( Q_r ) },
		\quad { \rm and ~ therefore } \quad
		{ \rm Osc }_{ 0 } ( \alpha R )  \leq \alpha^{ - \frac{d}{2} } { \rm Osc }_{ 0 } ( R ).
		\end{align*}
		%
		As in the case of ${\rm Sub}_{0}(\cdot)$, we conclude that \eqref{E: averaging part dilated} follows from \eqref{E: averaging part}.
		\end{proof}

Finally, we prove rotational invariance.  To make the discussion precise, it will be convenient to introduce some notation.  Given any orthogonal transformation $\mathcal{O}$ and $r > 0$, define $Q^{\mathcal{O}}_{r}$ by 
	\begin{equation} \label{E: rotated cube}
		Q^{\mathcal{O}}_{r} = \mathcal{O}(Q_{r}).
	\end{equation}
We denote by ${\rm Sub}^{\mathcal{O}}(\cdot)$ and ${\rm Osc}^{\mathcal{O}}(\cdot)$ the functions defined analogously to ${\rm Sub}(\cdot)$ and ${\rm Osc}(\cdot)$, see \eqref{E: sub defined again} and \eqref{E: osc defined again}, but with rotated cubes in place of the standard ones.

	\begin{prop} \label{P: rotational invariance} If the assumptions \eqref{E: sublinear part} and \eqref{E: averaging part} hold, then, for any orthogonal transformation $\mathcal{O}$ of $\mathbb{R}^{d}$ and any $\alpha > 0$,
		\begin{align}
			\lim_{\epsilon \to 0} \epsilon^{-d} \mathbb{P} \left\{ { \rm Sub }^{\mathcal{O}}_{0} \left( \frac{\alpha \epsilon}{\delta(\epsilon)} \right) > \nu \right\} &= 0, \label{E: sublinear part rotated}\\
			\lim_{\epsilon \to 0} \epsilon^{-d} \mathbb{P} \left\{ {\rm Osc}^{\mathcal{O}}_{0} \left( \frac{\alpha \epsilon}{\delta(\epsilon)} \right) > \nu \right\} &= 0. \label{E: averaging part rotated}
		\end{align}
	\end{prop}
	
		\begin{proof} Fix an orthogonal transformation $\mathcal{O}$ and let $r > 0$.  Since $Q^{\mathcal{O}}_{r}$ contains the origin in its interior, there is a $C > 1$, which is independent of $r$, such that $Q_{C^{-1} r} \subseteq Q^{\mathcal{O}}_{r} \subseteq Q_{Cr}$.  Thus, arguing as in the previous proposition, we deduce $ { \rm Sub }_0^{ \mathcal{O} } ( R ) \leq C^{ \frac{d}{2} + 1} { \rm Sub }_0 ( R ) $, so that the previous proposition and \eqref{E: sublinear part} imply \eqref{E: sublinear part rotated}. 
		
		Similarly, arguments similar to those in the previous proof show that if $f \in L^{2}(Q_{Cr})$, then 
			\begin{equation*}
				\|f\|_{H^{-1}(Q^{\mathcal{O}}_{r})} \leq C^{ \frac{d}{2} }  \|f\|_{H^{-1}(Q_{Cr})}
				\quad { \rm and ~ thus } \quad
				{ \rm Osc }_0^{ \mathcal{O} } (r) \leq C^{ \frac{d}{2} }  { \rm Osc }_0 ( C r ) .
			\end{equation*}
		As before, \eqref{E: averaging part} implies \eqref{E: averaging part rotated}.  \end{proof}
		
Finally, here is the proof of Proposition \ref{P: rotated fields} for completeness:

	\begin{proof}[Proof of Proposition \ref{P: rotated fields}]  Fix an orthogonal transformation $\mathcal{O} \in O(d)$.  It is immediate to check that the rotated coefficients $a^{\mathcal{O}}$ and $\theta^{\mathcal{O}}$ satisfy the assumptions of Section \ref{S: assumptions} with the action $(\tau_{x})_{x \in \mathbb{R}^{d}}$ replaced by the rotated action $(\tau^{\mathcal{O}}_{x})_{x \in \mathbb{R}^{d}}$.
	
	By the previous proposition (Proposition \ref{P: rotational invariance}), the rotated medium $(a^{\mathcal{O}},\theta^{\mathcal{O}})$  satisfies the assumptions \eqref{E: sublinear part} and \eqref{E: averaging part} if and only if $(a,\theta)$ does. 
	
	Finally, by manipulating the variational formula in Section \ref{S: bound on homogenized matrix}, one readily deduces that $\bar{a}^{\mathcal{O}} = \mathcal{O}^{-1} \bar{a} \mathcal{O}$.  Further, we immediately deduce that $\bar{\theta}^{\mathcal{O}} = \mathbb{E}[\theta^{\mathcal{O}}(0)] = \mathbb{E}[\theta(0)] = \bar{\theta}$. \end{proof}

\subsection{Proof of Theorem \ref{T: planar homogenization theorem lower bound}} \label{S: homogenization proof} In this section, we complete the proof of Theorem \ref{T: planar homogenization theorem lower bound}.  To get an idea why assumptions \eqref{E: sublinear part} and \eqref{E: averaging part} are the correct ones, recall from \eqref{eqn:length-scales-trafo} that, writing $R = \epsilon^{-1}$ and $\gamma(R) = \delta(\epsilon) \epsilon^{-1}$ for the macroscopic and microscopic length scales after mesoscopic rescaling, we can invoke dilatational invariance (Proposition \ref{P: dilational invariance}) to find
	\begin{equation} \label{E: reformulation of main homogenization assumptions}
		\lim_{R \to \infty} R^{d} \, \mathbb{P}\{{ \rm Sub }_{0}(r_{c} \gamma(R)^{-1}) > \nu\} = 0,
		\quad \lim_{R \to \infty} R^{d} \, \mathbb{P}\{{ \rm Osc }_{0}(r_{c} \gamma(R)^{-1}) > \nu\} = 0.
	\end{equation}
By ergodicity, if $E_{R},F_{R} \subseteq Q_{R}$ are the subsets defined by 
	\begin{equation*}
		E_{R} = \{x \in Q_{R} \, \mid \, { \rm Sub }_{x}(r_{c} \gamma(R)^{-1}) > \nu\}, \quad
		F_{R} = \{x \in Q_{R} \, \mid \, { \rm Osc }_{x}(r_{c} \gamma(R)^{-1}) > \nu\},
	\end{equation*}
then
	\begin{equation*}
		|E_{R}| \approx R^{d} \, \mathbb{P}\{{ \rm Sub }_{0}(\gamma(R)^{-1}) > \nu\} \quad \text{and} \quad
		|F_{R}| \approx R^{d} \, \mathbb{P}\{{ \rm Osc}_{0}(\gamma(R)^{-1}) > \nu\},
	\end{equation*}
and, thus, \eqref{E: reformulation of main homogenization assumptions} suggests that $E_{R}$ and $F_{R}$ are negligible.  For all intents and purposes, this puts us in a situation in which we have uniform control of the homogenization error in the entire box $Q_{R}$ (cf.\ \eqref{E: what we want to prove averaging}).

The proof given next closely follows the previous heuristic discussion.

	\begin{proof}[Proof of Theorem \ref{T: planar homogenization theorem lower bound}] Suppose that the scale $\epsilon \mapsto \delta(\epsilon)$ satisfies \eqref{E: sublinear part} and \eqref{E: averaging part}.  We want to show that, for any $\varrho > 0$ and any $x \in \mathbb{R}^{d}$,
		\begin{equation*}
			\liminf_{\epsilon \to 0} \min \left\{ \mathscr{F}_{\epsilon,\delta(\epsilon)}(u; Q_{\varrho}(x)) \, \mid \, u - q(\epsilon^{-1}(\cdot - x) \cdot e_{1}) \in H^{1}_{0}(Q_{\varrho}(x)) \right\} \geq \bar\sigma(e_1)
		\end{equation*}
	To be precise, we prove this holds in probability.
	
	Given $ x \in \mathbb{R}^{d}$, $S > 0$, and a functional $\mathscr{F}$, define $ m( \mathscr{F}, Q_S(x) , q) $ by 
		\begin{align} \label{E: definition of little m}
			 m( \F , Q_S(x) , q) &= \min \left\{ \F (u; Q_{S}(x)) \, \mid \, u - q((\cdot - x) \cdot e_{1}) \in H^{1}_{0}(Q_{S}(x)) \right\}.
		\end{align}
	Changing variables by setting $ R = \epsilon^{-1}$ and $\gamma(R) = \delta(\epsilon) \epsilon^{-1} $ (cf. \eqref{eqn:length-scales-trafo}), what we seek to prove becomes
		\begin{equation*}
			\liminf_{R \to \infty} R^{1 - d} \, m(\F_{ \gamma(R) }, Q_{ \varrho R }(x) , q)  \geq \bar \sigma(e_1) \quad \text{in probability.}
		\end{equation*}
	Note that, by stationarity (since we are only interested in convergence in probability), it suffices to assume $x = 0$.
	
	In order to appeal to Theorem \ref{T: elliptic term cube decomposition}, assume that $R \geq r_{c} \varrho^{-1}$ and define $K(R) \in \mathbb{N}$ in such a way that
		\begin{equation} \label{E: choice of K}
			(2(K(R) - 1) + 1) r_{c} \leq \varrho R \leq (2 K(R) + 1) r_{c}.
		\end{equation}
	At this stage, it will be convenient to define radii $S_{+}(R)$ and $S_{-}(R)$ by $S_{-}(R) = (2 (K(R) - 1) + 1)r_{c}$ and $S_{+}(R) = (2 K(R) + 1) r_{c}$.  By the definition of $K(R)$, we have $Q_{S_{-}(R)} \subseteq Q_{\varrho R} \subseteq Q_{S_{+}(R)}$. 
	
	By a monotonicity argument similar to \cite[Proposition 5]{morfe}, for any $ \F \in \{ \F_{\gamma(R)}, \overline{\F} \}$, we have that
		\begin{align} \label{E: basic monotonicity}
			\begin{aligned}
			m(\F, Q_{ S_{+}(R)} , q) - E(R) &\leq m(\F, Q_{ \varrho R } , q) \\ &\leq m(\F, Q_{ S_{-}(R) } , q) + E(R),
			\end{aligned}
		\end{align}
	where $E : (0,\infty) \to (0,\infty)$ is a deterministic function such that $R^{1-d} E(R) \to 0$ as $R \to \infty$. 
	
	Next, for any $K \in \mathbb{N}$, let $u_{K} : Q_{(2K + 1) r_{c}} \to [-1,1]$ be such that
		\begin{equation*}
			m(\F_{\gamma(R)} , Q_{ (2K+1) r_{c} } , q)  = \mathscr{F}_{\gamma(R)}(u_{K}; Q_{(2K+1) r_{c}} ),
		\end{equation*}
		and $ u_{K}(y) = q(y \cdot e_{1}) $ for $ y \in \partial Q_{(2K+1) r_{c}} $. Since $u_{K}$ is a minimizer, observe that if $v$ is the function $v(y) = q(y \cdot e_{1})$, then
		\begin{align*}
			\mathscr{F}_{\gamma(R)}(u_{K}; Q_{(2K+1) r_{c}}) &\leq \int_{Q_{(2K+1) r_{c}}} \left( \frac{1}{2} a(\gamma^{-1} y) Dv(y) \cdot Dv(y) + \theta(\gamma^{-1}y) W(v(y)) \right) \, dy
		\end{align*}
	and, thus, by slicing and \eqref{E: finite width}, $\mathscr{F}_{\gamma(R)}(u_{K}; Q_{(2K+1) r_{c}}) \leq C ((2K+1) r_{c})^{d-1}$, where the constant $C$ depends on $q$, $\Lambda$, $\theta^{*}$, and $W$, but not on $K$. 
	
	In what follows, let $\omega : [0,\infty) \to [0,\infty)$ be the modulus of continuity from Theorem \ref{T: elliptic term cube decomposition} and let $\{X_{\gamma}(z)\}_{z \in \mathbb{Z}^{d}}$ be the random variables 
		\begin{align*}
			X_{ \gamma }( z ) =  \omega \left( { \rm Sub }_{ r_c z } \left( \frac{ r_c }{ \gamma } \right) + { \rm Osc }_{ r_c z } \left( \frac{ r_c }{\gamma } \right) \right).
		\end{align*}
	Fix $\nu > 0$.  Using $u_{K}$ as the minimizer in Theorem \ref{T: elliptic term cube decomposition}, for any $K \in \mathbb{N}$ and $\gamma > 0$, we obtain that
		\begin{align*}
			& \mathbb{P} \left\{ m( \F_{\gamma}, Q_{ (2K+1) r_{c} }, q ) < m( \overline\F , Q_{ (2K+1) r_{c} }, q ) - \nu ((2K + 1) r_{c})^{d-1} \right\} \\
				&\qquad \qquad \leq \mathbb{P} \left\{ C\max_{ z \in \mathbb{Z}^{d} \cap [-K,K]^d } X_{\gamma}(z) > \nu \right\} \\
				&\qquad \qquad \leq \sum_{ z \in \mathbb{Z}^{d} \cap [-K,K]^d } \mathbb{P}\{X_{\gamma}(z) > C^{-1} \nu\} 
				= (2K + 1)^{d} \, \mathbb{P}\{X_{\gamma}(0) > C^{-1} \nu\}.
		\end{align*}
	Setting $\gamma = \gamma(R)$ and $K = K(R)$ with $K(R)$ determined by \eqref{E: choice of K}, we invoke the definitions of $X_{\gamma}$, $ S_+ (R) $, $R$, and $\gamma(R)$ to obtain, for $R \geq \varrho^{-1} r_{c}$, the error estimate
		\begin{align*}
			&\mathbb{P} \left\{ m( \F_{\gamma(R)}, Q_{ S_{+}(R)}, q ) < m( \overline\F , Q_{ S_{+}(R) }, q ) - \nu S_{+}(R)^{d-1} \right\} \\
			&\quad \leq (r_{c}^{-1} S_{+}(R))^{d} \, \mathbb{P}\{X_{\gamma(R)}(0) > C^{-1} \nu\} \\
				&\quad \leq 3^{d} r_{c}^{-d} \varrho^{d} \cdot R^{d} \, \mathbb{P}\{X_{\gamma(R)}(0) > C^{-1} \nu\} \\
				&\quad = 3^{d} r_{c}^{-d} \varrho^{d} \cdot \epsilon^{-d} \, \mathbb{P} \left\{ { \rm Sub }_{ 0 } \left( \frac{ \epsilon }{ \delta(\epsilon) } r_c  \right) + { \rm Osc }_{ 0 } \left( \frac{ \epsilon }{ \delta(\epsilon) } r_c \right) > \omega^{-1} \left( C^{-1} \nu \right) \right\},
		\end{align*}
	where $\omega^{-1}$ is any fixed left-inverse of $\omega$.
	Invoking our assumptions \eqref{E: sublinear part} and \eqref{E: averaging part} on $\delta$ in the form of \eqref{E: reformulation of main homogenization assumptions}, we combine the previous string of inequalities with \eqref{E: basic monotonicity} to find
		\begin{align*}
			\lim_{R \to \infty} \mathbb{P}\{ R^{1-d} m( \F_{\gamma(R)}, Q_{ \varrho R }, q ) < R^{1-d} m( \overline\F, Q_{ \varrho R }, q ) - \nu \} = 0.
		\end{align*}
	At the same time, since $\overline{\mathscr{F}}$ has constant coefficients, it is well-known (see \cite[Theorem 3.7]{ansini_braides_chiado-piat}) that 
		\begin{equation*}
			\lim_{R \to \infty} R^{1-d} m( \overline\F, Q_{ \varrho R }, q ) = \bar{\sigma}( e_{1} ).
		\end{equation*}
	Putting it all together, we conclude that, for any $\nu > 0$,
		\begin{align*}
			\lim_{R \to \infty} \mathbb{P} \left\{ R^{1-d} m( \F_{\gamma(R)}, Q_{ \varrho R }, q ) < \bar{\sigma} (e_{1} ) - \nu \right\} = 0.
		\end{align*}
	\end{proof}
	
\subsection{Proof of Corollary \ref{C: planar homogenization corollary}} \label{S: almost sure version} As advertised in the introduction, if we work with sequences $(\epsilon_{j})_{j \in \mathbb{N}}$ and $(\delta_{j})_{j \in \mathbb{N}}$ and ask for slightly better decay of $\delta$ relative to $\epsilon$, we can upgrade from convergence in probability to almost-sure convergence. The proof is a more-or-less standard modification of the proof of Theorem \ref{T: planar homogenization theorem lower bound} involving the Borel-Cantelli Lemma, hence, in the proof that follows, only the relevant changes are described.

	\begin{proof}[Proof of Corollary \ref{C: planar homogenization corollary}] The argument proceeds as in the proof of Theorem \ref{T: planar homogenization theorem lower bound}. Let $R_{j} = \epsilon_{j}^{-1}$ and $\gamma_{j} = \delta_{j} \epsilon_{j}^{-1}$. As before, it is convenient to define $K_{j} \in \mathbb{N}$ such that 
		\begin{equation*}
			(2(K_{j}-1) + 1) r_{c} \leq \varrho R_{j} \leq (2K_{j} + 1) r_{c}.
		\end{equation*}
		
	Let $\nu > 0$. The key observation is that, as in the proof of Theorem \ref{T: planar homogenization theorem lower bound}, we can write
		\begin{align*}
			& \mathbb{P} \left\{ m( \F_{ \gamma_j }, Q_{ (2K_{j} + 1)r_{c} }( R_{j}x ), q ) < m( \overline{ \F }, Q_{ (2K_{j} + 1) r_{c} }( R_{j}x ), q )  - \nu ((2K_{j} + 1)r_{c})^{d-1} \right\} \\
			&\qquad \leq (2K_{j} + 1)^{d} \mathbb{P}\{X_{\gamma_{j}}(0) > C^{-1} \nu\} \\
			&\qquad \leq 3^{d} r_{c}^{-d} \varrho^{d} \cdot \epsilon_{j}^{-d}  \mathbb{P} \left\{ { \rm Sub }_{ 0 } \left( \frac{ \epsilon_j }{ \delta_j  }  r_c \right) + { \rm Osc }_{ 0 } \left( \frac{ \epsilon_j }{ \delta_j} r_c  \right) > \omega^{-1} \left( \frac{\nu}{C} \right) \right\}
		\end{align*}
	and, thus, by \eqref{E: sublinear summable},
		\begin{equation*}
			\sum_{j = 1}^{\infty} \mathbb{P} \left\{ \begin{aligned} & m( \F_{ \gamma_j }, Q_{ (2K_{j} + 1)r_{c} }( R_{j}x ), q ) \\ & \qquad < m( \overline{ \F }, Q_{ (2K_{j} + 1) r_{c} }( R_{j}x ), q ) - \nu ((2K_{j} + 1)r_{c})^{d-1} \end{aligned} \right\} < \infty.
		\end{equation*}
	Therefore, by the Borel-Cantelli Lemma, with probability one,
		\begin{equation*}
			\liminf_{j \to \infty} ((2K_{j} + 1) r_{c})^{1-d} m( \F_{ \gamma_j }, Q_{ K_{j} r_{c} }( R_{j}x ), q ) \geq \bar{\sigma}(e_{1}) - \nu.
		\end{equation*}
	Since $\nu > 0$ was arbitrary, this concludes the proof.
\end{proof}

%
	
\subsection{Proof of Proposition \ref{prop:existence-of-scales}} \label{S: existence of scales}The previous results show that homogenization holds conditional on a decay assumption on the microscale $\delta$. We now prove that, as long as the medium $(a,\theta)$ is stationary and ergodic, there is always a choice of scales satisfying this assumption.

	\begin{proof}[Proof of Proposition \ref{prop:existence-of-scales}] Since \eqref{E: sublinearity averaging section} and \eqref{E: averaging theta averaging section} both hold, for any $j \in \mathbb{N}$, there is a scale $R_{j} \geq j$ such that
		\begin{gather*}
			\mathbb{P}\left\{ {\rm Sub}_{0}(R_{j}) > 2^{-j} \right\} \leq 2^{-j}, \\
			\mathbb{P} \left\{ {\rm Osc}_{0}(R_{j}) > 2^{-j} \right\} \leq 2^{-j}.
		\end{gather*}
	Define $\delta_{*,{\rm Sub}}$ via the following rule:
		\begin{equation*}
			\delta_{*,{\rm Sub}}(\epsilon) = \epsilon R_{j}^{-1} \quad \text{for each} \quad 2^{-(j + 1)/2d} \leq \epsilon < 2^{-j/2d}.
		\end{equation*}
	
	Suppose that $\epsilon \mapsto \delta(\epsilon)$ is any choice of scale such that $\delta(\epsilon) \leq \delta_{*,{\rm Sub}}(\epsilon)$ for all $\epsilon$ close enough to zero. Given any $\nu > 0$, if $2^{-(j+1)/2d} \leq \epsilon < 2^{-j/2d} < \nu^{\frac{1}{2d}}$, then $\epsilon \delta(\epsilon)^{-1} \geq R_{j}$ and, therefore,
		\begin{align*}
			\epsilon^{-d} &\mathbb{P} \left\{ {\rm Sub}_{0} \left( \frac{\epsilon}{\delta(\epsilon)} \right) > \nu \right\} \leq 2^{(j + 1)/2} \mathbb{P} \left\{ {\rm Sub}_{0} \left( \frac{\epsilon}{\delta(\epsilon)} \right) > 2^{-j} \right\}  \leq \sqrt{2} \cdot 2^{-j/2}.
		\end{align*}
	This proves that 
		\begin{equation*}
			\lim_{\epsilon \to 0} \epsilon^{-d} \mathbb{P} \left\{ {\rm Sub}_{0} \left( \frac{\epsilon}{\delta(\epsilon)} \right) > \nu \right\} = 0.
		\end{equation*}
	
	Via analogous arguments, we define a scale $\delta_{*,{\rm Osc}}$ in such a way that if $\delta(\epsilon) \leq \delta_{*,{\rm Osc}}(\epsilon)$ for all $\epsilon$ in a neighborhood of zero, then
		\begin{align*}
			\lim_{\epsilon \to 0} \epsilon^{-d} \mathbb{P} \left\{ {\rm Osc}_{0} \left( \frac{\epsilon}{\delta(\epsilon)} \right) > \nu \right\} = 0.
		\end{align*}
	We conclude by setting $\delta_{*}(\epsilon) = \min\{\delta_{*,{\rm Sub}}(\epsilon), \delta_{*,{\rm Osc}}(\epsilon)\}$. \end{proof}

\subsection{Quantitative Sublinearity of Correctors in the Random Checkerboard}\label{S: quantitative checkerboard -1} In this section, we recall the optimal decay rate of the limit \eqref{E: sublinearity averaging section} in the case when $a$ is a random checkerboard as in the discussion of Section \ref{S: intro rare events}.

Specifically, it follows from the main results of \cite{gloria_neukamm_otto} that, for any $\nu > 0$, there is a constant $c(\nu) > 0$ such that 
	\begin{align} \label{E: large deviations lower bound sublinearity}
		-  r^{-d}  \log\mathbb{P} \{ { \rm Sub }_{0}(r) > \nu \} \geq c(\nu) \quad \text{for any} \, \, r \geq 1 . 
	\end{align}
As sketched in \cite[Remark 5]{gloria_neukamm_otto}, provided $a$ is in fact nonconstant, it is possible to prove that the matching upper bound also holds, at least for small enough $\nu$.  That is, for any $\nu > 0$ small enough,
	\begin{align} \label{E: large deviations upper bound sublinearity}
		- r^{-d}  \log \mathbb{P} \{ { \rm Sub }_{0}(r) > \nu \} \leq c(\nu)^{-1} \quad \text{for any} \, \, r \geq 1 .
	\end{align}
A proof of this fact for a specific choice of the variables $\{A_{z}\}_{z \in \mathbb{Z}^{d}}$ can be found in the monograph \cite[Section 3.6]{armstrong_kuusi_mourrat_book}.  Below, for the reader's convenience, we give a short proof that it holds for \emph{any} (nonconstant) uniformly elliptic random checkerboard.

\begin{remark} The work \cite{gloria_neukamm_otto} considers stationary uniformly elliptic matrix fields, obtaining quantitative results under the assumption that certain functional inequalities hold.  This applies, in particular, to the random checkerboard.  To be precise, if $a$ is the random checkerboard considered here, then, by \cite[Proposition 2.3]{duerinckx_gloria_constructive} and its proof, there is a $\gamma_{*} = \gamma_{*}(d) > 0$ such that the rescaled field $a^{\gamma_{*}} = a(\gamma_{*}^{-1}\cdot)$ satisfies the standard logarithmic Sobolev inequality with respect to the essential oscillation $\partial^{\text{osc}}$, as defined in \cite[Definition 1]{gloria_neukamm_otto}.  Therefore, \cite[Theorem 2]{gloria_neukamm_otto} and its proof apply to the rescaled field.  Since rescaling the field has the effect of replacing $R \mapsto { \rm Sub}_{0}(R)$ by $R \mapsto { \rm Sub }_{0}(\gamma^{-1}_{*} R)$, it does not effect the qualitative statement \eqref{E: large deviations lower bound sublinearity}. \end{remark}

Concerning the large deviations lower bound \eqref{E: large deviations lower bound sublinearity}, this follows from the proof of \cite[Theorem 2]{gloria_neukamm_otto}.  In particular, the proof in \cite{gloria_neukamm_otto} begins by establishing that there is a deterministic exponent $\mu > 0$ and a random variable $r_{**} > 0$ such that 
	\begin{gather*}
		\mathbb{E} [ \exp ( C^{-1} r_{**}^{d} ) ] < \infty \quad \text{for some} \, \, C = C(d,\lambda,\Lambda) > 0  \\
		\text{and} \quad \frac{ 1 }{ R } \left( \fint_{ B_{R} } \left| (\phi,\sigma) - \fint_{ B_{R} } (\phi,\sigma) \right|^{2} \right)^{\frac{1}{2}} \leq C \left( \frac{r_{**}}{R} \right)^{\mu} \quad \text{for each} \, \, R \geq r_{**} ,
	\end{gather*}
see \cite[Proof of Theorem 2 \& Proposition 2]{gloria_neukamm_otto}, which yields a constant of the form $c(\nu) \sim \min\{ 1, \nu^{ \frac{ d }{ \mu } } \}$ in \eqref{E: large deviations lower bound sublinearity} after applying Markov's inequality. 


\begin{remark} The bound \eqref{E: large deviations lower bound sublinearity} can also be obtained using the approach exposed in the recent monograph \cite[Chapters 4 \& 6]{armstrong_kuusi_book}.  Like \cite{gloria_neukamm_otto}, the argument uses the notion of a random scale analogous to the minimal radius $r_{*}(\nu)$ (see \eqref{E: minimal radius} above), above which homogenization effects become apparent.  (To the best of our knowledge, this idea first appears in \cite{armstrong_smart}, which introduced some of the techniques used in \cite{armstrong_kuusi_book}.)   \end{remark}

We next give a short proof of the large deviations upper bound \eqref{E: large deviations upper bound sublinearity}.

	\begin{proof}[Proof of the Upper Bound \eqref{E: large deviations upper bound sublinearity}]  In the proof, it is convenient to take $a$ of the form $a(x) = \sum_{z \in \mathbb{Z}^{d}} A_{z} \mathbf{1}_{\tilde{Q}(z)}(x)$, where $\tilde{Q}(z) = \prod_{j = 1}^{d} [z_{j},z_{j}+1)$ for $z = (z_{1},\dots,z_{d})$, which is no loss of generality since it is a translation of the field given in \eqref{E: one d random checkerboard}.  Since $a$ is nonconstant, we can choose two distinct points $A_{\star} \neq A_{\bullet}$ in the support of the law of $a(0)$.  Let $\|\cdot\|$ be some norm on the space of $d \times d$ matrices.  For a given $\alpha : \{1,\dots,d\} \to \{-1,1\}$, let $H_{\alpha} \subseteq \mathbb{R}^{d}$ be the orthant consisting of points $x$ with $x_{j} = |x_{j}| \alpha_{j}$ for each $j$.  Define a deterministic matrix field $\underline{a} : \mathbb{R}^{d} \to \text{Sym}(d)$ by
		\begin{align*}
			\underline{a}(x) = \left\{ \begin{array}{r l}
								A_{\star}, & \text{if} \, \, x \in H_{\alpha} \, \, \text{for some} \, \, \alpha \, \, \text{with} \, \, |\{j \, \mid \, \alpha_{j} = 1\}| \, \, \text{even}, \\
								A_{\bullet}, & \text{if} \, \, x \in H_{\alpha} \, \, \text{for some} \, \, \alpha \, \, \text{with} \, \, |\{j \, \mid \, \alpha_{j} = 1\}| \, \, \text{odd}.	
							\end{array} \right.
		\end{align*}
	Since for any $\epsilon > 0$, the probabilities $\mathbb{P} \{ \| A_{z} - A_{\star} \| < \epsilon \}$ and $\mathbb{P} \{ \| A_{z} - A_{\bullet}\| < \epsilon\}$ are positive and independent of $z \in \mathbb{Z}^{d}$, and for any $z \in \mathbb{Z}^{d}$ we have $\tilde{Q}(z) \subseteq H_{\alpha}$ for some $\alpha$, there is a constant $c(\epsilon) > 0$ such that if $E_{\epsilon}(\gamma)$ is the event 
		\begin{align*}
			E_{\epsilon}(\gamma) =  \{ \| a^{\gamma}(x) - \underline{a}(x) \| < \epsilon \, \, \text{for a.e.} \, \, x \in Q_{1} \} ,
		\end{align*}
	then 
		\begin{align} \label{E: large deviations upper bound example}
			-\lim_{ \gamma \to 0 } \gamma^{d} \log \mathbb{P}(E_{\epsilon}(\gamma)) = c(\epsilon).
		\end{align}
	At the same time, we know that, for any $\xi \in \mathbb{R}^{d}$, with probability one,
		\begin{align*}
			- \nabla \cdot a^{\gamma}(x) ( \xi + \nabla \phi_{\xi}^{\gamma} ) = 0 \quad \text{in} \, \, Q_{1} .
		\end{align*}
	Thus, if we define the space of matrix fields $\mathcal{A}(\epsilon)$ and the constant $M(\epsilon,\xi)$ by 
		\begin{gather*}
			\mathcal{A}(\epsilon) = \{ b : Q_{1} \to \text{Sym}(d) \, \mid \, \| b(x) - \underline{a}(x) \| < \epsilon \, \, \text{for a.e.} \, \, x \in Q_{1} \}, \\
			\nu(\epsilon,\xi) = \inf \left\{ \big( \int_{Q_{1}} | v |^{2} \, dx \big)^{\frac{1}{2}} \, \mid \, (v,b) \in H^{1}(Q_{1}) \times \mathcal{A}(\epsilon), \, \, - \nabla \cdot b ( \xi + \nabla v ) = 0 \, \, \text{in} \, \, Q_{1} \right\},
		\end{gather*}
	then, by scaling,
		\begin{align*}
			\mathbb{P} \{ { \rm Sub }_{0}(\gamma^{-1}) \geq \nu(\epsilon,\xi) \} \geq \mathbb{P} \left\{ \left( \int_{Q_{1}} | \phi_{\xi}^{\gamma} - \fint_{Q_{1}} \phi_{\xi}^{\gamma} |^{2} \right)^{\frac{1}{2}} \geq \nu(\epsilon,\xi) \right\} \geq \mathbb{P} ( E_{\epsilon}(\gamma) ) . 
		\end{align*}
	In view of \eqref{E: large deviations upper bound example}, it only remains to observe that there is a $\xi_{*} \in \mathbb{R}^{d} \setminus \{0\}$ such that 
		\begin{align} \label{E: liminf deviations argument}
			\liminf_{ \epsilon \to 0 } \nu(\epsilon,\xi_{*}) > 0,
		\end{align}
	so that the set $\{ \epsilon > 0 \, \mid \, \nu(\epsilon,\xi_{*}) > 0\}$ is nonempty and the bound obtained above is thus nontrivial.  This last claim we prove by contrapositive.
	
	Indeed, if for some $\xi \in \mathbb{R}^{d}$ we have that  $\nu(\epsilon,\xi) \to 0$ as $\epsilon \to 0$, then this implies that the equation $- \nabla \cdot \underline{a}(x) \xi = 0$ holds in the distributional sense in $Q_{1}$.  By definition of $\underline{a}$, this implies that the measure
		\begin{align*}
			n_{M} \cdot ( A_{\star} - A_{\bullet} ) \xi \, \mathcal{H}^{d-1} \lfloor_{M} = 0 \quad \text{in} \, \, Q_{1},
		\end{align*}
	where $M = \{ (x_{1},\dots,x_{d}) \in \mathbb{R}^{d} \, \mid \, x_{j} = 0 \, \, \text{for some} \, \, j \}$, $n_{M}$ is the normal vector to $M$, and $\mathcal{H}^{d-1} \lfloor_{M}$ is the $(d - 1)$-dimensional Hausdorff measure on $M$.  Since each of the standard coordinate vectors is attained by $n_{M}$ in a set of positive $\mathcal{H}^{d-1}$ measure, this last identity implies that
		\begin{align*}
			( A_{\star} - A_{\bullet} ) \xi = 0.
		\end{align*}
	Thus, $\nu(\epsilon,\xi)$ can only vanish in the limit $\epsilon \to 0$ if $A_{\star} \xi = A_{\bullet} \xi$.  Since $A_{\star} \neq A_{\bullet}$, it is possible to find a $\xi_{*} \neq 0$ such that $A_{\star} \xi_{*} \neq A_{\bullet} \xi_{*}$, and \eqref{E: liminf deviations argument} follows.
	  \end{proof}

%
	
In view of \eqref{E: large deviations lower bound sublinearity} and \eqref{E: large deviations upper bound sublinearity}, for the random checkerboard, the assumption \eqref{E: sublinear part} in our main theorem amounts to requiring that $\epsilon^{-d} \exp ( - C (\frac{\epsilon}{\delta(\epsilon)} )^{d} )$ vanishes as $\epsilon \to 0$ for any $C> 0$, or, equivalently,
	\begin{equation*}
		\lim_{ \epsilon \to 0 } \frac{ \delta ( \epsilon ) | \log ( \epsilon ) |^{ 1/d } }{ \epsilon } = 0.
	\end{equation*}
In the next subsection, we will see that this same sufficient condition applies to our other assumption \eqref{E: averaging part} when the field $\theta$ is a random checkerboard.  The main result of the companion paper \cite{part2} shows this is also a necessary condition when $d = 1$. We expect that the higher dimensional random checkerboard is fundamentally different; this will be treated in future work.

\begin{remark} More generally, the results of \cite{gloria_neukamm_otto} and \cite{armstrong_kuusi_book} provide estimates on the decay of $\mathbb{P} \{ {\rm Sub}_{0}(R) > \nu \}$ provided the law of $a$ satisfies certain mixing conditions, which are specified in both cases using functional derivatives analogous to the one from Malliavin calculus.  In that context, the estimates from \cite{gloria_neukamm_otto,armstrong_kuusi_book} can be used to specify sufficient conditions for the assumption \eqref{E: sublinear part} to hold. \end{remark}

\subsection{Quantitative Averaging of the Random Checkerboard} \label{S: quantitative checkerboard} Let us next consider the case when $\theta$ is the random checkerboard as in \eqref{E: one d random checkerboard}.  We claim that, analogous to the previous discussion, for any $\nu > 0$ small enough, there is a constant $C_{\nu} \geq 1$ such that
	\begin{equation} \label{E: exponential asymptotic of osc}
		C_{\nu}^{-1} \leq - r^{-d} \log \mathbb{P}\{{ \rm Osc }_{0}(r) > \nu\} \leq C_{\nu} \quad \text{for any} \quad r \geq 1.
	\end{equation}
Since we are not aware of a reference for this relatively elementary large deviations estimate, we provide a proof below.  

In fact, we expect that the i.i.d.\ assumption used here is stronger than what is needed to obtain the lower bound in \eqref{E: exponential asymptotic of osc}. We comment on this in the proof.

In the proof below, we assume for notational convenience that $\theta_{*}$ is the essential infimum of $\theta(x)$ for any $x \in \mathbb{R}^{d}$.

\begin{proof}[Proof of \eqref{E: exponential asymptotic of osc}]  By the definition of ${ \rm Osc}_{0} (\cdot)$ and a union bound, it suffices to establish that 
	\begin{align*}
		0 <  &-\limsup_{r \to \infty} r^{-d} \log \left( \mathbb{P}\{ r^{-1} \|\theta - \bar{\theta}\|_{H^{-1}(Q_{r})} > \nu\} \right), \\
		&-\liminf_{r \to \infty} r^{-d} \log \left( \mathbb{P}\{ r^{-1} \|\theta - \bar{\theta}\|_{H^{-1}(Q_{r})} > \nu\} \right) < \infty.
	\end{align*}
Toward this end, it is somewhat convenient to rescale, writing $\theta^{\gamma}(x) = \theta(\gamma^{-1}x)$, in which case the statement becomes
	\begin{align} 
		0 < &-\limsup_{\gamma \to 0} \gamma^{d} \log \left( \mathbb{P}\{ \|\theta^{\gamma} - \bar{\theta}\|_{H^{-1}(Q_{1})} > \nu \} \right), \label{E: basic reformulation lower bound} \\
		&-\liminf_{\gamma \to 0} \gamma^{d} \log \left( \mathbb{P}\{ \|\theta^{\gamma} - \bar{\theta}\|_{H^{-1}(Q_{1})} > \nu \} \right) < \infty. \label{E: basic reformulation upper bound}
	\end{align}

To see that \eqref{E: basic reformulation lower bound} holds, let $(e_{n})_{n \in \mathbb{N}^{d}}$ be the trigonometric orthonormal system of eigenfunctions of the Laplacian with Dirichlet boundary conditions, so that $-\Delta e_{n} = \pi^{2} | n |^{2} e_{n}$ and $\int_{Q_{1}} e_{n}(x) e_{m}(x) \, dx = \delta_{nm}$, and recall that
	\begin{align} \label{E: H minus one formula}
		\|\theta^{\gamma} - \bar{\theta}\|_{H^{-1}(Q_{1})}^{2} = \sum_{n \in \mathbb{N}^{d}} \frac{1}{\pi^{2} |n|^{2}} \left| \int_{Q_{1}} (\theta^{\gamma}( x) - \bar{\theta}) e_{n}(x) \, dx \right|^{2}.
	\end{align}
	
On the one hand, for any $N \in \mathbb{N}$,
	\begin{align} \label{E: very easy bound}
		\sum_{n \in \mathbb{N}^{d} \, : \, |n| \geq N} \frac{1}{ 4 \pi^{2} |n|^{2}} \left| \int_{Q_{1}} (\theta^{\gamma}(x) - \bar{\theta}) e_{n}(x) \, dx \right|^{2} &\lesssim N^{-2} \int_{ Q_{1} } | \theta^{\gamma}( x ) - \bar{\theta} |^{2} \, dx \\
        &\lesssim N^{-2} ( \theta^{*} - \theta_{*} )^{2} . \nonumber
	\end{align}
On the other hand, since $\theta$ is the random checkerboard, the standard large deviations estimate for sums of bounded, independent random variables implies that, for any $n \in \mathbb{N}$ and any $\nu > 0$, there is a constant $C_{\nu,n} > 0$ such that 
	\begin{equation*}
		\mathbb{P}\left\{ \left| \int_{Q_{1}} (\theta^{\gamma}(x) - \bar{\theta}) e_{n}(x) \, dx \right| > \nu \right\} \leq \exp \left( - C_{\nu,n} \gamma^{-d} \right).
	\end{equation*}
see, for instance, \cite[Chapter 2]{dembo_zeitouni} or \cite[Section 2.4, Example 4]{varadhan}.
This is the only place in the proof of \eqref{E: basic reformulation lower bound} where we use the fact that $\theta$ is a random checkerboard.  We expect this would still hold provided that the law of $ \theta $ satisfies a logarithmic Sobolev inequality in a suitable sense; see \cite[Proposition 1.7]{duerinckx_gloria_concentration} for related results in this direction.

Taking $N \gtrsim (\theta^{*} - \theta_{*}) \nu^{-\frac{1}{2}}$ and combining the last estimate with \eqref{E: very easy bound}, we conclude
	\begin{equation*}
		-\limsup_{\gamma \to 0} \gamma^{d} \log \left( \mathbb{P} \left\{ \|\theta(\gamma^{-1} \cdot) - \bar{\theta}\|_{H^{-1}(Q_{1})} > \nu \right\} \right) \geq \sum_{ n \in \mathbb{N}^{d} \, : \, | n | < N } C_{\nu/2,n} > 0 ,
	\end{equation*}
	which proves \eqref{E: basic reformulation lower bound}.

Finally, regarding \eqref{E: basic reformulation upper bound}, it suffices to notice that since $\theta_{*}$ is the essential infimum of $\theta(x)$ for any $x \in \mathbb{R}^{d}$, the definition of $\theta$ implies that, for any $\epsilon \in (0, \bar{\theta} - \theta_{*})$,
	\begin{align*}
		\lim_{ \gamma \to 0 } \gamma^{d} \log \mathbb{P} \{ \theta^{\gamma}(x) \leq \theta_{*} + \epsilon \, \, \text{for each} \, \, x \in Q_{1} \} = \log \mathbb{P} \{ \theta(0) \leq \theta_{*} + \epsilon \} > -\infty.
	\end{align*}
Let $\mu = \int_{Q_{1}} e_{(1,1,\dots,1)}(x) \, dx$ be the total mass of the principal eigenfunction $e_{(1,1,\dots,1)}$.  Since the principle eigenfunction is positive, for any $\epsilon \in (0,\bar{\theta} - \theta_{*})$,
	\begin{align*}
		\mathbb{P} \{ \theta^{\gamma} \leq \theta_{*} + \epsilon \, \, \text{ in} \, \, Q_{1} \} \leq \mathbb{P} \left\{ \left| \int_{Q_{1}} ( \theta^{\gamma}( x) - \bar{\theta} ) e_{(1,1,\dots,1)}(x) \, dx \right| \geq \mu ( \bar{\theta} - \theta_{*} - \epsilon )  \right\} .
	\end{align*}
In view of \eqref{E: H minus one formula}, this implies \eqref{E: basic reformulation upper bound} for any $\nu > 0$ small enough.   \end{proof}

\section{Periodic and (Uniformly) Almost Periodic Media} \label{S: periodic almost periodic section}

This section discusses the relevant adaptations of the arguments of Section \ref{S: homogenization 2} to the case when the coefficients $a$ and $\theta$ are both either periodic or in the class $B^{\infty}(\mathbb{R}^{d})$.  Recall that functions in $B^{\infty}(\mathbb{R}^{d})$ are sometimes called \emph{uniformly almost periodic.}  

The main aim of this section is to prove Corollary \ref{cor:periodic-homogenization}, which shows that if $a$ and $\theta$ are periodic or uniformly almost periodic, then homogenization determines the limiting behavior of the energy $\mathscr{F}_{\epsilon,\delta}$, no matter the choice of the scale $\epsilon \mapsto \delta(\epsilon)$ (as long as $ \epsilon^{-1} \delta(\epsilon) \rightarrow 0 $). In order to do this, we need to show that the quantities ${\rm Osc}(Q_{R}(x))$, ${\rm Sub}(Q_{R}(x))$, and $| \bar{a}(Q_{R}(x)) - \bar{ a} |$  appearing in Theorem \ref{T: elliptic term cube decomposition} converge uniformly to zero as $R \to \infty$ in this setting, which is the main focus of Sections \ref{S: periodic} and \ref{S: uniformly almost periodic}. 

It should be emphasized that there are nonuniformly almost periodic media for which the homogenized functional $\bar{\mathscr{E}}$ is not necessarily the $\Gamma$-limit of the energy $\mathscr{F}_{\epsilon,\delta}$ for certain choices of the scale $\epsilon \mapsto \delta(\epsilon)$. This is proved in the companion paper \cite{part2}.

\subsection{Periodic Media} \label{S: periodic} Let us recall the relevant facts from periodic homogenization, which imply, in particular, that the assumptions of Section \ref{S: deterministic assumptions} all hold.  As in the random case, we work with correctors $(\phi_{Q},\sigma_{Q}) = (\phi,\sigma)$ and matrices $\bar{a}(Q) = \bar{a}$ that do not depend on the cube $Q$.

First, if $a$ is $\mathbb{Z}^{d}$-periodic, then, for any $\xi \in S^{d-1}$, recall that it is possible to fix a $\phi_{\xi} \in H^{1}(\mathbb{T}^{d})$ such that 
	\begin{equation*}
		- \nabla \cdot (a(x) (\xi + \nabla \phi_{\xi})) = 0 \quad \text{in} \, \, \mathbb{T}^{d}, \quad \fint_{\mathbb{T}^{d}} \phi_{\xi} = 0.
	\end{equation*}
Indeed, the existence and uniqueness of such $\phi_{\xi}$ is discussed in detail in the classic books \cite{bensoussan_lions_papanicolaou} and \cite{jikov_kozlov_oleinik}, and $\xi \mapsto \phi_{\xi}$ is linear by linearity of the equation.  Observe that, for any $ x \in \mathbb{Z}^d $, in the limit as $ R \rightarrow \infty $, by the Poincaré inequality,
	\begin{equation*}
		\begin{aligned}
		\frac{1}{R^{2}} \fint_{Q_{R}(x) } \left|\phi_{\xi} - \fint_{Q_R(x)} \phi_{\xi} \right|^{2}
		& \leq  \frac{1}{R^{2}} \fint_{ Q_{R}(x) } \left| \phi_{\xi} \right|^{2} \\
		& \sim \frac{1}{R^2} \fint_{ \mathbb{T}^d } | \phi_{\xi} |^{2} \lesssim \frac{1}{R^2} \fint_{\mathbb{T}^{d}} | \nabla \phi_{\xi} |^{2},
		\end{aligned}
	\end{equation*}
where the implicit constants above do not depend on $x$.
Writing $\phi = (\phi_{e_{1}},\dots,\phi_{e_{d}})$ for the vector-valued corrector, this proves
	\begin{align*}
		\lim_{R \to \infty} \sup_{x \in \mathbb{R}^{d}} \frac{1}{R} \left( \fint_{Q_R(x)} \left| \phi - \fint_{Q_R(x)} \phi \right|^{2} \right)^{\frac{1}{2}} = 0.
	\end{align*}

The flux corrector $\sigma_{\xi}$ can be defined as follows.  Since the flux $q_{\xi} = a (\xi + \nabla \phi_{\xi})$ is an element of $L^{2}(\mathbb{T}^{d})$, for any $i,j \in \{1,2,\dots,d\}$, there is a unique function $\sigma_{\xi i j} \in H^{1}(\mathbb{T}^{d})$ satisfying the following two equations:
	\begin{align*}
		- \Delta \sigma_{\xi i j} = \partial_{j} q_{\xi i} - \partial_{i} q_{\xi j} \quad \text{in} \, \, \mathbb{T}^{d}, \quad \fint_{\mathbb{T}^{d}} \sigma_{\xi i j} = 0.
	\end{align*}
By the uniqueness of $ \sigma_{ \xi i j } $, we learn that $ \sigma_{ \xi i j } = - \sigma_{ \xi j i } $. Hence the matrix $ \sigma_{ \xi } $ is skew-symmetric.  By the same argument as for $\phi_{\xi}$, we have that $R^{-1} ( \fint_{Q_R(x)} |\sigma_{\xi} - \fint_{Q_R(x)} \sigma_{\xi}|^{2} )^{\frac{1}{2}} \to 0$ uniformly with respect to $x \in \mathbb{R}^{d}$. 

Arguing as in \cite[Section 2.1]{JosienOtto}, one can show that the Helmholtz-type decomposition \eqref{eqn:a-helmholtz-decomposition} holds with the matrix $\bar{a}$ determined by the formula
	\begin{align*}
		\bar{a} \xi = \int_{\mathbb{T}^{d}} a(y) (\xi + \nabla \phi_{\xi}(y)) \, dy.
	\end{align*}
Furthermore, a variational argument along the lines of the one in Section \ref{S: bound on homogenized matrix} shows that $\lambda \text{Id} \leq \bar{a} \leq \Lambda \text{Id}$ (see also \cite{bensoussan_lions_papanicolaou}).

The previous two paragraphs prove that if $a$ is $\mathbb{Z}^{d}$-periodic, then it certainly satisfies the assumptions of Section \ref{S: deterministic assumptions} and
	\begin{align*}
		\lim_{R \to \infty} \sup_{x \in \mathbb{R}^{d}} \big( { \rm Sub } (Q_{R}(x)) + | \bar{a}(Q_{R}(x)) - \bar{a} | \big) = 0 .
	\end{align*}

Similarly, if $\theta$ is $\mathbb{Z}^{d}$-periodic, we define $\bar{\theta} = \int_{\mathbb{T}^{d}} \theta(y) \, dy$, which clearly satisfies $\theta_{*} \leq \bar{\theta} \leq \theta^{*}$. Observe that if we let $\psi \in H^{1}(\mathbb{T}^{d})$ denote the solution of the PDE
	\begin{align*}
		-\Delta \psi = \theta - \bar{\theta} \quad \text{in} \, \, \mathbb{T}^{d}, \quad \fint_{\mathbb{T}^{d}} \psi = 0,
	\end{align*}
then it is straightforward to check that, for any $x \in \mathbb{R}^{d}$ and any $R > 0$,
	\begin{equation}\label{eqn:osc-thing-periodic}
		\frac{1}{R^{2}} \|\theta - \bar{\theta}\|_{H^{-1}( Q_R(x)  )}^{2} \lesssim \frac{1}{R^{2}} \fint_{ Q_R(x) } |\nabla \psi|^{2} \lesssim \frac{1}{R^{2}} \fint_{\mathbb{T}^{d}} |\nabla \psi|^{2},
	\end{equation}
where the implied constants do not depend on $x$. 
Therefore,
	\begin{equation*}
		\lim_{R \to \infty} \sup_{x \in \mathbb{R}^{d}} { \rm Osc } ( Q_{R}(x) ) = \lim_{R \to \infty} \sup_{x \in \mathbb{R}^{d}} \frac{1}{R} \|\theta - \bar{\theta}\|_{H^{-1}({ Q_R(x) })} = 0.
	\end{equation*}

\subsection{Uniformly Almost Periodic Media} \label{S: uniformly almost periodic} In this section, we consider the case when $a$ or $\theta$ are uniformly almost periodic, or, more precisely, elements of $B^{\infty}(\mathbb{R}^{d})$.  We prove that it is possible to construct fields $(\phi_{Q},\sigma_{Q})$ and matrices $\bar{a}(Q)$ satisfying the assumptions of Section \ref{S: deterministic assumptions} in such a way that
 	\begin{align*}
		\lim_{R \to \infty} \sup_{x \in \mathbb{R}^{d}} \Big( { \rm Osc } ( Q_{R}(x) ) + { \rm Sub } ( Q_{R}(x) ) + | \bar{a}(Q_{R}(x)) - \bar{a} | \Big) = 0.
	\end{align*}

We begin by recalling the definition of the Bohr compactification, which provides a natural ergodic-theoretic setting for studying almost periodic functions.  For more information and references, see \cite[Section 3]{berg}, \cite{pankov}, or \cite[Chapter 7, Section 5]{katznelson}.

Let $\Omega_{\text{Bohr}}$ be the group of all complex (possibly discontinuous) multiplicative characters of $\mathbb{R}^{d}$:
	\begin{equation*}
		\Omega_{\text{Bohr}} = \left\{ \omega : \mathbb{R}^{d} \to \mathbb{C} \, \mid \, \omega(x + y) = \omega(x) \omega(y), \, \, |\omega(x)| = 1 \, \, \text{for each} \, \, x, y \in \mathbb{R}^{d} \right\}.
	\end{equation*} 
$\Omega_{\text{Bohr}}$ is an Abelian group under pointwise multiplication: $(\omega' \omega)(x) = \omega'(x) \omega(x)$.
As is standard, we equip $\Omega_{\text{Bohr}}$ with the product topology inherited from $\mathbb{C}^{\mathbb{R}^{d}}$. Hence by Tychonoff's theorem, $\Omega_{\text{Bohr}}$ is a compact topological Abelian group, referred to as the Bohr compactification of $\mathbb{R}^{d}$.

Note that $\mathbb{R}^{d}$ itself includes into $\Omega_{\text{Bohr}}$.  Indeed, define the group action $(\tau_{x})_{x \in \mathbb{R}^{d}}$ such that, for any $\omega \in \Omega_{\text{Bohr}}$ and any $x \in \mathbb{R}^{d}$, the character $\tau_{x} \omega$ is given by
	\begin{equation*}
		(\tau_{x} \omega)(y) = e^{i 2 \pi x \cdot y} \omega(y).
	\end{equation*}
Clearly, if $\mathbf{1} \in \Omega_{\text{Bohr}}$ denotes the constant function $\mathbf{1} \equiv 1$, then $x \mapsto \tau_{x} \mathbf{1}$ is an embedding (group monomorphism) of $\mathbb{R}^{d}$ into $\Omega_{\text{Bohr}}$.  It is possible to show that the set $\{\tau_{x} \mathbf{1}\}_{x \in \mathbb{R}^{d}}$ is dense in $\Omega_{\text{Bohr}}$, hence $\Omega_{\text{Bohr}}$ really is a topological compactification of $\mathbb{R}^{d}$. 

  Further, and most important for us, $ f \in B^{\infty}(\R^d) $\footnote{A similar equivalence holds relating $B^{p}(\mathbb{R}^{d})$ to $L^{p}(\Omega_{\text{Bohr}})$ in case $p < \infty$, see \cite{pankov}.} if and only if there exists an $F \in C(\Omega_{ \rm Bohr })$ such that 
	\begin{equation*}
		f(x) = F(\tau_{x}\mathbf{1}) \quad \text{for each} \quad x \in \mathbb{R}^{d}.
	\end{equation*}
Letting $\mathbb{P}_{\text{Bohr}}$ denote the normalized Haar measure on $\Omega_{\text{Bohr}}$, we have the identity
	\begin{equation*}
		\mathbb{E}_{\text{Bohr}}[F] = \lim_{R \to \infty} \fint_{B_{R}(x)} f,
	\end{equation*}
where the convergence is uniform with respect to the center point $x \in \mathbb{R}^{d}$.  From this, it follows that the action $(\tau_{x})_{x \in \mathbb{R}^{d}}$ is ergodic on $\Omega_{\text{Bohr}}$ equipped with its Borel $\sigma$-algebra and $\mathbb{P}_{\text{Bohr}}$.

We next establish that if $\theta \in B^{\infty}(\mathbb{R}^{d})$, then the quantity ${ \rm Osc }(\cdot)$ defined in \eqref{eqn:defn-osc} vanishes as $R \to \infty$ uniformly with respect to the center point $x$.

	\begin{prop} If $\theta \in B^{\infty}(\mathbb{R}^{d})$, then 
		\begin{equation*}
			\lim_{R \to \infty} \sup_{x \in \mathbb{R}^{d}} { \rm Osc }_{x}(R) = \lim_{R \to \infty} \sup_{x \in \mathbb{R}^{d}} { \rm Osc}(Q_{R}(x)) = 0.
		\end{equation*}
	\end{prop}
	
The proof of the proposition is a routine extension of the periodic case already covered in the previous subsection, which we include for completeness.  Here it is convenient to recall that the characters $\{e_{\lambda}\}_{\lambda \in \mathbb{R}^{d}}$ defined by $e_{\lambda}(x) = \exp \left( i 2 \pi \lambda \cdot x \right)$ densely span the space $B^{\infty}(\mathbb{R}^{d})$; see \cite[Chapter 3]{corduneanu} or \cite[Section 7.4]{jikov_kozlov_oleinik}.

\begin{proof}  Given a $\lambda \in \mathbb{R}^{d} \setminus \{0\}$, let $E_{\lambda} : \Omega_{\text{Bohr}} \to \mathbb{C}$ be defined by $E_{\lambda}(\omega) = \omega_{\lambda}$.  Observe that if we define $e_{\lambda}(x) = E_{\lambda}(\tau_{x} \mathbf{1})$, then $e_{\lambda}(x) = e^{i 2 \pi \lambda \cdot x}$ as in the previous paragraph.  Notice that $e_{\lambda}$ is periodic with respect to some lattice, which can be obtained by rotating and dilating $\mathbb{Z}^{d}$.  Thus, as in the previous section, since $e_{\lambda}$ has mean zero,
	\begin{align*}
		\lim_{R \to \infty} \sup_{x \in \mathbb{R}^{d}} \frac{1}{R} \|e_{\lambda}\|_{H^{-1}(Q_{R}(x))} = 0.
	\end{align*}

	Next, since $\theta  \in B^{\infty}(\mathbb{R}^{d})$, we have $\theta(x) = \Theta(\tau_{x} \omega)$ for some $\Theta \in C(\Omega_{\text{Bohr}})$ and $\bar{\theta} = \mathbb{E}_{\text{Bohr}}[\Theta]$ by definition.  
	
	Fix $\epsilon > 0$.  Since $\Theta - \bar{\theta}$ has mean zero, there are points $\lambda_1, \hdots, \lambda_N \in \R^{d} \setminus \{0\}$ and numbers $ \alpha_1, \hdots, \alpha_N \in \mathbb{C} $ such that
	\begin{align*}
	\sup_{ \omega \in \Omega_{\text{Bohr}}} \Big| \Theta(\omega) - \bar \theta - \sum_{ k = 1 }^N \alpha_k E_{\lambda_k}(\omega) \Big| < \epsilon,
	\end{align*}
	see \cite[Theorem 4.5]{corduneanu}.  In particular, $\left| \theta(x) - \bar{\theta} - \sum_{k = 1}^{N} \alpha_{k} e_{\lambda_{k}}(x) \right| < \epsilon$ for every $x \in \mathbb{R}^{d}$. Taken together with our choice of normalization for the $ H^{-1} $ norm, see \eqref{E: H minus one norm}, and the Poincaré inequality (which implies $R^{-1} \| \mathbf{1} \|_{H^{-1}(Q_{R})} \lesssim 1$), this implies 
	\begin{align*}
	\frac{1}{R} \| \theta - \bar \theta \|_{ H^{-1} ( Q_R(x) ) } \lesssim \epsilon + \sum_{ k = 1 }^N | \alpha_k | \frac{1}{R} \|  e_{\lambda_k} \|_{ H^{-1} ( Q_R(x) ) }.
	\end{align*}
We conclude upon sending first $R \to \infty$ and then $\epsilon \to 0$.
\end{proof}

It remains to analyze the case when $a \in B^{\infty}(\mathbb{R}^{d}; \text{Sym}(d))$, i.e., each entry of the matrix is in $B^{\infty}(\mathbb{R}^{d})$.  Here is where we use the generality of the assumptions in Section \ref{S: deterministic assumptions}.  Specifically, we work with finite-volume correctors rather than the infinite-volume ones we used in the random and periodic cases.  Technically, the reason we treat the almost periodic case differently is simple: it is not known whether or not the quantity ${ \rm Sub }_{x}(R)$ measuring the sublinearity of the infinite-volume corrector (i.e., as defined in \eqref{eqn:defn-sub}) vanishes uniformly with respect to the center point $x$.  As discussed further in Remark \ref{R: sub remark} below, this is a question of its own interest, related to open questions in the large-scale regularity theory for elliptic equations with almost periodic coefficients.

More specifically, in this setting, for any $\xi \in \mathbb{R}^{d}$ and any bounded, Lipschitz open set $U$, we let $\phi_{U,\xi}$ be the solution of the Dirichlet problem
    \begin{align*}
        - \nabla \cdot a(x) ( \xi + \nabla \phi_{U,\xi} ) = 0 \quad \text{in} \, \, U, \quad \phi_{U,\xi} = 0 \quad \text{on} \, \, \partial U.
    \end{align*}
By uniqueness, the map $\xi \mapsto \phi_{U,\xi}$ is linear, hence as before we define the vector-valued corrector $\phi_{U} = (\phi_{U,e_{1}},\dots,\phi_{U,e_{d}})$.
Below we will construct, for any cube $Q$, a field $\sigma_{Q} = ( \sigma_{Q,e_{1}},\dots,\sigma_{Q,e_{d}})$ and a matrix $\bar{a}(Q)$ in such a way that 
	\begin{align*}
		a(x) ( \xi + \nabla \phi_{Q,\xi} ) &= \bar{a}(Q) \xi + \nabla \cdot \sigma_{Q,\xi}
	\end{align*}
and
	\begin{align*}
		\lim_{R \to \infty} \sup_{x \in \mathbb{R}^{d}} \Big( { \rm Sub }(Q_{R}(x)) + | \bar{a}(Q_{R}(x)) - \bar{a} | \Big) = 0.
	\end{align*}
Here $\bar{a}$ is the homogenized matrix associated with the operator $\nabla \cdot a(x) \nabla$, which we will access using the variational approach.

First, we construct $\sigma_{Q,\xi}$ using the following lemma.  The difficulty lies in obtaining a suitable estimate for $\sigma_{Q,\xi}$, which is essentially folklore but is complicated by the fact that we work in a bounded domain.  Since we are not aware of a reference where this is done in detail, a proof is given at the end of this section (Section \ref{S: proof of lemma}).  In the statement, we denote by $\mathcal{C}$ the space of skew-symmetric matrix fields $\sigma' \in L^{2}(Q; \text{Skew}(d) )$ such that $\nabla \cdot \sigma' = 0$ weakly in $Q$.   

	\begin{lemma} \label{L: differential forms} Given a cube $Q \subseteq \mathbb{R}^{d}$ and a divergence-free vector field $b \in L^{2}(Q ; \mathbb{R}^{d} )$, there is a unique skew-symmetric matrix field $\sigma \in L^{2} (Q ; \text{Skew}(d) )$ such that 
		\begin{align} \label{E: sigma eqn differential form proof}
			\nabla \cdot \sigma = b \quad \text{in} \, \, Q \quad \text{and} \quad \int_{Q} \text{tr} ( \sigma(x)^{T} \sigma'(x) ) \, dx = 0 \quad \text{for each} \, \, \sigma' \in \mathcal{C}.
		\end{align}
	Further, there is a constant $C(d) > 0$ such that 
		\begin{align*}
			\fint_{Q} \text{tr} ( \sigma(x)^{T} \sigma(x) ) \, dx \leq C(d) \| b \|_{H^{1}(Q)'}^{2},
		\end{align*}
	where $\|\cdot\|_{H^{1}(Q)'}$ is the dual norm
		\begin{align*}
			\| b \|_{H^{1}(Q)'} = \sup \left\{ \fint_{ Q } b \cdot V \, \mid \, \int_{Q} |\nabla V|^{2} \, dx + |Q|^{-2/d} \int_{Q} |V|^{2} \, dx \leq 1  \right\}.
		\end{align*}
	\end{lemma}
	
We apply the lemma to obtain, for each cube $Q$ and each $\xi \in \mathbb{R}^{d}$, a unique skew-symmetric matrix field $\sigma_{Q,\xi}$ such that $\nabla \cdot \sigma_{Q,\xi} = a(x) (\xi + \nabla \phi_{Q,\xi}) - \bar{a}(Q) \xi$ in $Q$ and $\sigma_{Q,\xi}$ is orthogonal to divergence-free, skew-symmetric fields.  By uniqueness, the map $\xi \mapsto \sigma_{Q,\xi}$ is linear.

Finally, we establish uniform control over the homogenization error using variational techniques.  Given a bounded open set $U \subseteq \mathbb{R}^{d}$ and a vector $\xi \in \mathbb{R}^{d}$, let $\bar{a}(U)$ be the symmetric matrix such that
	\begin{align}
		\frac{1}{2} \bar{a}(U) \xi \cdot \xi &= \min \left\{ \frac{1}{2} \fint_{U} a(x) \nabla u \cdot \nabla u \, dx \, \mid \, u - \xi \cdot x \in H^{1}_{0}(U) \right\}. \label{E: dirichlet minimization}
	\end{align}
For the proof that the right-hand sides of the displays above define quadratic forms in the variable $\xi$, and thus $\bar{a}(U)$ and $\bar{a}_{*}(U)$ are well-defined, see \cite[Section 1.2]{armstrong_kuusi_book} or \cite{armstrong_kuusi_mourrat_book}.

Using the topology of $\Omega_{\text{Bohr}}$ in conjunction with results from the theory of $\Gamma$-convergence, we will prove that the quantities $ | \bar{a}(Q_{R}(x)) - \bar{a}|$ and $ { \rm Sub } ( Q_{R}(x) ) $ vanish as $R \to \infty$ uniformly with respect to the center point $x$.

	\begin{prop} \label{P: uniform convergence variational quantities} If $a \in B^{\infty}(\mathbb{R}^{d} ; \text{Sym}(d))$ and there are constants $0 < \lambda \leq \Lambda$ such that $\lambda \text{Id} \leq a \leq \Lambda \text{Id}$ pointwise, then
        		\begin{gather}
			\lim_{R \to \infty} \sup_{x \in \mathbb{R}^{d}} | \bar{a}(Q_{R}(x)) - \bar{a}|  = 0, \quad \text{and} \label{E: uniform convergence of a quantities} \\
            		\lim_{R \to \infty} \sup_{x \in \mathbb{R}^{d}} { \rm Sub } ( Q_{R}(x) ) = 0. \label{E: finite volume sub}
        		\end{gather}   
	\end{prop}

	\begin{proof} We first prove uniform convergence of $\bar{a}(Q)$, then we recall some facts from $\Gamma$-convergence and homogenization, and, finally, we prove the uniform convergence of ${ \rm Sub } (Q)$.  All the properties of $\Gamma$-convergence used in this proof are standard and can be found in the textbook \cite[Chapter 22]{DalMasoIntroGC}.
	
	\textit{Step 1: Uniform Convergence of Variational Quantities.} A change of perspective is useful.  Since $a \in B^{\infty}(\mathbb{R}^{d} ; \text{Sym}(d) )$, there is a continuous function $A : \Omega_{\text{Bohr}} \to \text{Sym}(d)$ such that $a(x) = A(\tau_{x}\mathbf{1})$ for each $x \in \mathbb{R}^{d}$.  It is convenient to define a symmetric matrix-valued set function $U \mapsto \bar{A}(U;\omega)$ , proceeding by analogy with $\bar{a}(U)$, via the formula
			\begin{align*}
				\frac{1}{2} \bar{A}(U;\omega) \xi \cdot \xi &= \min \left\{ \frac{1}{2} \fint_{U} A(\tau_{x} \omega) \nabla u \cdot \nabla u \, dx \, \mid \, u - \xi \cdot x \in H^{1}_{0}(U) \right\}.
			\end{align*} 
	As in \cite[Section 1.2]{armstrong_kuusi_mourrat_book}, $\lambda \text{Id} \leq \bar{A}(U;\omega) \leq \Lambda \text{Id}$ (as symmetric matrices) for any $U$ and $\omega$.
	Notice that $\bar{a}(U) = \bar{A}(U; \mathbf{1})$ for any $U$.  In this step of the proof, we will show that
		\begin{align*}
			\lim_{R \to \infty} \sup_{x \in \mathbb{R}^{d}} | \bar{A}(Q_{R}(x);\mathbf{1}) - \bar{a} | = 0.
		\end{align*}
	Since $A(x + U ; \omega ) = A(U ; \tau_{x} \omega)$ for any $U$ and $x$, it suffices to prove that 
		\begin{align*}
			\lim_{R \to \infty} \sup_{\omega \in \Omega_{\text{Bohr}}} | \bar{A}(Q_{R}; \omega) - \bar{a} | = 0.
		\end{align*}
		
	As in \cite{dal-maso_modica}, $\bar{A}(\cdot;\omega)$ is subadditive in the sense that if $U_{1},\dots,U_{N}$ are disjoint bounded open sets and $U = U_{1} \cup \dots \cup U_{N}$, then, for any $\omega \in \Omega_{\text{Bohr}}$,
			\begin{align*}
				A(U;\omega) \leq \sum_{i = 1}^{N} A(U_{i};\omega), \quad A_{*}(U;\omega)^{-1} \leq \sum_{i = 1}^{N} A_{*}(U_{i};\omega)^{-1},
			\end{align*}
		where again the inequalities hold at the level of quadratic forms.
		Therefore, the subadditive ergodic theorem implies that there is a (constant) symmetric matrix $\bar{A}$ such that
			\begin{align*}
				\bar{A}(Q_{R}; \omega) \to \bar{A} \quad \text{as} \, \, R \to \infty \quad \text{for} \, \,  \mathbb{P}_{\text{Bohr}}\text{-almost every} \, \, \omega \in \Omega.
			\end{align*}
		The classical result of \cite{dal-maso_modica} implies that $\bar{A} = \bar{a}$, the homogenized matrix.  It only remains to prove that the above convergence actually holds uniformly with respect to $\omega \in \Omega_{\text{Bohr}}$.

		Here we invoke the fact that $\Omega_{\text{Bohr}}$ is a compact Abelian topological group, and, therefore, the continuity of $A$ implies
			\begin{align*}
				\lim_{\omega' \to \mathbf{1}} \sup \left\{ |A( \omega' \omega) - A(\omega)| \, \mid \, \omega \in \Omega_{\text{Bohr}} \right\} = 0,
			\end{align*}
		see, e.g.,~\cite[Proposition 2.6]{FollandHarmonic}.  At the same time, by using the uniform estimate $\bar{A}(U; \omega) \leq \Lambda \text{Id}$, we see that, for any $\omega',\omega \in \Omega_{\text{Bohr}}$ and any bounded open set $U$,
			\begin{align*}
				| \bar{A} ( U ; \omega' \omega ) - \bar{A} ( U ; \omega ) | &\leq \Lambda \lambda^{-1} \sup \left\{ |A(\tau_{x} \omega' \omega) - A(\tau_{x} \omega)| \, \mid \, x \in \mathbb{R}^{d}, \, \, \omega \in \Omega_{\text{Bohr}} \right\} \\
					&= \Lambda \lambda^{-1} \sup \left\{ | A ( \omega' \omega'' ) - A( \omega'' ) | \, \mid \, \omega'' \in \Omega_{\text{Bohr}} \right\}.
			\end{align*}
		Thus,
			\begin{align*}
				\lim_{\omega' \to \mathbf{1}} \sup_{R > 0} \sup_{\omega \in \Omega_{\text{Bohr}}} | A(Q_{R}; \omega' \omega) - A(Q_{R}; \omega) | = 0,
			\end{align*}
		Since $\Omega_{\text{Bohr}}$ is compact, this implies $\bar{A}(Q_{R};\omega) \to \bar{a}$ uniformly with respect to $\omega$ as $R \to \infty$. 
		
	\textit{Step 2: $\Gamma$-convergence and Homogenization.} For any $\epsilon > 0$, let $\mathcal{E}_{\epsilon}$ be the energy functional defined by
		\begin{align*}
			\mathcal{E}_{\epsilon} ( u ; U ) = \frac{1}{2} \int_{ U } a ( \epsilon^{-1} x ) \nabla u \cdot \nabla u \, dx.
		\end{align*}
	Since $a \in B^{\infty}(\mathbb{R}^{d} ; \text{Sym}(d) )$, the theory of homogenization implies there is a symmetric matrix $\bar{a}$ such that if $\bar{\mathcal{E}}$ is the energy given by 
		\begin{align*}
			\bar{\mathcal{E}}(u;U) = \int_{U} \bar{a} \nabla u \cdot \nabla u \, dx,
		\end{align*}
	then $\mathcal{E}_{\epsilon}(\cdot;U)$ $\Gamma$-converges in $L^{2}(U)$ to $\bar{\mathcal{E}}(\cdot;U)$ as $\epsilon \to 0$ for any bounded open set $U$.
	
	The convergence above can be upgraded: For any sequence of scales $(\epsilon_{n})_{n \in \mathbb{N}}$ and points $(x_{n})_{n \in \mathbb{N}}$, if $\mathcal{E}_{n}$ is the shifted functional (abusing notation)
		\begin{align*}
			\mathcal{E}_{n} (u; U) = \mathcal{E}_{\epsilon_{n}} ( u( \cdot - x_{n} ) ; x_{n} + U),
		\end{align*}
	then compactenss and integral representation results from the theory of $\Gamma$-convergence (see \cite[Theorem 22.3]{DalMasoIntroGC}) imply that there is a subsequence $(\mathcal{E}_{n_{j}})_{j \in \mathbb{N}}$ and an energy functional $\tilde{\mathcal{E}}$ such that $\mathcal{E}_{n_{j}}(\cdot;U)$ $\Gamma$-converges to $\tilde{\mathcal{E}}(\cdot;U)$ as $j \to \infty$ for any bounded open set $U$ and $\tilde{\mathcal{E}}(u;U) = \frac{1}{2} \int_{U} \tilde{a}(x) \nabla u \cdot \nabla u \, dx$ for some bounded, positive definite matrix field $\tilde{a}$.    Since $\Gamma$-convergence implies convergence of solutions of the Dirichlet problem in this setting (see \cite[Theorem 22.9]{DalMasoIntroGC}), this implies that, for any $x \in \mathbb{R}^{d}$, $r > 0$, and $\xi \in \mathbb{R}^{d}$, if we let $y_{n} = \epsilon_{n} (x + x_{n})$, then
		\begin{align*}
			&\min \left\{ \tilde{\mathcal{E}}(u; Q_{r}(x)) \, \mid \, u(y) = \xi \cdot y \, \, \text{for each} \, \, y \in \partial Q_{r}(x) \right\} \\
			&\qquad = \lim_{j \to \infty} \min \left\{ \mathcal{E}_{n_{j}}(u; Q_{r}(x)) \, \mid \, u(y) = \xi \cdot y \, \, \text{for each} \, \, y \in \partial Q_{r}(x) \right\} \\
				&\qquad = \lim_{j \to \infty} \frac{1}{2} \bar{A}( Q_{\epsilon_{n_{j}}^{-1}r}(y_{n_{j}}) ; \mathbf{1} ) \xi \cdot \xi \\
				&\qquad = \lim_{j \to \infty} \frac{1}{2} \bar{A} ( Q_{\epsilon_{n_{j}}^{-1}r}(0) ; \tau_{y_{n}} \mathbf{1} ) \xi \cdot \xi = \frac{1}{2} \bar{a} \xi \cdot \xi.
		\end{align*}
	In view of the fact that this is true independent of the choice of the triple $(x,r,\xi)$, it follows that $\tilde{a}(y) = \bar{a}$ for almost every $y \in \mathbb{R}^{d}$.  This proves $\tilde{\mathcal{E}} = \bar{\mathcal{E}}$.
	
\textit{Step 3: Uniform Convergence of ${ \rm Sub } ( Q_{R}(x) )$.} Fix a sequence $(R_{n})_{n \in \mathbb{N}}$ of scales and a sequence of points $(x_{n})_{n \in \mathbb{N}}$ such that
	\begin{align*}
		\limsup_{R \to \infty} \sup_{x \in \mathbb{R}^{d}} { \rm Sub } ( Q_{R}(x) ) = \lim_{n \to \infty} { \rm Sub } ( Q_{R_{n}}(x_{n}) ).
	\end{align*}
We will show that the right-hand side is zero using what was proved in Step 2.

Let $\epsilon_{n} = R_{n}^{-1}$ for each $n$.  Associated to the sequences $(\epsilon_{n})_{n \in \mathbb{N}}$ and $(x_{n})_{n \in \mathbb{N}}$, define the sequence $( \mathcal{E}_{n} )_{n \in \mathbb{N}}$ as in Step 2.  As was shown in that step, $\mathcal{E}_{n}(\cdot;U)$ $\Gamma$-converges to $\bar{\mathcal{E}}(\cdot;U)$ as $n \to \infty$ for any fixed $U$.  

Given $\xi \in \mathbb{R}^{d}$, let $u_{n}$ be the solution of the PDE
	\begin{align*}
		- \nabla \cdot a(\epsilon_{n}^{-1} y + x_{n}) \nabla u_{n} = 0  \quad \text{in} \, \, Q_{1}, \quad u_{n}(y) = \xi \cdot y \quad \text{for each} \, \, y \in \partial Q_{1}.
	\end{align*}
By construction,
	\begin{align*}
		\frac{1}{R_{n}^{2}} \fint_{Q_{R_{n}}(x_{n})} | \phi_{Q_{R_{n}}(x_{n}),\xi} - \fint_{Q_{R_{n}}(x_{n})} \phi_{Q_{R_{n}}(x_{n}),\xi} |^{2} = \fint_{ Q_{1} } | u_{n} - \xi \cdot y - \fint_{ Q_{1} } u_{n} |^{2}.
	\end{align*}
At the same time, as a consequence of $\Gamma$-convergence, the sequence $(u_{n})_{n \in \mathbb{N}}$ converges in $L^{2}( Q_{1} )$ to the solution $\bar{u}$ of 
	\begin{align*}
		- \nabla \cdot \bar{a} \nabla \bar{u} = 0 \quad \text{in} \, \, Q_{1}, \quad \bar{u}(y) = \xi \cdot y \quad \text{for each} \, \, y \in \partial Q_{1}.
	\end{align*}
Since $\bar{a}$ is constant, the solution is precisely $\bar{u}(y) = \xi \cdot y$ for all $y$, hence
	\begin{align*}
		&\lim_{n \to \infty} \frac{1}{R_{n}^{2}} \fint_{Q_{R_{n}}(x_{n})} | \phi_{Q_{R_{n}}(x_{n}),\xi} - \fint_{Q_{R_{n}}(x_{n})} \phi_{Q_{R_{n}}(x_{n}),\xi} |^{2} \\
		&\qquad= \fint_{ Q_{1} } | \bar{u} - \xi \cdot y - \fint_{ Q_{1} } \bar{u} |^{2} = 0.
	\end{align*}
	
Similarly, at the level of $\sigma_{Q_{R_{n}}(x_{n})}$, it is well-known that homogenization implies the weak convergence of fluxes:
	\begin{align*}
		a ( \epsilon_{n}^{-1} y + x_{n} ) \nabla u_{n} \rightharpoonup \bar{a} \xi \quad \text{in} \, \, L^{2}(Q_{1}; \mathbb{R}^{d} ) \quad \text{as} \, \, n \to \infty.
	\end{align*}
Since weak convergence in $L^{2}(Q_{1}; \mathbb{R}^{d})$ implies strong convergence with respect to the norm $\|\cdot\|_{H^{1}(Q_{1})'}$ defined in the statement of Lemma \ref{L: differential forms}, this combines with the estimate in said lemma and \eqref{E: uniform convergence of a quantities} to yield
	\begin{align*}
		\lim_{n \to \infty} \frac{1}{R_{n}^{2}} \fint_{ Q_{R_{n}}(x) } | \sigma_{Q_{R_{n}}(x_{n}),\xi} |^{2} \, dy = 0.
	\end{align*}
Since $\xi$ was arbitrary, by linearity, we obtain ${ \rm Sub } ( Q_{R_{n}}(x_{n}) ) \to 0$ as desired.
\end{proof}

\begin{remark} \label{R: sub remark} In the previous proposition, we proved that ${\rm Sub }(Q)$, which measures the sublinearity of the finite-volume corrector $\phi_{Q}$, converges uniformly to zero as $|Q| \to \infty$.  It is not clear that the same could be said if finite-volume correctors were replaced by the infinite-volume corrector.  This is the reason for using correctors $(\phi_{Q},\sigma_{Q})$ and matrices $\bar{a}(Q)$ that depend on $Q$ in this section.

If it were known that, in the uniformly almost periodic setting, the quantity ${ \rm Sub }_{x}(R) $ associated with the infinite volume corrector (see \eqref{eqn:defn-sub} above) vanished uniformly with respect to $x$ as $R \to \infty$, then this fact could be combined with the approach of \cite{gloria_neukamm_otto} to prove (uniform-in-space) large-scale Lipschitz estimates for solutions of the corresponding elliptic equation.  To the best of the authors' knowledge, whether or not such estimates hold for arbitrary uniformly almost periodic coefficients remains an open question.  (This is highlighted as an open problem in \cite{armstrong_shen}, where Lipschitz estimates are proved under slightly stronger assumptions on $a$.)  Conversely, if true, such estimates could be used in conjunction with the techniques of \cite{armstrong_kuusi_book} to prove the uniform convergence of ${ \rm Sub }_{x} (R)$.    \end{remark}

\subsection{Proof of Corollary \ref{cor:periodic-homogenization}} \label{S: proof of periodic corollary} By combining the observations of the previous two subsections with the results of Section \ref{S: homogenization 1}, we readily obtain the proof of Corollary \ref{cor:periodic-homogenization}, that is, the $\Gamma$-convergence of $\mathscr{F}_{\epsilon,\delta}$ to $\bar{\mathscr{E}}$ for any scale $\epsilon \mapsto \delta(\epsilon)$ satisfying $\delta(\epsilon) \ll \epsilon$.

	\begin{proof}[Proof of Corollary \ref{cor:periodic-homogenization}] Assume that both $a$ and $\theta$ are either $\mathbb{Z}^{d}$-periodic or functions in $B^{\infty}(\mathbb{R}^{d})$, and let $\epsilon \mapsto \delta(\epsilon)$ be any scale such that $\epsilon^{-1} \delta(\epsilon) \to 0$ as $\epsilon \to 0$.  As in the stochastic setting, we need to show that, for any $x \in \mathbb{R}^{d}$, any $\varrho > 0$, and any $e \in S^{d-1}$,
		\begin{align*}
			\lim_{\epsilon \to 0} \min \left\{ \mathscr{F}_{\epsilon,\delta(\epsilon)}(u; Q_{\varrho}(x)) \, \mid \, u(y) = q(\epsilon^{-1}(y - x) \cdot e) \, \, \text{for} \, \, y \in \partial Q_{\varrho}(x) \right\} = \bar{\sigma}(e).
		\end{align*}
	As before, we restrict attention to the case $e = e_{1}$. Otherwise, one can check that composition with a rotation preserves $B^{\infty}(\mathbb{R}^{d})$, while such a rotation would transform a $\mathbb{Z}^{d}$-periodic function into a function periodic with respect to a rotated lattice. (It is a coincidence that $ e_1 $ is a lattice direction. Recall from Section \ref{S: periodic} that the correctors are defined for arbitrary directions $\xi$.) In either case, the relevant facts from Sections \ref{S: periodic} and \ref{S: uniformly almost periodic} are preserved.
	
	The proof of the upper bound (Theorem \ref{T: planar homogenization theorem upper bound}) carries over to the present setting, as is discussed briefly in Remark \ref{R: upper bound periodic almost periodic} in Appendix \ref{A: upper bound} below.  In particular, this means that 
		\begin{align*}
			\limsup_{\epsilon \to 0} \min \left\{ \mathscr{F}_{\epsilon,\delta(\epsilon)}(u; Q_{\varrho}(x)) \, \mid \, u(y) = q(\epsilon^{-1}(y - x) \cdot e) \, \, \text{for} \, \, y \in \partial Q_{\varrho}(x) \right\} \leq \bar{\sigma}(e).
		\end{align*}
		
	It only remains to check the lower bound.  Toward this end, as in Section \ref{S: homogenization 1}, we rescale space by a factor $R \coloneqq \epsilon^{-1}$ and introduce the parameter $\gamma(R) \coloneqq \epsilon^{-1} \delta(\epsilon)$, so that now the goal is to prove
		\begin{align*}
			&\liminf_{R \to \infty} \min \left\{ \mathscr{F}_{\gamma(R)}(u; Q_{\varrho R}(Rx) ) \, \mid \, u(y) = q((y - Rx) \cdot e_{1}) \, \, \text{for} \, \, y \in \partial Q_{\varrho R}(Rx) \right\} \\
			& \qquad \qquad \qquad \qquad \qquad \qquad \qquad \qquad \qquad \qquad \qquad \qquad \qquad \qquad \qquad \qquad \geq \bar{\sigma}(e_{1}).
		\end{align*}
		
	As in the proof of Theorem \ref{T: planar homogenization theorem lower bound}, there is no loss of generality assuming that $\varrho R = (2K+1) r_{c}$ for some $K \in \mathbb{N}$.  Fix a minimizer $u$ of the variational problem above.  According to Theorem \ref{T: elliptic term cube decomposition}, there is an $R$-independent modulus of continuity $\omega$ and a function $\tilde{u} \in H^{1}(Q_{\varrho R}(Rx))$ such that
		\begin{align*}
			&\frac{ \overline{\mathscr{F}}(\tilde{u}; Q_{\varrho R}(Rx)) }{ \mathscr{F}_{\gamma(R)}(u; Q_{\varrho R}(Rx) ) } -1 \\
			&\qquad \qquad \leq \sup_{ x \in \R^d } \omega \Big(  {\rm Sub} ( Q_{ r_{c} \gamma(R)^{-1} } ( x ) ) + {\rm Osc} ( Q_{ r_{c} \gamma(R)^{-1} } ( x ) ) + | \bar{a}(Q_{r_{c} \gamma(R)^{-1} } ( x ) - \bar{a} | \Big) .
		\end{align*}
	In view of the discussion of Sections \ref{S: periodic} and \ref{S: uniformly almost periodic} above, our construction of the fields $(\phi_{Q},\sigma_{Q})$ and the matrices $\bar{a}(Q)$ for cubes $Q$ yields that
		\begin{align*}
			\lim_{R \to \infty} \sup_{ x \in \R^d } \omega \Big(  {\rm Sub} ( Q_{ r_{c} \gamma(R)^{-1} } ( x ) ) + {\rm Osc} ( Q_{ r_{c} \gamma(R)^{-1} } ( x ) ) + | \bar{a}(Q_{r_{c} \gamma(R)^{-1} } ( x ) - \bar{a} | \Big) = 0,
		\end{align*}
	hence, upon invoking $\Gamma$-convergence results for the spatially homogeneous functional $\overline{\mathscr{F}}$ (see \cite[Theorem 3.7]{ansini_braides_chiado-piat}), we find
		\begin{align*}
			&\liminf_{R \to \infty} \min \left\{ \mathscr{F}_{\gamma(R)}(u; Q_{\varrho R}(Rx) ) \, \mid \, u(y) = q((y - Rx) \cdot e_{1}) \, \, \text{for} \, \, y \in \partial Q_{\varrho R}(Rx) \right\} \\
				&\qquad \geq \lim_{R \to \infty} \min \left\{ \overline{\mathscr{F}}(u; Q_{\varrho R}(Rx) ) \, \mid \, u(y) = q((y - Rx) \cdot e_{1}) \, \, \text{for} \, \, y \in \partial Q_{\varrho R}(Rx) \right\} \\
				&\qquad = \bar{\sigma}(e_{1}).
		\end{align*}
	\end{proof}
	
\subsection{Proof of Lemma \ref{L: differential forms}} \label{S: proof of lemma} Finally, we address the construction of $\sigma_{Q}$ from Section \ref{S: uniformly almost periodic}.  We treat the equation $\nabla \cdot \sigma = b$ as an equation for differential forms and then quantify the dependence of $\sigma$ on $b$ following \cite{dario}.
	
	\begin{proof}[Proof of Lemma \ref{L: differential forms}] By scaling, it suffices to prove the result when $|Q| = 1$, which we assume henceforth for simplicity.
	
	After suitably reformulating the problem, it becomes a question about differential forms.  To reformulate,  we pass from vector fields and skew-symmetric matrix fields (interpreted as $1$- and $2$-forms, respectively) to $(d - 1)$- and $(d - 2)$-forms using the Hodge star operator $\star$.  The definition of the Hodge star operator in full generality can be found in \cite{dario}, but does not concern us here.  For our purposes, it suffices to know that if $b$ is a vector field in $Q$, and if $B = \star b$ is the corresponding $(d - 1)$-form, then
			\begin{align*}
				B = \sum_{j = 1}^{d} (-1)^{j + 1} b_{j} dx_{1} \wedge \cdots \wedge dx_{j - 1} \wedge dx_{j + 1} \wedge \cdots \wedge dx_{d},
			\end{align*}
		and, thus,
			\begin{align*}
				dB = 0 \quad \text{if and only if} \quad \nabla \cdot b = 0.
			\end{align*}
		Similarly, if $\sigma$ is a skew-symmetric matrix field, and if $\Sigma = \star \sigma$ is the corresponding $(d - 2)$-form, then
			\begin{align*}
				\Sigma = \sum_{i < j} (-1)^{i + j + 1} \sigma_{ij} dx_{1} \wedge \cdots \wedge dx_{i - 1} \wedge dx_{i + 1} \wedge \cdots \wedge dx_{j - 1} \wedge dx_{j + 1} \wedge \cdots \wedge dx_{d},
			\end{align*}
		from which a straightforward computation shows that 
			\begin{align*}
				d \Sigma = - B \quad \text{if and only if} \quad \nabla \cdot \sigma = b.
			\end{align*}
		Finally, we will use the fact that, provided one treats the exterior power $\wedge^{k} \mathbb{R}^{d}$ for a given $k$ as an inner product space with orthonormal basis $\{ dx_{i_{1}} \wedge \cdots \wedge dx_{i_{k}} \, \mid \, 1 \leq i_{1} < \cdots < i_{k} \leq d\}$, then the Hodge star operator is an isometry between $\wedge^{k} \mathbb{R}^{d}$ and $\wedge^{d- k} \mathbb{R}^{d}$.  In particular, a skew-symmetric matrix field $\sigma$ is perpendicular to $\mathcal{C}$ as in the second half of \eqref{E: sigma eqn differential form proof} if and only if the corresponding $(d - 2)$-form $\Sigma = \star \sigma$ is perpendicular to the space of weakly closed, square-integrable $(d - 2)$-forms, that is,
			\begin{gather*}
				 \langle \Sigma, \Sigma' \rangle_{L^{2} ( Q )} = 0 \quad \text{for each} \, \, \Sigma' \in \mathcal{C}_{\star}, \\
                 \text{where} \quad \mathcal{C}_{\star} = \{ \Sigma' \in L^{2} ( Q ; \wedge^{d - 2} \mathbb{R}^{d} ) \, \mid \, d \Sigma' = 0 \}.
			\end{gather*}

		In view of these correspondences, we seek to prove that for a given $(d-1)$-form $B \in L^{2} ( Q ; \wedge^{d - 1} \mathbb{R}^{d} )$ satisfying $dB = 0$, there is a unique $\Sigma \in L^{2} ( Q ; \wedge^{d - 2} \mathbb{R}^{d} )$ such that 
			\begin{align} \label{E: differential forms equation}
				d\Sigma = B \quad \text{in} \, \, Q \quad \text{and} \quad \langle \Sigma, \Sigma' \rangle_{L^{2}(Q)} = 0 \quad \text{for each} \, \, \Sigma' \in \mathcal{C}_{\star}
			\end{align}
		and for which the following estimate holds
			\begin{align} \label{E: differential forms estimate}
				\| \Sigma \|_{L^{2}(Q)} \leq C(d) \| B \|_{H^{-1}(Q)},
			\end{align}
        where $\|\cdot\|_{H^{1}(Q)'}$ is the norm 
            \begin{align*}
                \|V\|_{H^{1}(Q)'} &= \sup \left\{ \langle V, V' \rangle_{L^{2}(Q)} \, \mid \, V' \in H^{1}(Q;\Lambda^{d-1}\mathbb{R}^{d}) \quad \text{such that} \right. \\  
                   &\qquad\left. \| \nabla V' \|^{2}_{L^{2}(Q)} + \| V' \|^{2}_{L^{2}(Q)} \leq 1 \right\}.
            \end{align*}
		
		Concerning existence, since $Q$ is convex, it is well-known that if $B$ is smooth and satisfies $dB = 0$, then there is a smooth $\Sigma$ such that $d \Sigma = B$ in $Q$.  If $B$ is merely square integrable with $dB = 0$, then \cite[Proposition 2.3]{dario} implies there is a $\Sigma \in H^{1}( Q ; \wedge^{d-2} \mathbb{R}^{d} )$ such that $d \Sigma = B$.  Further, we can invoke a version of the Poincar\'{e} inequality, namely,  \cite[Proposition 3.6]{dario} to find that 
			\begin{align*}
				\inf \left\{ \| \Sigma - \Sigma' \|_{L^{2}(Q)} \, \mid \, \Sigma' \in \mathcal{C}_{\star} \right\} \leq C(d) \| B \|_{ H^{-1} ( Q ) }.
			\end{align*}
	Notice that if $\Sigma' \in \mathcal{C}_{\star}$, then $d ( \Sigma - \Sigma' ) = B$.  Therefore, the orthogonal projection $\Sigma^{\perp}$ of $\Sigma$ perpendicular to $\mathcal{C}_{\star}$ satisfies $d \Sigma^{\perp} = B$ and $\|\Sigma^{\perp}\|_{L^{2}(Q)} \leq C(d) \|B\|_{ H^{-1}(Q) }$.  In particular, replacing $\Sigma$ by $\Sigma^{\perp}$, we obtain a solution of \eqref{E: differential forms equation} and \eqref{E: differential forms estimate} holds.  
		
		Finally, if $\Sigma_{1}$ and $\Sigma_{2}$ are two solutions of \eqref{E: differential forms equation} with the same right-hand side $B$, then $d ( \Sigma_{2} - \Sigma_{1} ) = 0$ so $\langle \Sigma_{1} - \Sigma_{2}, \Sigma_{1} - \Sigma_{2} \rangle_{L^{2}(Q)} = 0$.  Thus, the solution is unique.
		\end{proof}

\appendix

\section{Homogenization Upper Bound} \label{A: upper bound}

In this appendix, we prove the homogenization upper bound, Theorem \ref{T: planar homogenization theorem upper bound}, which shows that the homogenized energy is always an upper bound in the regime where $\delta(\epsilon) \ll \epsilon$.  It is important to note that in this appendix, while we continue to assume that $(a,\theta)$ satisfy the assumptions of Section \ref{S: assumptions}, we only assume that $W : \mathbb{R} \to [0,\infty)$ is continuous and 
    \begin{align*}
        \{ u \in \mathbb{R} \, \mid \, W(u) = 0 \} = \{-1,1\}.
    \end{align*}

The proof will be presented in two steps.  First, we consider the special case when $x_{0} = 0$, in which case we can readily apply an ergodic theorem to obtain the desired result.  Next, when $x_{0} \neq 0$, we once again appeal to the ergodic theorem to argue that the behavior near $x_{0}$ looks sufficiently similar to that at the origin.

In the first step, we will use the following variant of the standard ergodic theorem.  As in \cite{morfe}, we are motivated to consider averages over a $(d - 1)$-dimensional group due to the fact that the energy scales like $R^{d-1}$.  

	\begin{prop} \label{P: ergodic theorem continuous egorov} Assume that $(\Omega,\mathcal{F},\mathbb{P})$ is a probability space supporting a measurable group action $\{\tau_{x}\}_{x \in \mathbb{R}^{d}}$ and satisfying the assumptions of Section \ref{S: assumptions}.  Let $(X_{\gamma})_{\gamma > 0}$ be a process (one-parameter family of $\mathcal{F}$-measurable random variables) such that the function $\gamma \mapsto X_{\gamma}$ is continuous on an event of probability one and for which there is a constant $C_{X} > 0$ such that, with probability one,
	\begin{equation*}
		|X_{\gamma}| \leq C_{X} \quad \text{for each} \, \, \gamma > 0.
	\end{equation*}  
If there is a constant $\bar{X} \in \mathbb{R}$ such that 
	\begin{equation*}
		\lim_{\gamma \to 0} X_{\gamma} = \bar{X} \quad \text{almost surely,}
	\end{equation*}
then, for any $L \in \mathbb{N}$,
	\begin{equation*}
		\lim_{(R^{-1},\gamma) \to (0,0)} \left( \frac{L}{R} \right)^{d-1}\sum_{k \in (\{0\} \times L\mathbb{Z}^{d-1}) \cap Q_{R}} X_{\gamma} \circ \tau_{k} = \bar{X} \quad \text{almost surely.}
	\end{equation*}
\end{prop}

In the above proposition, we  impose a continuity assumption in the variable $\gamma$ only because it is real-valued.  This allows us to sidestep measurability issues.  Such an assumption would not be necessary if we worked with sequences $(\delta_{j})_{j \in \mathbb{N}}$ rather than functions $\epsilon \mapsto \delta(\epsilon)$.

The proof of Proposition \ref{P: ergodic theorem continuous egorov} is deferred to the end of this appendix.  For now, we use the proposition to prove Theorem \ref{T: planar homogenization theorem upper bound} in the special case when $x_{0} = 0$.

	\begin{proof}[Proof of Theorem \ref{T: planar homogenization theorem upper bound} (case: $x_{0} = 0$)] Fix an arbitrary (deterministic) scaling function $\epsilon \mapsto \delta(\epsilon)$ such that $\epsilon^{-1} \delta(\epsilon) \to 0$ as $\epsilon \to 0$.  As in Section \ref{S: homogenization proof} and \eqref{eqn:defn-min}, we will use the following shorthand notation:
        \begin{equation*}
            m( \mathscr{F} , U , q ) = \inf \left\{ \mathscr{F}(u;U) \, \mid \, u \in H^{1}(U), \, \, u(y) = q(y\cdot e_{1}) \, \, \text{for each} \, \, y \in \partial U \right\}.
        \end{equation*}
        Recall that we aim to prove that if $q$ satisfies the assumptions \eqref{E: finite width} and \eqref{E: finite width 2} of Section \ref{section: homogenization introduction}, then
		\begin{equation*}
			\limsup_{ (R^{-1},\gamma) \to (0,0) } R^{1-d} m(\mathscr{F}_{\gamma}, Q_{\varrho R}, q) \leq \bar{\sigma}(e_{1}) \varrho^{d-1} \quad \text{almost surely,}
		\end{equation*}
	where $\bar{\sigma}(e_{1})$ is the positive number with $\bar{\sigma}(e_{1})^{2} = \sigma_{W}^{2} \bar{\theta} e_{1} \cdot \bar{a} e_{1}$ and $\sigma_{W}$ is given by \eqref{E: homogeneous surface tension}.
	
	Up to replacing the medium $(a,\theta)$ by the rescaled one $(a(\varrho \cdot), \theta(\varrho \cdot))$, we can and will assume without loss of generality that $\varrho = 1$.  
	
	In the proof that follows, we will work with the cylinder $\mathcal{C}(R,h) \subseteq \mathbb{R}^{d}$ defined for $R, h > 0$ by 
	\begin{align*}
		\mathcal{C}(R,h) = \left\{ y \in \mathbb{R}^{d} \, \mid \, |y \cdot e_{1}| < \frac{h}{2}, \, \, |y \cdot e_{2}|, \dots, |y \cdot e_{d}| < \frac{R}{2} \right\}.
	\end{align*}

We first observe that, as in \cite[Proposition 5]{morfe}, there is a modulus of continuity $\omega : [0,\infty) \to [0,\infty)$ depending only on $q$, $\Lambda$, and $\theta^{*}$ such that 
	\begin{align} \label{E: almost monotone}
		&m(\mathscr{F}_{\gamma}, Q_{R}, q) \leq m(\mathscr{F}_{\gamma}, \mathcal{C}(R,h), q) + R^{1-d} \omega(h^{-1}).
	\end{align}
In what follows, we take advantage of the fact that $m(\mathscr{F}_{\gamma},\mathcal{C}(R,h),q)$ has a natural subadditive structure in the transversal variables $(y \cdot e_{2},\dots,y \cdot e_{d})$. 

Specifically, we fix a discretization parameter $L \in \mathbb{N}$ and discretize the cylinder $\mathcal{C}(R,h)$ using the subcylinders $\{k + \mathcal{C}(L,h) \, \mid \, k \in (\{0\} \times L\mathbb{Z}^{d-1}) \cap Q_{R}\}$.  It is convenient to define the discretization $\mathcal{C}_{L}(R,h)$ by 
	\begin{align*}
		\mathcal{C}_{L}(R,h) = \bigcup_{k \in (\{0\} \times L \mathbb{Z}^{d-1}) \cap Q_{R}} (k + \mathcal{C}(L,h)).
	\end{align*}
By subadditivity (see \cite[Proposition 2]{morfe}), we can write
	\begin{align*}
		m(\mathscr{F}_{\gamma},\mathcal{C}(R,h), q) &\leq m(\mathscr{F}_{\gamma}, \mathcal{C}(R,h) \setminus \mathcal{C}_{L}(R,h), q) \\
			&\qquad + \sum_{k \in (\{0\} \times L\mathbb{Z}^{d-1}) \cap Q_{R}} m(\mathscr{F}_{\gamma}, k + \mathcal{C}(L,h),q).
	\end{align*}
Observe that 
	\begin{align*}
		m(\mathscr{F}_{\gamma}, \mathcal{C}(R,h) \setminus \mathcal{C}_{L}(R,h), q) &\leq \mathscr{F}_{\gamma}(q(y \cdot e_{1}); \mathcal{C}(R,h) \setminus \mathcal{C}_{L}(R,h)) \\
			&\leq \int_{\mathcal{C}(R,h) \setminus \mathcal{C}_{L}(R,h)} \left( \frac{\Lambda}{2} q'(y \cdot e_{1})^{2} + \theta^{*} W(q(y \cdot e_{1})) \right) \, dy.
	\end{align*}
From this, we deduce that, with probability one, for any $L \in \mathbb{N}$,
	\begin{align*}
		\lim_{R \to \infty} \sup_{ \gamma } \frac{m(\mathscr{F}_{\gamma}, \mathcal{C}(R,h) \setminus \mathcal{C}_{L}(R,h),q)}{R^{d-1}} = 0.
	\end{align*}
It only remains to analyze the terms in the sum.  This is where we apply Proposition \ref{P: ergodic theorem continuous egorov}.

Define the process $(M_{\gamma})_{\gamma > 0}$ by setting
	\begin{equation*}
		M_{\gamma}= m(\mathscr{F}_{\gamma}, \mathcal{C}(L,h), q)
	\end{equation*}
Recall that the pointwise bounds $\lambda \text{Id} \leq a \leq \Lambda \text{Id}$ and $\theta_{*} \leq \theta \leq \theta^{*}$ hold Lebesgue almost everywhere in $\mathbb{R}^{d}$ with probability one.  Thus, it is straightforward to check that, by the direct method of the calculus of variations, the function $\gamma \mapsto M_{\gamma}$ is continuous in $(0,\infty)$ almost surely.

At the same time, recall that, with probability one, the gradient functional
	\begin{equation*}
		\int_{\mathcal{C}(L,h)}  a(\gamma^{-1} x) \nabla u \cdot \nabla u \, dx
	\end{equation*}
$\Gamma$-converges with respect to the weak topology of $H^{1}(\mathcal{C}(L,h))$) as $\gamma \to 0$ to the homogenized functional 
	\begin{equation*}
		\int_{\mathcal{C}(L,h)} \bar{a} \nabla u \cdot \nabla u \, dx .
	\end{equation*}
By \cite[Theorem 21.1]{DalMasoIntroGC}, this statement remains true if $H^{1}(\mathcal{C}(L,h))$ is replaced by $q + H^{1}_{0}(\mathcal{C}(L,h))$, i.e., the set of functions $u$ with $u(y) = q(y\cdot e_{1})$ on $\partial \mathcal{C}(L,h)$.  Further, the potential well functional
	\begin{equation*}
		\int_{\mathcal{C}(L,h)} \theta(\gamma^{-1} x) \tilde{W}(u) \, dx, \quad \text{with} \, \, \tilde{W}(u) = \left\{ \begin{array}{r l}
        W(u), & \text{if} \, \, |u| \leq 1, \\
            |u|^{2} - 1, & \text{if} \, \, |u| > 1 ,
        \end{array} \right.
	\end{equation*}
is uniformly continuous on bounded subsets of $H^{1}(\mathcal{C}(L,h))$ (with a deterministic modulus of continuity, by the boundedness of $\theta$) and, thus, by the ergodic theorem, with probability one, it converges uniformly on bounded subsets of $H^{1}(\mathcal{C}(L,h))$ as $\gamma \to 0$ to the averaged functional
	\begin{equation*}
		\bar{\theta} \int_{\mathcal{C}(L,h)} \tilde{W}(u) \, dx.
	\end{equation*}
Since by Lemma \ref{L: minus one one} and its proof, the value of $M_{\gamma}$ is unchanged if we replace $W$ by $\tilde{W}$ in the definition of $\mathscr{F}_{\gamma}$, we can invoke the standard perturbation result in the theory of $\Gamma$-convergence (see \cite[Proposition 6.20]{DalMasoIntroGC}) to find
	\begin{equation*}
		\lim_{\gamma \to 0} M_{\gamma} = \bar{M} \quad \text{almost surely,}
	\end{equation*}
where 
	\begin{align*}
		\bar{M} &= \min \left\{ \overline{\mathscr{F}}(u, \mathcal{C}(L,h)) \, \mid \, u(y) = q(y \cdot e_{1}) \, \, \text{if} \, \, y \in \partial \mathcal{C}(1,h) \right\} = m( \overline{\mathscr{F}}, \mathcal{C}(L,h), q).
	\end{align*}
Thus, by Proposition \ref{P: ergodic theorem continuous egorov}, with probability one,
	\begin{align*}
		\lim_{(R^{-1},\gamma) \to (0,0)} \frac{1}{R^{d-1}}\sum_{k \in (\{0\} \times L\mathbb{Z}^{d-1}) \cap Q_{R}} m(\mathscr{F}_{\gamma}, \mathcal{C}(L,h), q) = L^{1-d} m(\overline{\mathscr{F}},\mathcal{C}(L,h),q) .
	\end{align*}
	
Combining the last limit with \eqref{E: almost monotone}, we conclude that, with probability one, 
	\begin{align*}
		\limsup_{ ( R^{-1}, \gamma ) \to ( 0, 0 ) } R^{1-d} m(\mathscr{F}_{\gamma}, Q_{R},q) \leq L^{1-d} m(\overline{\mathscr{F}}, \mathcal{C}(L,h), q) + \omega(h^{-1}).
	\end{align*}
Setting $h = L$ and sending $L \to \infty$, well-known results on the spatially homogeneous Allen-Cahn functional (see \cite[Theorem 3.7]{ansini_braides_chiado-piat}) imply that
	\begin{align*}
		\limsup_{ (R^{-1},\gamma) \to (0,0) } R^{1-d} m(\mathscr{F}_{\gamma},Q_{R},q) \leq \lim_{L \to \infty} L^{1-d} m(\overline{\mathscr{F}},\mathcal{C}(L,L),q) = \bar{\sigma}(e_{1}) \quad \text{almost surely.}
	\end{align*}
\end{proof}

Next, arguing as in \cite{morfe}, we extend to the case where the center $x_{0}$ is nonzero.  

In this step of the proof, we use the following continuous variant of Egorov's Theorem: If $(X_{y})_{y \in Y}$ is a random process indexed by some open set $Y \subseteq \mathbb{R}^{2}$ such that $y \mapsto X_{y}$ is continuous and $X_{y} \to 0$ as $|y| \to 0$ with probability one, then, for any $\delta > 0$, there is an event $E_{\delta}$ and a modulus of continuity $\omega_{\delta} : [0,\infty) \to [0,\infty)$ such that $\mathbb{P}(E_{\delta}) > 1 - \delta$ and
	\begin{align*}
		\sup_{E_{\delta}} |X_{y}| \leq \omega_{\delta}(|y|).
	\end{align*}
This follows from an application of the standard version of Egorov's Theorem to the discrete sequence $(X^{*}_{K})_{K \in \mathbb{N}}$ given by $X^{*}_{K} = \sup_{|y| \leq K^{-1}} |X_{y}|$.  (The continuity assumption is only imposed to ensure the $X^{*}_{K}$ is measurable for each $K$.)

\begin{proof}[Proof of Theorem \ref{T: planar homogenization theorem upper bound} (case: $x_{0} \neq 0$)] As in the $x_{0} = 0$ case, we assume without loss of generality that $\varrho = 1$.  

We begin by invoking the ergodic theorem.  Notice that the two-parameter random process
	\begin{align*}
		(R,\gamma) \mapsto \min \left\{ \mathscr{F}_{\gamma}(u; Q_{R}) \, \mid \, u(y) = q(y \cdot e_{1}) \, \, \text{for each} \, \, y\in  \partial Q_{R} \right\}
	\end{align*}
is continuous, as follows once again from an application of the direct method of the calculus of variations.  Thus, by the $x_{0} = 0$ case of the theorem and the continuous variant of Egorov's Theorem mentioned just before the start of this proof, there is an event $E \in \mathcal{F}$ and a modulus of continuity $\eta : [0,\infty) \to [0,\infty)$ such that $\mathbb{P}(E) \geq \frac{1}{2}$ and
	\begin{align*}
		\sup_{E} \left( R^{1-d} \min \left\{ \mathscr{F}_{\gamma}(u; Q_{R}) \, \mid \, u(y) = q(y \cdot e_{1}) \, \, \text{on} \, \, \partial Q_{R} \right\} - \bar{\sigma}(e_{1}) \right) \leq \eta(\gamma + R^{-1})
	\end{align*}
for any $R, \gamma > 0$. Let $\alpha \in (0,1/2)$ be a small parameter.  By the ergodic theorem, with probability one,
	\begin{align*}
		\lim_{R \to \infty} \frac{|\{x \in Q_{2\alpha R}(Rx_{0}) \, \mid \, \tau_{x} \omega \in E \}|}{R^{d}} = (2\alpha)^{d} \mathbb{P}(E) > 0.
	\end{align*}
Thus, we can fix random points $(X_{R})_{R > 0}$ such that $X_{R} \in Q_{2\alpha R}(Rx_{0})$ for any $R > 0$ and, for large enough $R$, $\tau_{X_{R}} \omega \in E$. 

The remainder of the argument is deterministic.  By the definition of $E$, for any $R, \gamma > 0$, there is a function $u : Q_{(1 - \alpha) R}(X_{R}) \to [-1,1]$ such that 
	\begin{gather*}
		( 1 - \alpha )^{ 1 - d } R^{1-d} \mathscr{F}_{\gamma}(u; Q_{(1 - \alpha) R}(X_{R})) \leq \bar{\sigma}(e_{1}) + \eta(\gamma + R^{-1}), \\
		u(y)= q((y - X_{R}) \cdot e_{1}) \, \, \text{for each} \, \, y \in \partial Q_{(1 - \alpha) R}(X_{R}).
	\end{gather*}
(We suppress the dependence on $R$, $\gamma$, and $\alpha$ for notational ease.)  This is nearly what we want, except that $u$ has the wrong boundary condition: the boundary condition has been translated slightly (but no more than a distance $\alpha R$, which is macroscopically small).  As we explain next, it is well-known that the error induced by this translation becomes negligible in the limit $\alpha \to 0$.

Arguing as in \cite[Proposition 9]{morfe}, we can define cut-off functions $\psi_{R,\alpha}$ such that if $v : Q_{R}(Rx_{0}) \to [-1,1]$ is the function 
	\begin{equation*}
		v(y) = (1 - \psi_{R,\alpha}(y)) u(y) + \psi_{R,\alpha}(y) q((y - Rx_{0}) \cdot e_{1}),
	\end{equation*} 
then $v(y) = q((y - R x_{0}) \cdot e_{1})$ for each $y \in \partial Q_{R}(Rx_{0})$ and
	\begin{align*}
		\limsup_{\alpha \to 0} \limsup_{(R^{-1},\gamma) \to 0} R^{1-d} \mathscr{F}_{\gamma}(v; Q_{R}(Rx_{0})) &\leq \limsup_{\alpha \to 0} \limsup_{(R^{-1},\gamma) \to 0} R^{1-d} \mathscr{F}_{\gamma}(u; Q_{(1 - \alpha) R}(X_{R})) \\
		&\leq \bar{\sigma}(e_{1}).
	\end{align*}
For more details, the reader is again referred to \cite[Proposition 9]{morfe}.  Of course, $v$ has the correct boundary condition, and its energy is asymptoptically no larger than $\bar{\sigma}(e_{1})$ plus an error that vanishes as $\alpha \to 0$, so we are done. 
\end{proof}

To conclude this appendix, it only remains to prove Proposition \ref{P: ergodic theorem continuous egorov}.

	\begin{proof}[Proof of Proposition \ref{P: ergodic theorem continuous egorov}] First, for any $\gamma > 0$, define $M^{*}_{\gamma}$ by 
		\begin{equation*}
			M^{*}_{\gamma} = \sup \left\{ |X_{\mu} - \bar{X}| \, \mid \, \mu \leq \gamma \right\}.
		\end{equation*}
	This is measurable since $\gamma \mapsto X_{\gamma}$ is continuous on an event of probability one.\footnote{Technically, since we have not specified that $(\Omega,\mathcal{F},\mathbb{P})$ is complete, we should define $M^{*}_{\gamma}$ more precisely by setting $M_{\gamma}^{*} = \sup \left\{ |X_{\mu} - \bar{X}| \, \mid \, \mu \in \mathbb{Q} \cap (0, \gamma] \right\}$, let $E \in \mathcal{F}$ be an event of probability one on which $\gamma \mapsto X_{\gamma}$ is continuous, and then restrict attention in the rest of the proof to the almost-sure event $\cap_{k \in \mathbb{Z}^{d}} \tau_{k}^{-1}(E)$.}  By assumption, $M^{*}_{\gamma} \to 0$ almost surely as $\gamma \to 0$.  Thus, by Egorov's Theorem applied to the sequence $\{M_{n^{-1}}^{*}\}_{n \in \mathbb{N}}$, for any $\delta > 0$, there is an event $E_{\delta} \in \mathcal{F}$ and a modulus of continuity $\omega_{\delta} : [0,\infty) \to [0,\infty)$ such that $\mathbb{P}(E_{\delta}) \geq 1 - \delta$ and 
		\begin{equation*}
			\sup_{E_{\delta}} M^{*}_{\gamma} \leq \omega_{\delta}(\gamma).
		\end{equation*}
	
	Observe that we can write
		\begin{align*}
			\left( L^{-1} R \right)^{1-d} \sum_{k \in (\{0\} \times L\mathbb{Z}^{d-1}) \cap Q_{R}} &|X_{\gamma} \circ \tau_{k} - \bar{X}| \leq \left( L^{-1} R \right)^{1-d} \sum_{k \in (\{0\} \times L \mathbb{Z}^{d-1}) \cap Q_{R}} M_{\gamma}^{*} \circ \tau_{k} \\
				&\quad \leq \omega_{\delta}(\gamma) \left( L^{-1} R \right)^{1-d}  \sum_{k \in (\{0\} \times L \mathbb{Z}^{d-1}) \cap Q_{R}}  \mathbf{1}_{E_{\delta}} \circ \tau_{k} \\
				&\quad \quad + 2C_{X} \left( L^{-1} R \right)^{1-d}  \sum_{k \in (\{0\} \times L\mathbb{Z}^{d-1}) \cap Q_{R}} \mathbf{1}_{\Omega \setminus E_{\delta}} \circ \tau_{k}.
		\end{align*}
	
	At this stage, we apply the ergodic theorem, although we apply it to the (not necessarily ergodic) group action $\{\tau_{k}\}_{k \in \{0\} \times L\mathbb{Z}^{d-1}}$.  Let $\mathscr{I}_{L}$ denote the $\sigma$-algebra of sets invariant under this action:
		\begin{equation*}
			\mathscr{I}_{L} = \{E \in \mathcal{F} \, \mid \, \tau_{k}^{-1}(E) = E \, \, \text{for each} \, \, k \in \{0\} \times L\mathbb{Z}^{d-1}\}.
		\end{equation*}
	By the ergodic theorem (see, e.g., \cite[Section 1.2]{armstrong_kuusi_book} or \cite[Appendix B]{morfe}),
		\begin{align*}
			\lim_{R \to \infty} (L^{-1}R)^{1-d} \sum_{k \in (\{0\} \times L\mathbb{Z}^{d-1}) \cap Q_{R}} \mathbf{1}_{\Omega \setminus E_{\delta}} \circ \tau_{k} = \mathbb{E}[\mathbf{1}_{\Omega \setminus E_{\delta}} \mid \mathscr{I}_{L}] \quad \text{almost surely.}
		\end{align*}
	Since $\omega_{\delta}$ is a modulus of continuity, we deduce that, for any $\delta > 0$, with probability one,
		\begin{align} \label{E: last hurrah}
			\limsup_{(R^{-1},\gamma) \to (0,0)} (L^{-1} R)^{1 - d} \sum_{k \in (\{0\} \times L \mathbb{Z}^{d-1}) \cap Q_{R}} |X_{\gamma} \circ \tau_{k} - \bar{X}| &\leq 2C_{X} \mathbb{E}[\mathbf{1}_{\Omega \setminus E_{\delta}} \mid \mathscr{I}_{L}].
		\end{align}
	At the same time, since $\mathbb{E}[ \mathbb{E}[\mathbf{1}_{\Omega \setminus E_{\delta}} \mid \mathscr{I}_{L}]] = \mathbb{P}(\Omega\setminus E_{\delta}) \leq \delta$, we have
		\begin{align*}
			\lim_{\delta \to 0} \mathbb{E}[\mathbf{1}_{\Omega \setminus E_{\delta}} \mid \mathscr{I}_{L}] = 0 \quad \text{in probability.}
		\end{align*}
	Thus, since the left-hand side of \eqref{E: last hurrah} is independent of $\delta$, we conclude
		\begin{align*}
			\limsup_{(R^{-1},\gamma) \to (0,0)} (L^{-1}R)^{1 - d} \sum_{k \in (\{0\} \times L\mathbb{Z}^{d-1}) \cap Q_{R}} |X_{\gamma} \circ \tau_{k} - \bar{X}| = 0 \quad \text{almost surely.}
		\end{align*}
	\end{proof}
	
Finally, we comment on the case when $a$ and $\theta$ are both either $\mathbb{Z}^{d}$-periodic or $B^{\infty}$ functions.

\begin{remark} \label{R: upper bound periodic almost periodic} The above proof generalizes without fanfare to the case when $a$ and $\theta$ are both either $\mathbb{Z}^{d}$-periodic or $B^{\infty}$ functions, as in Corollary \ref{cor:periodic-homogenization}.  In that case, it is not necessary to treat the $x_{0} = 0$ and $x_{0} \neq 0$ cases separately.  In the $x_{0} = 0$ stage of the proof, one should simply work with the cylinders $x_{0} + \mathcal{C}(R,h)$ and $x_{0} + \mathcal{C}_{L}(R,h)$ directly.  

In lieu of the ergodic theorem, one needs to show that, for any $L \in \mathbb{N}$, there is a uniform limit in this context:
	\begin{equation*}
		\lim_{\gamma \to 0} \sup \left\{ |m(\mathscr{F}_{\gamma}, x + \mathcal{C}(L,h), q) - m(\overline{\mathscr{F}}, \mathcal{C}(L,h), q)| \, \mid \, x \in \mathbb{R}^{d} \right\} = 0.
	\end{equation*}
This can be proved by mimicking the argument in the proof of Proposition \ref{P: uniform convergence variational quantities} above.  The details are left to the reader. \end{remark}



\section{$\Gamma$-Convergence in Probability}\label{sec:gamma-convergence-in-prob}

In this section, we complete the proof of Theorem \ref{T: homogenization theorem} by showing that our analysis of the planar cell problem \eqref{E: conv in prob part of the proof} implies $\Gamma$-convergence.  This will build on the approach in \cite{ansini_braides_chiado-piat,morfe,marziani}, the only novelty being we will consider convergence in probability in addition to almost-sure convergence.  Where possible, we follow \cite{ansini_braides_chiado-piat} and \cite{morfe}.

Before continuing, the reader may want to recall the abbreviation \eqref{eqn:defn-min} that we use to paraphrase the energy in \eqref{E: planar limit we want} from now on.

\subsection{Assumptions in this Appendix} Throughout this appendix, we assume that the medium $(a,\theta)$ satisfies the assumptions of Section \ref{S: assumptions}, but concerning the potential well $W$, we only assume that it is a nonnegative continuous function, that $\{u \in \mathbb{R} \, \mid \, W(u) = 0\} = \{-1,1\}$, and that the (super)quadratic growth assumption \eqref{eqn:w-growth-at-inf} holds.

\subsection{Subsequential Limits of the Cell Problems} To begin the section, we observe that, after possibly passing to a subsequence of scales $(\epsilon_{j})$, it is always possible to get almost-sure convergence in the cell problems studied in Section \ref{S: homogenization 1}, independently of the direction.  This is the main probabilistic ingredient used in the proof, which is tantamount to the well-known fact that convergence in probability is equivalent to subsequential almost-sure convergence. The precise assertion is made in the next proposition.

First, we need a bit of notation. Recall that $\{e_{1},\dots,e_{d}\}$ of $\mathbb{R}^{d}$ denotes the standard orthonormal basis in $\mathbb{R}^{d}$.

Rather than varying the basis as in \cite{morfe}, it is convenient to instead work with orthogonal transformations. Toward that end, let $O(d)$ denote the family of all orthogonal $d \times d$-matrices. Given $\mathcal{O} \in O(d)$, $R > 0$, and $x \in \mathbb{R}^{d}$, we define the (tilted) cube $Q^{\mathcal{O}}_{R}(x)$ via the following procedure: first, define an orthogonal coordinate system on $\mathbb{R}^{d}$ via the rule
	\begin{equation*}
		y = \sum_{i = 1}^{d} y_{i}^{\mathcal{O}} \mathcal{O}(e_{i}), \quad y_{i}^{\mathcal{O}} = y \cdot \mathcal{O}(e_{i}),
	\end{equation*}
then this is the cube determined by 
	\begin{equation*}
		Q^{\mathcal{O}}_{R}(x) = \left\{ y \in \mathbb{R}^{d} \, \mid \, |y_{i}^{\mathcal{O}} - x_{i}^{\mathcal{O}}| < \frac{R}{2} \right\}.
	\end{equation*}
Equivalently, if $Q_{R}(x)$ is defined as in \eqref{E: rotated cube}, then 
	\begin{equation*}
		Q^{\mathcal{O}}_{R}(x) = \mathcal{O} (Q_{R}(\mathcal{O}^{-1}x)).
	\end{equation*}
Furthermore, we will write
	\begin{align*}
		q^{ \mathcal{O } } (x) \coloneqq q ( \mathcal{O}(e_{1})  \cdot  x).
	\end{align*}

Here is the result concerning subsequential limits of the cell problems.

\begin{prop} \label{P: subsequential cell problem} Fix a positively one-homogeneous convex function $\sigma : \mathbb{R}^{d} \to [0,\infty)$. Suppose that the scale $\epsilon \mapsto \delta(\epsilon)$ has been chosen in such a way that, for any $(x,R,\mathcal{O}) \in \mathbb{R}^{d} \times (0,\infty) \times O(d)$ and any $\nu > 0$, 
	\begin{equation}\label{E: conv in prob thing}
		\lim_{\epsilon \to 0} \mathbb{P}\{ | m( \F_{\epsilon, \delta(\epsilon)}, Q^{\mathcal{O}}_{R}(x), q^{ \mathcal{O} } ( \epsilon^{-1}( \cdot - x ) ) ) - \sigma(\mathcal{O}(e_{1})) R^{d-1}| > \nu\} = 0.
	\end{equation}
Then, given any sequence $(\epsilon_{j})_{j \in \mathbb{N}} \subseteq (0,\infty)$ such that $\epsilon_{j} \to 0$ as $j \to \infty$, there is a subsequence $(j_{k})_{k \in \mathbb{N}} \subseteq \mathbb{N}$ such that $j_{k} \to \infty$ as $k \to \infty$ and, with probability one, for any $ (x,R,\mathcal{O}) \in \mathbb{R}^{d} \times (0,\infty) \times O(d)$,
		\begin{align} \label{E: almost sure conv upgrade}
			\lim_{k \to \infty } m( \F_{\epsilon_{j_k}, \delta(\epsilon_{j_k})}, Q^{\mathcal{O}}_{R}(x), q^{ \mathcal{O} } ( \epsilon_{j_k} ^{-1}( \cdot - x ) ) ) = \sigma(\mathcal{O}(e_{1})) R^{d-1}.
		\end{align}
\end{prop} 

Since the set $\mathbb{R}^{d} \times (0,\infty) \times O(d)$ appearing above is uncountable, it will be convenient to employ a countable, dense approximation.

\begin{proof}[Proof of Proposition \ref{P: subsequential cell problem}]
		We split the proof in two steps. First, we we establish that \eqref{E: almost sure conv upgrade} holds for all $(x,R,\mathcal{O}) \in \mathscr{E} $ in a countable dense subset $ \mathscr{E} \subset \mathbb{R}^{d} \times (0,\infty) \times O(d) $. Afterwards we use the continuity of $ m $  and $ \sigma $ to conclude almost-sure convergence.
		
		\textit{Step 1 (Almost-sure convergence on a dense subset).} Since $\mathscr{E}$ is countable, let us enumerate it in the form of a sequence: 
			\begin{equation*}
				\mathscr{E} = \{(x_{N},R_{N},\mathcal{O}_{N})\}_{N \in \mathbb{N}}.
			\end{equation*}
		By \eqref{E: conv in prob thing}, for any $k \in \mathbb{N}$, we can fix a $j_{k} \in \mathbb{N}$ such that $j_{k} \geq k$ and 
			\begin{align*}
				\sum_{n = 1}^{k} \mathbb{P} \left\{ \left| m( \F_{\epsilon_{j_k}, \delta(\epsilon_{j_k})}, Q^{\mathcal{O}_n}_{R_n}(x_n), q^{ \mathcal{O}_n } ( \epsilon_{j_k} ^{-1}( \cdot - x_n ) ) ) - \sigma(\mathcal{O}(e_{1})) R_{n}^{d-1} \right| > 2^{-k} \right\} \leq 2^{-k}.
			\end{align*}
		In particular, by the union bound,
			\begin{align*}
				\mathbb{P} \left( \bigcup_{n = 1}^{k} \left\{ \left| m( \F_{\epsilon_{j_k}, \delta(\epsilon_{j_k})}, Q^{\mathcal{O}_n}_{R_n}(x_n), q^{ \mathcal{O}_n } ( \epsilon_{j_k} ^{-1}( \cdot - x_n ) ) ) - \sigma(\mathcal{O}(e_{1})) R_{n}^{d-1} \right| > 2^{-k} \right\} \right) \leq 2^{-k}.
			\end{align*}
		Thus, the Borel-Cantelli Lemma implies there is a random variable $K \in \mathbb{N}$ such that $\mathbb{P}\{K < \infty\} = 1$ and, with probability one, for any $N \in \mathbb{N}$ and any $k \geq \max\{N,K\}$, we have 
			\begin{equation*}
				 \left| m( \F_{\epsilon_{j_k}, \delta(\epsilon_{j_k})}, Q^{\mathcal{O}_N}_{R_N}(x_N), q^{ \mathcal{O}_N } ( \epsilon_{j_k} ^{-1}( \cdot - x_N ) ) ) - \sigma(\mathcal{O}(e_{1})) R_{N}^{d-1} \right| \leq 2^{-k} .
			\end{equation*}
		Of course, since $\mathscr{E} = \{(x_{N},R_{N},\mathcal{O}_{N})\}_{N \in \mathbb{N}}$, this proves \eqref{E: almost sure conv upgrade} for $ (x,R,\mathcal{O}) \in \mathscr{E}  $.
		
		\textit{Step 2 (Upgrade to almost-sure convergence).}  By Step 1, there is an event of probability one on which, for any $(x,R,\mathcal{O}) \in \mathscr{E}$, the pointwise limit \eqref{E: almost sure conv upgrade} holds. It only remains to extend this to arbitrary parameters $(x,R,\mathcal{O}) \in \mathscr{E}$.
		
		Fix $(x,R,\mathcal{O}) \in \mathscr{E}$. Choose a sequence $\{(x_{n},R_{n},\mathcal{O}_{n})\}_{n \in \mathbb{N}} \subseteq \mathscr{E}_{\mathcal{D}}$ such that 
			\begin{itemize}
				\item[(i)] $R_{n} < R$ for each $n$, 
				\item[(ii)] $Q^{\mathcal{O}_{n}}_{R_{n}}(x_{n}) \subset \subset Q^{\mathcal{O}}_{R}(x)$ for each $n$, and
				\item[(iii)] $(x_{n},R_{n},\mathcal{O}_{n}) \to (x,R,\mathcal{O})$ as $n \to \infty$.
			\end{itemize}
		Arguing using the fundamental estimate as in \cite[Proof of Proposition 10]{morfe}, we deduce that 
			\begin{align*}
				\liminf_{n \to \infty} \liminf_{k \to \infty} \Big( & m( \F_{\epsilon_{j_k}, \delta(\epsilon_{j_k})}, Q^{\mathcal{O}_n}_{R_n}(x_n), q^{ \mathcal{O}_n } ( \epsilon_{j_k} ^{-1}( \cdot - x_n ) ) ) \\
				& \quad - m( \F_{\epsilon_{j_k}, \delta(\epsilon_{j_k})}, Q^{\mathcal{O}}_{R}(x), q^{ \mathcal{O} } ( \epsilon_{j_k} ^{-1}( \cdot - x ) ) ) \Big) \geq 0.
			\end{align*}
		Thus, by Step 1, 
			\begin{equation*}
				\limsup_{k \to \infty} m( \F_{\epsilon_{j_k}, \delta(\epsilon_{j_k})}, Q^{\mathcal{O}}_{R}(x), q^{ \mathcal{O} } ( \epsilon_{j_k} ^{-1}( \cdot - x ) ) )  \leq \sigma(\mathcal{O}(e_{1})) R^{d-1}.
			\end{equation*}
			
		To conclude, choose a sequence $\{(y_{n},S_{n},\mathcal{U}_{n})\}_{n \in \mathbb{N}} \subseteq \mathscr{E}_{\mathcal{D}}$ such that 
			\begin{itemize}
				\item[(i)] $S_{n} > R$ for each $n$, 
				\item[(ii)] $Q^{\mathcal{U}_{n}}_{S_{n}}(y_{n}) \supset \supset Q^{\mathcal{O}}_{R}(x)$ for each $n$, and
				\item[(ii)] $(y_{n},S_{n},\mathcal{U}_{n}) \to (x,R,\mathcal{O})$ as $n \to \infty$
			\end{itemize}
		Once again, we argue as in \cite{morfe}. First, fix $n$. For any $k$, choose $u_{k}$ such that 
			\begin{gather*}
				\mathscr{F}_{\epsilon_{j_{k}},\delta(\epsilon_{j_{k}})}(u_{k}; Q^{\mathcal{O}}_{R}(x)) = m( \F_{\epsilon_{j_k}, \delta(\epsilon_{j_k})}, Q^{\mathcal{O}}_{R}(x), q^{ \mathcal{O} } ( \epsilon_{j_k} ^{-1}( \cdot - x ) ) ) , \\
				u_{k} - q((\cdot - x) \cdot \mathcal{O}(e_{1})) \in H^{1}_{0}(Q^{\mathcal{O}}_{R}(x)).
			\end{gather*}
		It is convenient to extend $u_{k}$ so that $u_{k}(y) = q((y - x) \cdot \mathcal{O}(e_{1}))$ for each $y \in \mathbb{R}^{d} \setminus Q^{\mathcal{O}}_{R}(x)$. Applying the fundamental estimate in a manner similar to \cite[Proof of Proposition 10]{morfe}, we obtain a sequence $\{E_{n}\}_{n \in \mathbb{N}}$ such that
			\begin{gather*}
				\liminf_{k \to \infty} m( \F_{\epsilon_{j_k}, \delta(\epsilon_{j_k})}, Q^{\mathcal{U}_n}_{R_n}(x_n), q^{ \mathcal{U}_{n} } ( \epsilon_{j_k} ^{-1}( \cdot - x ) ) ) 
				\leq \liminf_{k \to \infty} \mathscr{F}_{\epsilon_{j_{k}},\delta(\epsilon_{j_{k}})}(u_{k}; Q^{\mathcal{O}}_{R}(x)) + E_{n},
			\end{gather*}
			and $ \lim\limits_{n \to \infty} E_{n} = 0 $.
		In particular, after sending $n \to \infty$, we deduce that, on the event $\hat{\Omega}$, we have
			\begin{align*}
				 \sigma(\mathcal{O}(e_{1})) R^{d - 1} &\leq \liminf_{k \to \infty} \mathscr{F}_{\epsilon_{j_{k}},\delta(\epsilon_{j_{k}})}(u_{k}; Q^{\mathcal{O}}_{R}(x)) \\
				&= \liminf_{k \to \infty} m( \F_{\epsilon_{j_k}, \delta(\epsilon_{j_k})}, Q^{\mathcal{O}}_{R}(x), q^{ \mathcal{O} } ( \epsilon_{j_k} ^{-1}( \cdot - x ) ) ).
			\end{align*}
		\end{proof}
		
\subsection{Subsequential $\Gamma$-Convergence} Next, we show that subsequential convergence of the energy in the planar cell problems implies subsequential $\Gamma$-convergence. This part of the argument is completely deterministic, hence we emphasize that the results are stated for a fixed realization of the medium $(a,\theta)$.
	
	\begin{prop} \label{P: general convergence} Let $\sigma : \mathbb{R}^{d} \to [0,\infty)$ be a positively one-homogeneous convex function and fix a realization of the medium $(a,\theta)$. Suppose that $(\epsilon_{j})_{j \in \mathbb{N}}, (\delta_{j})_{j \in \mathbb{N}} \subseteq (0,\infty)$ are sequences such that $\epsilon_{j} \to 0$ as $j \to \infty$ and, for any $(x,R,\mathcal{O}) \in \mathbb{R}^{d} \times (0,\infty) \times O(d) $
		\begin{equation} \label{E: cell problem convergence}
			\lim_{j \to \infty} m( \F_{\epsilon_{j}, \delta(\epsilon_{j})}, Q^{\mathcal{O}}_{R}(x), q^{ \mathcal{O} } ( \epsilon_{j} ^{-1}( \cdot - x ) ) ) = \sigma(\mathcal{O}(e_{1})) R^{d-1}.
		\end{equation}
	Then, for any bounded Lipschitz open set $U \subseteq \mathbb{R}^{d}$ and any $u \in L^{1}(U)$,
		\begin{equation*}
			\Gamma\text{-}\lim_{j \to \infty} \mathscr{F}_{\epsilon_{j},\delta_{j}}(u;U) = \mathscr{E}_{\sigma}(u; U), 
		\end{equation*}
	where $\mathscr{E}_{\sigma}$ is the functional
		\begin{equation} \label{E: limit functional}
			\mathscr{E}_{\sigma}(u;U) =  \left\{ \begin{array}{r l}
				\int_{U \cap \partial^{*} \{u = 1\}} \sigma(\nu_{\{u = 1\}}(\xi)) \, \mathcal{H}^{d-1}(d \xi), & \text{if} \, \, u \in BV(U; \{-1,1\}), \\
				+ \infty, & \text{otherwise.}
			\end{array} \right.
		\end{equation}
	\end{prop}
	
This proposition, which is implicit already in \cite{morfe}, improves the corresponding result in \cite{ansini_braides_chiado-piat}. In particular, in \cite[Theorem 3.7]{ansini_braides_chiado-piat}, it is necessary to first check a translation-invariance condition, whereas here we bypass this extra step.  

We expect that the proposition could also be proved following the arguments in \cite{marziani}.

The proof of Proposition \ref{P: general convergence} will use the following compactness result from \cite{ansini_braides_chiado-piat}. 

	\begin{theorem}[Theorem 3.3 and 3.5 in \cite{ansini_braides_chiado-piat}] \label{T: compactness result} Fix a realization of the medium $(a,\theta)$. Given any sequence $(\epsilon_{j})_{j \in \mathbb{N}}$ such that $\epsilon_{j} \to 0$ as $j \to \infty$, there is a subsequence $(j_{k})_{k \in \mathbb{N}} \subseteq \mathbb{N}$ and a bounded Borel function $\varphi : \mathbb{R}^{d} \times \mathbb{R}^{d} \to [0,\infty)$, both of which may depend on $(a,\theta)$, such that, for any $u \in L^{1}(U)$ and any bounded Lipschitz open set $U \subseteq \mathbb{R}^{d}$, 
		\begin{equation*}
		\Gamma\text{-}\lim_{k \to \infty} \mathscr{F}_{\epsilon_{j_{k}},\delta(\epsilon_{j_{k}})}(u; U) = \mathscr{E}_{\varphi}(u;U),
		\end{equation*} 
where $\mathscr{E}_{\varphi}$ is the functional given by 
		\begin{equation*}
		\mathscr{E}_{\varphi}(u;U) = \left\{ \begin{array}{r l}
        \int_{U \cap \partial^{*}\{u = 1\}} \varphi(\xi,\nu_{\{u = 1\}}(\xi)) \, \mathcal{H}^{d-1}(d \xi), & \text{if} \, \, u \in BV(U;\{-1,1\}), \\
        \infty, & \text{otherwise.}
        \end{array} \right.
		\end{equation*}
For any $(x,\mathcal{O}) \in \mathbb{R}^{d} \times O(d)$, the integrand $\varphi$ is determined by the formula
		\begin{align} \label{E: derivation formula}
		\varphi(x,\mathcal{O}(e_{1})) = \limsup_{\varrho \to 0} \inf \left\{ \frac{ \mathscr{E}_{\varphi}(u;Q^{\mathcal{O}}_{\varrho}(x)) }{ \varrho^{d-1} }\, \mid \, u = \chi_{\mathcal{O}(e_{1})}(\cdot - x) \, \, \text{in} \, \, Q^{\mathcal{O}}_{\varrho}(x)^c \right\} ,
		\end{align}
		where $ \chi_{ \nu } $ denotes the function equal to $ 1 $ on $ \{ x \cdot \nu \geq 0 \} $ and $ - 1 $ on $ \{ x \cdot \nu < 0 \} $.
	\end{theorem}

Finally, here is the proof of Proposition \ref{P: general convergence}. Note that the proof is simpler than that of \cite[Theorem 1]{morfe}; in particular, Proposition 12 in that reference is much more than is needed here.

	\begin{proof}[Proof of Proposition \ref{P: general convergence}] To establish that $\mathscr{F}_{\epsilon_{j},\delta(\epsilon_{j})} \overset{\Gamma}{\to} \mathscr{E}_{\sigma}$ as $j \to \infty$, it suffices to show that, for any subsequence $(j_{k})_{k \in \mathbb{N}}$, we have that $\mathscr{F}_{\epsilon_{j_{k}},\delta(\epsilon_{j_{k}})} \overset{\Gamma}{\to} \mathscr{E}_{\sigma}$. 
	
	Fix a subsequence $(j_{k})_{k \in \mathbb{N}}$. By Theorem \ref{T: compactness result}, there is a further subsequence along which $\Gamma$-convergence holds. Thus, up to relabelling, we can assume that there is a bounded Borel function $\varphi : \mathbb{R}^{d} \times \mathbb{R}^{d} \to [0,\infty)$ such that, for any $u \in L^{1}_{\text{loc}}(\mathbb{R}^{d})$ and any bounded Lipschitz open set $U \subseteq \mathbb{R}^{d}$, 
		\begin{equation*}
			\Gamma\text{-}\lim_{k \to \infty} \mathscr{F}_{\epsilon_{j_{k}},\delta(\epsilon_{j_{k}})}(u;U) = \mathscr{E}_{\varphi}(u;U)
		\end{equation*}
	We claim that \eqref{E: cell problem convergence} implies that $\varphi(x,e) = \sigma(e)$ for each $(x,e) \in \mathbb{R}^{d} \times S^{d-1}$, which is enough to complete the proof since then $\mathscr{E}_{\varphi} \equiv \mathscr{E}_{\sigma}$. We proceed by showing first that $\varphi(x,e) \leq \sigma(e)$ and then by showing that $\varphi(x,e) \geq \sigma(e)$. 
	
	Fix $x \in \mathbb{R}^{d}$ and $e \in S^{d-1}$. Let $\mathcal{O} \in O(d)$ be such that $e = \mathcal{O}(e_{1})$. By \eqref{E: derivation formula}, to establish that $\varphi(x,e) \leq \sigma(e)$, it suffices to show that, for any $\varrho > 0$
		\begin{equation} \label{E: lower bound part first ack}
			\inf \left\{ \mathscr{E}_{\varphi}(u;Q^{\mathcal{O}}_{\varrho}(x)) \, \mid \, u = \chi_{\mathcal{O}(e_{1})}(\cdot - x) \, \, \text{in} \, \, \mathbb{R}^{d} \setminus Q^{\mathcal{O}}_{\varrho}(x) \right\} \leq \sigma(\mathcal{O}(e_{1})) \varrho^{d-1}.
		\end{equation}
	
	Toward that end, for any $k \in \mathbb{N}$, fix a $u_{k} \in H^{1}_{\text{loc}}(\mathbb{R}^{d}; [-1,1])$ such that
		\begin{gather}
			\mathscr{F}_{\epsilon_{j_{k}},\delta(\epsilon_{j_{k}})}(u_{k}; Q^{\mathcal{O}}_{\varrho}(x)) = m( \F_{\epsilon_{j_k}, \delta(\epsilon_{j_k})}, Q^{\mathcal{O}}_{R}(x), q^{ \mathcal{O} } ( \epsilon_{j_k} ^{-1}( \cdot - x ) ) ), \label{E: minimizer thingie} \\ 
			u_{k} = q(\epsilon_{j_{k}}^{-1}(\cdot - x)) \quad \text{in} \, \, \mathbb{R}^{d} \setminus Q^{\mathcal{O}}_{\varrho}(x). \nonumber
		\end{gather}
	Observe that, for any bounded open set $U \subseteq \mathbb{R}^{d}$, there holds
		\begin{align*}
			\sup \left\{ \mathscr{F}_{\epsilon_{j_{k}},\delta(\epsilon_{j_{k}})}(u_{k}; U) \, \mid \, k \in \mathbb{N} \right\} < \infty.
		\end{align*}
	Thus, up to passing to yet another subsequence, we can assume that there is a $u \in BV_{\text{loc}}(\mathbb{R}^{d}; \{-1,1\}) $ such that 
		\begin{gather*}
			u = \lim_{k \to \infty} u_{k} \quad { \rm in } ~ L^1_{ \text{loc} } ( \R^d ) .
		\end{gather*}
By $\Gamma$-convergence, this implies
		\begin{gather*}
			\mathscr{E}_{\varphi}(u; Q^{\mathcal{O}}_{\varrho}(x)) \leq \liminf_{k \to \infty} \mathscr{F}_{\epsilon_{j_{k}},\delta(\epsilon_{j_{k}})}(u_{k}; Q^{\mathcal{O}}_{\varrho}(x)).
		\end{gather*}
	Due to the boundary conditions imposed on $u_{k}$ outside of $Q^{\mathcal{O}}_{\varrho}(x)$, we have that $u = \chi_{\mathcal{O}(e_{1})}(\cdot - x)$ in $\mathbb{R}^{d} \setminus Q^{\mathcal{O}}_{\varrho}(x)$. Thus, in view of \eqref{E: minimizer thingie} and \eqref{E: cell problem convergence}, the bound \eqref{E: lower bound part first ack} follows. 
	
	It only remains to establish the lower bound $\varphi(x,e) \geq \sigma(e)$. This half of the proof follows by arguing exactly as in \cite[Proof of Theorem 1]{morfe} (which, in particular, does not use Proposition 12 of that reference), hence the details are omitted. \end{proof}



\subsection{The Metric $d_{\Gamma}$}

Now that the main subsequential convergence results are proved, we turn to the task of providing a suitable definition for $\Gamma$-convergence in probability.  In what follows, we fix a $c_{*} > 0$ such that $W(u) \geq c_{*}^{2} |u|^{p}$ if $|u| \geq 2$, which is possible by \eqref{eqn:w-growth-at-inf}, and let $\Phi : \mathbb{R} \to \mathbb{R}$ be the function determined by 
	\begin{equation} \label{E: phi function}
		\Phi(0) = 0, \quad \Phi'(u) = \left\{ \begin{array}{r l}
								\sqrt{ W(u) } , & \text{if} \, \, |u| \leq 2 , \\
								c_{*} |u|^{ \frac{p}{2} } ,  & \text{if} \, \, |u| > 2 . \\
							\end{array} \right.
	\end{equation}
Notice that $\Phi$ is increasing and surjective and $\Phi'(u) \leq \sqrt{W(u)}$ for each $u \in \mathbb{R}$.  We argue that there exists a metric $ d_{\Gamma} $ on the space of functionals
\begin{align}\label{gc03}
\mathcal{M} = \left\{ \mathscr{F} : L^1(U) \rightarrow [0,\infty] ~  { \rm l.s.c. } ~\middle|~ \frac{1}{c} \mathscr{F}(u) \geq \int_{ U } | \nabla \Phi(u) | + \int_U | u |^p - 2^{p} | U | \right\},
\end{align}
so that convergence in $(\mathcal{M},d_{\Gamma})$ is equivalent to $ \Gamma $-convergence (induced by the strong $ L^1(U) $ topology). In \eqref{gc03}, $ c > 0 $ denotes some constant. Recall the following ``BV trick" (cf.\ \cite[Proof of Proposition 3]{modica}), which follows from the lower bound \eqref{eqn:w-growth-at-inf} on $W$ and Young's inequality:
\begin{align*}
	\mathscr{F}_{\epsilon,\delta}(u,U) &\gtrsim_{\lambda,\theta_{*},W} \int_{U} \sqrt{ W(u) } | \nabla u |  \geq \int_{U} |\nabla \Phi(u)| .
\end{align*}
In addition, for any $\epsilon \leq 1$, the assumption \eqref{eqn:w-growth-at-inf} implies $\mathscr{F}_{\epsilon,\delta}(u,U) \geq \int_{U} W(u) \gtrsim_{W} \int_{U} |u|^{p} - 2^{p} |U|$. Combining these two lower bounds, we deduce that
\begin{align}\label{eqn:gc02}
\mathscr{F}_{ \epsilon, \delta } (u, U)
\gtrsim_{ \lambda, \theta_*,W }  \int_{ U } | \nabla \Phi(u)| + \int_{ U } | u |^p - 2^{p} | U |
\end{align}
so that in fact $ \mathbb{P} \{ \mathscr{F}_{ \epsilon, \delta } \in \mathcal{M} \, \, \text{for each} \, \, \epsilon \leq 1 \, \, \text{and} \, \, \delta > 0 \} = 1 $ for an appropriate choice of $c$.

In the next lemma, we note that the functional in the lower bound \eqref{gc03} has compact sublevel sets.  Hence this functional is equicoercive in the sense of \cite[Definition 7.6]{DalMasoIntroGC}. 

\begin{lemma} For any $t \in \mathbb{R}$, the set $F(t) \subseteq L^{1}(U)$ given by 
\begin{align*}
F(t) = \left\{ u \in L^1(U) ~\middle|~ \int_{ U } | \nabla \Phi(u) | + \int_{ U } | u |^p \leq t + 2^{p} | U | \right\}
\end{align*}
is compact in $L^{1}(U)$.
\end{lemma}

	\begin{proof} We first observe that by the definition \eqref{E: phi function}, $\Phi$ is strictly increasing and surjective, hence the inverse $\Phi^{-1} : \mathbb{R} \to \mathbb{R}$ exists and is continuous.  Further, the definition of $\Phi$ implies that $|\Phi(u)| \lesssim |u|^{ \frac{p}{2} + 1}$, that is,
        \begin{align}\label{E: growth of phi guy}
            \limsup_{ |u| \to \infty }  |u|^{ - ( \frac{ p }{ 2 } + 1 ) }  | \Phi(u) | < \infty .
        \end{align}
        
	
We conclude by showing that if $ (u_k)_{ k \in \N } \subseteq F(t)$, then it has a subsequence that converges in $L^{1}(U)$ and pointwise almost everywhere.  From the definition of $F(t)$, we directly see that $ ( u_k )_{ k \in \N } $ is bounded in $ L^p(U) $, so that $ ( | u_k |^{ \frac{p}{2} + 1 } )_{ k \in \N } $ is bounded in $ L^1(U) $ by our assumption that $p \geq 2$.  In view of \eqref{E: growth of phi guy}, this implies that $( \Phi(u_{k}) )_{k \in \mathbb{N}}$ is bounded in $L^{1}(U)$.  Combined with $ \int_{ U } | \nabla \Phi(u_{k})| \lesssim 1 $ and the compactness of the embedding $ BV(U) \subseteq L^1(U) $, we learn from this that $ ( \Phi(u_{k}) )_{ k \in \N } $ is precompact in $ L^1(U) $.  Therefore, it has a subsequence $( \Phi(u_{k_{j}}) )_{j \in \mathbb{N}}$ that converges to some function $w \in L^{1}(U)$.  By the continuity of $\Phi^{-1}$, the boundedness of $(u_{k})_{k \in \mathbb{N}}$ in $L^{p}(U)$, and Vitali's Theorem, this implies that $(u_{k_{j}})_{j \in \mathbb{N}}$ converges to $\Phi^{-1}(w)$ in $L^{1}(U)$.  Finally, by lower semicontinuity, $\Phi^{-1}(w) \in F(t)$. \end{proof}

Due to the above equicoercivity,  we can invoke \cite[Theorem 10.22]{DalMasoIntroGC}: To this end, let us recall the Yosida-Moreau transform of some functional  $ \mathscr{F} \in \mathcal{M} $ is defined by
\begin{align*}
\mathscr{F}^{\alpha, \lambda}(u) \coloneqq \inf \left\{ \mathscr{F}(v) + \lambda \| u - v \|_{ L^1(U) }^{ \alpha } ~\middle|~ v \in L^1(U) \right\}
\end{align*}
for $ u \in L^1(U) $,~cf.~ \cite[Definition 9.8]{DalMasoIntroGC}. Using this, we define the metric
\begin{align}\label{eqn:gc-metric}
d_{ \Gamma }( \mathscr{F}, \mathscr{G} ) \coloneqq \sum_{ i, j \in \N } \frac{1}{ 2^{ i + j } } | \arctan( \mathscr{F}^{ 1, j } (u_i) ) - \arctan( \mathscr{G}^{ 1, j } (u_i) ) |
\end{align}
for a dense subset $ ( u_j )_{ j \in \N } \subset L^1(U) $.

\begin{prop}[Theorem 10.22 in \cite{DalMasoIntroGC}]\label{prop:gc-metric}
Definition \eqref{eqn:gc-metric} yields a metric $ d_{ \Gamma } $ on $ \mathcal{M} $ such that $ ( \mathcal{M}, d_{ \Gamma } ) $ is a compact metric space and for any sequence $ ( \mathscr{F}_k )_{ k \in \N } \subseteq \mathcal{M} $
\begin{align*}
\mathscr{F}_k \stackrel{ \Gamma }{ \rightarrow } \mathscr{F}
\quad \Longleftrightarrow \quad
d_{ \Gamma }( \mathscr{F}_k, \mathscr{F} ) \rightarrow 0.
\end{align*}
as $ k \rightarrow \infty $.
\end{prop}

We will rely on the fact that $\Gamma$-convergence holds if and only if it holds along any subsequence. Toward that end, of course, it helps to be precise about the definition of $\Gamma$-convergence from the very beginning. Since standard texts such as \cite{braides,DalMasoIntroGC} define $\Gamma$-convergence in terms of sequences, for the sake of completeness, let us give a precise definition of $\Gamma$-convergence of the one-parameter family $(\mathscr{F}_{\epsilon,\delta(\epsilon)})_{\epsilon > 0}$.

\begin{definition} \label{D: def of gamma conv}The functionals $\Gamma\text{-}\liminf \mathscr{F}_{\epsilon,\delta(\epsilon)}$ and $\Gamma\text{-}\limsup \mathscr{F}_{\epsilon,\delta(\epsilon)}$ are defined, for any bounded Lipschitz open set $U \subseteq \mathbb{R}^{d}$ and any $u \in L^{1}_{\text{loc}}(\mathbb{R}^{d})$
	\begin{align*}
		\Gamma\text{-}\liminf_{\epsilon \to 0} \mathscr{F}_{\epsilon,\delta(\epsilon)}(u;U) &= \lim_{r \to 0} \liminf_{\epsilon \to 0} \inf_{v \in L^{1}(U) \, : \, \|v - u\|_{L^{1}(U)} \leq r} \mathscr{F}_{\epsilon,\delta}(v;U), \\
		\Gamma\text{-}\limsup_{\epsilon \to 0} \mathscr{F}_{\epsilon,\delta(\epsilon)}(u;U) &= \lim_{r \to 0} \limsup_{\epsilon \to 0} \inf_{v \in L^{1}(U) \, : \, \|v - u\|_{L^{1}(U)} \leq r} \mathscr{F}_{\epsilon,\delta}(v;U).
	\end{align*}
In particular, we say that $(\mathscr{F}_{\epsilon,\delta(\epsilon)}(\cdot \, ; U))_{\epsilon > 0}$ $\Gamma$-converges to a functional $\mathscr{G}$ as $\epsilon \to 0$, written $\mathscr{F}_{\epsilon,\delta(\epsilon)}(\cdot \,; U) \overset{\Gamma}{\to} \mathcal{G}$, if the following equality holds:
	\begin{equation*}
		\mathscr{G} = \Gamma\text{-}\liminf_{\epsilon \to 0} \mathscr{F}_{\epsilon,\delta(\epsilon)}(\cdot \, ;U) = \Gamma\text{-}\limsup_{\epsilon \to 0} \mathscr{F}_{\epsilon,\delta(\epsilon)}(\cdot \, ;U) \quad \text{in} \, \, L^{1}_{\text{loc}}(\mathbb{R}^{d}).
	\end{equation*}\end{definition}

The next lemma asserts that the above definition is equivalent to subsequential $\Gamma$-convergence along an arbitrary subsequence.

\begin{lemma} Give any bounded Lipschitz open set $U \subseteq \mathbb{R}^{d}$ and any lower semi-continuous functional $\mathcal{G}$ on $L^{1}(U)$, we have that
	\begin{equation*}
		\mathscr{F}_{\epsilon,\delta(\epsilon)}(\cdot \,; U) \overset{\Gamma}{\to} \mathcal{G} \quad \text{as} \quad \epsilon \to 0
	\end{equation*}
if and only if, for any sequence $(\epsilon_{j})_{j \in \mathbb{N}} \subseteq (0,\infty)$ for which $\epsilon_{j} \to 0$ as $j \to \infty$, we have 
	\begin{equation*}
		\mathscr{F}_{\epsilon_{j},\delta(\epsilon_{j})}(\cdot \,; U) \overset{\Gamma}{\to} \mathcal{G}.
	\end{equation*}
\end{lemma}

Since the proof follows in a more-or-less routine way from Definition \ref{D: def of gamma conv}, the details are left to the reader.


\subsection{Convergence in Probability} We showed above that the $\Gamma$-convergence problem of interest to us can be understood via convergence in some abstract metric space $ ( \mathcal{M}, d_{ \Gamma } ) $. Here we recall the corresponding notion of convergence in probability for $ ( \mathcal{M}, d_{ \Gamma } ) $-valued random variables.

Denote by $\mathcal{B}$ the $\sigma$-algebra generated by the metric topology on $(\mathcal{M}, d_{\Gamma} )$. Furthermore, let $(\Omega,\mathcal{F},\mathbb{P})$ be a probability space. A map $ X : \Omega \to \mathcal{M} $ is called an $ ( \mathcal{M}, d_{ \Gamma } ) $-valued random variable if, for any $A \in \mathcal{B}$, we have that $X^{-1}(A) \in \mathcal{F}$. 

	\begin{definition} Given $ (\mathcal{M},d_{ \Gamma }) $-valued random variables $( \mathscr{F}_{\epsilon} )_{\epsilon > 0}$ and $ \mathscr{F} $ (on some common probability space), we say that $(\mathscr{F}_{\epsilon})_{\epsilon > 0}$ converges in probability to $ \mathscr{F} $ as $\epsilon \to 0$ provided that, for any $\nu > 0$,
		\begin{equation*}
			\lim_{\epsilon \to 0} \mathbb{P}\{d_{ \Gamma } (\mathscr{F}_{\epsilon}, \mathscr{F}) > \nu\} = 0.
		\end{equation*}
	\end{definition}

%
%
%
		
We will take advantage of the fact that, in metric spaces, convergence is determined by subsequences.  

\begin{prop}\label{P: almost sure is enough}
	Let $ ( \mathscr{F}_{ \epsilon } )_{ \epsilon > 0 } $ and $ \mathscr{F} $ be $ ( \mathcal{M}, d_{ \Gamma } ) $-valued random variables. The following are equivalent:
	\begin{enumerate}
		\item It holds $ d_{\Gamma}(\mathscr{F}_{ \epsilon },\mathscr{F}) \to 0 $ in probability as $ \epsilon \rightarrow 0 $.
	
		\item Every subsequence $ (\mathscr{F}_{k})_{ k \in \N } \subseteq ( \mathscr{F}_{ \epsilon } )_{ \epsilon > 0 } $ has a further subsequence $ (\mathscr{F}_{k_j})_{ j \in \N } $ such that $ d_{\Gamma}(\mathscr{F}_{ k_j },\mathscr{F}) \to 0$ almost surely as $ j \rightarrow \infty $.

		\item Every subsequence $ (\mathscr{F}_{k})_{ k \in \N } \subseteq ( \mathscr{F}_{ \epsilon } )_{ \epsilon > 0 } $ has a further subsequence $ (\mathscr{F}_{k_j})_{ j \in \N } $ such that $ \mathscr{F}_{ k_j } \stackrel{ \Gamma }{ \rightarrow } \mathscr{F} $ almost surely as $ j \rightarrow \infty $.
	\end{enumerate}
	In particular, if along every subsequence $ (\mathscr{F}_{k})_{ k \in \N } \subseteq ( \mathscr{F}_{ \epsilon } )_{ \epsilon > 0 } $ there is a further subsequence $(\mathscr{F}_{k_{j}})_{j \in \mathbb{N}}$ such that $ \mathscr{F}_{k_{j}} \stackrel{ \Gamma }{ \rightarrow } \mathscr{F} $ almost surely as $j \to \infty$, then $ ( \mathscr{F}_{ \epsilon } )_{ \epsilon > 0 } $ converges to $ \mathscr{F} $ in probability (w.r.t.~$ d_{ \Gamma } $) as $ \epsilon \rightarrow 0 $.
\end{prop} 

\begin{proof}
	Note that by the very definition, (1) is equivalent to the fact that every subsequence $ (\mathscr{F}_{k})_{ k \in \N } \subseteq ( \mathscr{F}_{ \epsilon } )_{ \epsilon > 0 } $ converges in probability (w.r.t.~$ d_{ \Gamma } $) to $ \mathscr{F} $. The equivalence to (2) and (3) follows from Lemma 5.2 in \cite{KallenbergFoMP} and Proposition \ref{prop:gc-metric}.
\end{proof}
		
\subsection{$\Gamma$-Convergence in Probability: Proof of Theorem \ref{T: planar homogenization theorem reduction}} Finally, to complete the proof of Theorem \ref{T: homogenization theorem}, we establish that if convergence in probability holds at the level of the planar cell problems, then $\Gamma$-convergence holds in probability.
	
		\begin{proof}[Proof of Theorem \ref{T: planar homogenization theorem reduction}.] The proof amounts to a concatenation of what has been proved so far in this section. The details are provided for the reader's convenience.  
		
		In view of Proposition \ref{P: almost sure is enough}, we only need to show that, given any bounded Lipschitz open set $U \subseteq \mathbb {R}^{d}$ and any sequence $(\epsilon_{k})_{k \in \mathbb{N}}$ such that $\epsilon_{k} \to 0$ as $k \to \infty$, there is a subsequence $(k_{j})_{j \in \mathbb{N}} \subseteq \mathbb{N}$ such that the limit $\mathscr{F}_{\epsilon_{k_{j}},\delta(\epsilon_{k_{j}})}(\cdot \, ; U ) \overset{\Gamma}{\to} \mathscr{E}_{\sigma}( \cdot \, ; U )$ holds with probability one as $j \to \infty$. 
		
		Toward that end, notice that the assumption \eqref{E: conv in prob part of the proof} implies that the hypotheses of Proposition \ref{P: subsequential cell problem} hold. Therefore, that proposition implies that there is a subsequence $(j_{k})_{k \in \mathbb{N}} \subseteq \mathbb{N}$ such that, with probability one, for any $(x,R,\mathcal{O}) \in \mathbb{R}^{d} \times (0,\infty) \times O(d)$, 
			\begin{align*}
				\lim_{k \to \infty} m( \F_{\epsilon_{j_k}, \delta(\epsilon_{j_k})}, Q^{\mathcal{O}}_{R}(x), q^{ \mathcal{O} } ( \epsilon_{j_k} ^{-1}( \cdot - x ) ) )= \sigma(e) R^{d-1}.
			\end{align*}
		This in turn implies, by Proposition \ref{P: general convergence}, that, with probability one, 			\begin{equation*}
				\Gamma\text{-}\lim_{k \to \infty} \mathscr{F}_{\epsilon_{j_{k}},\delta(\epsilon_{j_{k}})}(u;U) = \mathscr{E}_{\sigma}(u;U)
			\end{equation*}
		for any $u \in L^{1}(U)$. \end{proof}

\bibliographystyle{plain}
\bibliography{bibliography}

\end{document}